%% file: main.tex
\documentclass[letterpaper, oneside, reqno]{amsart}

\pdfoutput=1

\usepackage[margin=1in]{geometry}
\input{preamble}
\togglefalse{conf}

\usepackage{graphicx} 

\begin{document}

\author{Ewan Davies}
\address{Department of Computer Science, Colorado State University, USA}
\email{\href{mailto:ewan.davies@colostate.edu}{ewan.davies@colostate.edu}}

\author{Holden Lee}
\address{Department of Applied Mathematics and Statistics, Johns Hopkins University, USA}
\email{\href{mailto:hlee283@jhu.edu}{hlee283@jhu.edu}}

\author{Juspreet Singh Sandhu}
\address{Department of Computer Science, Colorado State University, USA}
\email{\href{mailto:jsinghsa@ucsc.edu}{js.sandhu@colostate.edu}}

\author{Jonathan Shi}
\address{Chipletics Inc, Redmond WA, USA}
\email{\href{mailto:jshi@cs.cornell.edu}{jshi@cs.cornell.edu}}

\title[Weak Poincar\'e Inequalities via Approximate Stochastic Localization]{\!\!\!\!\!\!\!Weak Poincar\'e Inequalities via Approximate Stochastic Localization:\!\!\!\!\!\! \\ Application to Sampling the Sherrington--Kirkpatrick Model}

\begin{abstract}
\small
\noindent We develop a new method for proving a weak functional inequality by first proving it for a sufficiently regular sequence of distributions approximating the stochastic localization (SL) process, and then transferring it to the desired distribution via regularity of the SL process and conductance arguments. We use this strategy to prove a weak Poincar\'e inequality (WPI) holds for the Gibbs measure of the Sherrington--Kirkpatrick model when $\beta < \rc2$. A prior result \ifconf{}{of the authors }\cite{DLSS26} proves the ASL process for the SK model satisfies the required regularity conditions. \\

\noindent A consequence of the WPI is that a much simpler algorithm---Glauber dynamics with a warm-start---efficiently samples the Gibbs measure of the SK model at $\beta < 1/2$. This is a significant structural step towards resolution of the conjecture that Glauber dynamics mixes fast in the replica-symmetric regime for the Sherrington--Kirkpatrick model~\cite[Open Problem 15]{bandeira2025randomstrasse101}. 

\end{abstract}

\vspace*{-0.5em}
\maketitle

\thispagestyle{empty}
\vspace{-5mm}

{
  \hypersetup{linkcolor=Red}
  \setcounter{tocdepth}{1}
  \tableofcontents
}

%
%
%
%
%
%

\newpage 
\pagenumbering{arabic}

\input{body}

\newpage

\addtocontents{toc}{\protect\setcounter{tocdepth}{-1}}

\ifconf{}{
\section*{Acknowledgements}

HL and JSS are grateful to the organizers of the Rocky Mountain Summer Workshop on Combinatorics, Probability and Algorithms (2025), where the initial stages of this work were conducted.
This work and the workshop were supported by NSF grant 2309707.}

\addtocontents{toc}{\protect\setcounter{tocdepth}{1}}


\addtocontents{toc}{\protect\setcounter{tocdepth}{-1}}
\renewcommand{\baselinestretch}{1.05}\normalsize
{
    \small
    \bibliographystyle{alpha_beta}
    \bibliography{main.bib}
}
\renewcommand{\baselinestretch}{1.0}\normalsize

\addtocontents{toc}{\protect\setcounter{tocdepth}{1}}

\newpage

\appendix
\normalsize
\input{appendix}

\end{document}

%% file: preamble.tex
\usepackage{tikz}
\usepackage{amsmath,amssymb,mathtools}
\usepackage[foot]{amsaddr}

\usepackage{placeins}

\usepackage{amssymb} 
\usepackage{stmaryrd}
\usepackage{graphicx}
\usepackage[colorlinks=true, pdfstartview=FitV, linkcolor=blue, citecolor=Green, urlcolor=WildStrawberry, linktoc=page]{hyperref} 

\usepackage{array}
\usepackage{ragged2e}

\usepackage{bm}

\usepackage{dsfont}
\usepackage[T1]{fontenc}
\usepackage[utf8]{inputenc}

\DeclareFontFamily{OT1}{rsfs}{}
\DeclareFontShape{OT1}{rsfs}{n}{it}{<-> rsfs10}{}
\DeclareMathAlphabet{\mathscr}{OT1}{rsfs}{n}{it}
\usepackage{mathrsfs}
\usepackage{MnSymbol}
\usepackage{enumitem}
\usepackage[dvipsnames]{xcolor}

\usepackage[normalem]{ulem}

\usepackage{booktabs} 

\usepackage{comment}

\usepackage{braket}

\usepackage{etoolbox}
\newtoggle{conf}

\newcommand{\ifconf}[2]{\iftoggle{conf}{#1}{#2}}

\definecolor{darkgreen}{rgb}{0,0.5,0}
\definecolor{darkblue}{rgb}{0,0,0.7}
\definecolor{darkred}{rgb}{0.9,0.1,0.1}

\newcommand{\blu}[1]{{\color{blue}{#1}}}

\newlength{\bibitemsep}
\let\oldthebibliography\thebibliography
\renewcommand\thebibliography[1]{%
  \oldthebibliography{#1}%
  \setlength{\parskip}{\bibitemsep}%
  \setlength{\itemsep}{-7pt}%
}

\newtheoremstyle{break}%
{}{}%
{\itshape}{}%
{\bfseries}{.\vphantom{$p_{p_{p_p}}$}}%
{\newline}
{\thmname{#1}\thmnumber{ #2}\thmnote{\ \,\textmd{(#3)}}}

\theoremstyle{break}

\newtheorem{proposition}{Proposition}
\newtheorem{theorem}[proposition]{Theorem}
\newtheorem*{theorem*}{Theorem}
\newtheorem{lemma}[proposition]{Lemma}
\newtheorem{corollary}[proposition]{Corollary}

\newtheoremstyle{remarkbreak}%
{6pt}{6pt}
{\normalfont}{}
{\itshape}{.}
{ }
{\thmname{#1}\thmnumber{ #2}\thmnote{\ \,\textnormal{(#3)}}}
\theoremstyle{remarkbreak}
\newtheorem{remark}[proposition]{Remark}
\newtheorem*{remark*}{Remark}

\newtheoremstyle{definitionbreak}%
{8pt}{8pt}%
{\normalfont}{}%
{\bfseries}{.}%
{\newline}%
{\thmname{#1}\thmnumber{ #2}\thmnote{\ \,\textmd{(#3)}}}
\theoremstyle{definitionbreak}
\newtheorem{definition}[proposition]{Definition}

\newtheorem{assumption}[proposition]{Assumption}

\newtheorem*{lemma*}{Lemma}

\newcommand{\vocab}[1]{\emph{#1}}

\numberwithin{equation}{section}
\numberwithin{proposition}{section}
\numberwithin{figure}{section}
\numberwithin{table}{section}

\newcommand{\Z}{\mathbb{Z}}
\newcommand{\N}{\mathbb{N}}

\newcommand{\R}{\mathbb{R}}

\newcommand{\calN}{\mathcal N}

\renewcommand{\P}{\mathop{{}\mathbb{P}}}

\newcommand{\Cov}{\mathop{{}\boldsymbol{\mathrm{Cov}}}}
\newcommand{\E}{\mathop{{}\mathbb{E}}}

\newcommand{\sF}{\mathscr{F}}
\newcommand{\Err}{\operatorname{Err}}

\renewcommand{\le}{\leqslant}
\renewcommand{\ge}{\geqslant}
\renewcommand{\leq}{\leqslant}

\renewcommand{\subset}{\subseteq}
\renewcommand{\bar}{\overline}

\renewcommand{\tilde}{\widetilde}
\newcommand{\td}{\widetilde}

\renewcommand{\hat}{\widehat}

\newcommand{\1}{\mathbf{1}}

\newcommand{\subeq}{\subseteq}

\newcommand{\al}{\alpha}
\newcommand{\be}{\beta}
\newcommand{\ga}{\gamma}
\newcommand{\de}{\delta}

\newcommand{\lm}{\lambda}
\newcommand{\La}{\Lambda}
\newcommand{\ph}{\varphi}

\newcommand{\ep}{\varepsilon}
\newcommand{\eps}{\varepsilon}
\newcommand{\si}{\sigma}

\newcommand{\om}{\omega}
\newcommand{\Om}{\Omega}

\newcommand{\Te}{\Theta}
\newcommand{\rh}{\varrho}
\newcommand{\rp}{c_{\mathrm{P}}}
\newcommand{\rls}{c_{\mathrm{LS}}}

\newcommand{\mg}[0]{m}
\newcommand{\ub}[2]{\underbrace{#1}_{#2}}

\newcommand{\pl}{\partial}
\newcommand{\dd}[2]{\frac{d #1}{d #2}}

\newcommand{\ddd}[1]{\frac{d}{d #1}}

\DeclareMathOperator{\sign}{sign}

\newcommand{\KL}{\operatorname{KL}}
\newcommand{\TV}{\operatorname{TV}}
\newcommand{\GOE}{\operatorname{GOE}}
\newcommand{\op}{\mathrm{op}}

\newenvironment{e*}{\begin{equation*}}{\end{equation*}\ignorespacesafterend}
\newcommand{\norm}[1]{\left\lVert{#1}\right\rVert}

\newcommand{\FT}{\calF_{\,\mathrm{TAP}}}

\newcommand{\xra}[1]{\xrightarrow{#1}}

\newcommand{\an}[1]{\left\langle#1\right\rangle}

\newcommand{\sT}{\mathsf{T}}

\newcommand{\Lip}{\mathsf{Lip}}
\newcommand{\ot}{\otimes}

\newcommand{\ab}[1]{\left| {#1} \right|}
\newcommand{\ba}[1]{\left[ {#1} \right]}
\newcommand{\bc}[1]{\left\{ {#1} \right\}}
\newcommand{\pa}[1]{\left( {#1} \right)}
\newcommand{\ve}[1]{\left\Vert {#1}\right\Vert}

\newcommand{\fc}[2]{\frac{#1}{#2}}
\newcommand{\rc}[1]{\frac{1}{#1}}
\newcommand{\sfc}[2]{\sqrt{\frac{#1}{#2}}}
\newcommand{\pf}[2]{\pa{\frac{#1}{#2}}}
\newcommand{\prc}[1]{\pa{\frac{1}{#1}}}

\newcommand{\sumo}[2]{\sum_{#1=1}^{#2}}

\newcommand{\gd}[0]{\nabla}
\newcommand{\bs}[0]{\setminus}

\newcommand{\Id}[0]{I}

\DeclareMathOperator{\im}{Im}

\DeclareMathOperator{\osc}{osc}

\renewcommand{\norm}[1]{\left\lVert{#1}\right\rVert}

\newcommand{\sop}{_\mathsf{op}}
\newcommand{\opnorm}[1]{\ensuremath{\left\lVert #1 \right\rVert\sop}}

\newcommand{\lipnorm}[1]{\ensuremath{\left\lVert {#1} \right\rVert_{\Lip}}}

\newcommand{\poly}{\mathsf{poly}}

\newcommand{\Var}{\mathop{{}\boldsymbol{\mathrm{Var}}}}
\newcommand{\Ent}{\mathop{{}\boldsymbol{\mathrm{Ent}}}}

\newcommand{\opl}{\oplus}

\renewcommand{\le}{\leqslant}
\renewcommand{\leq}{\leqslant}
\renewcommand{\ge}{\geqslant}

\usepackage{prettyref}
\newcommand{\savehyperref}[2]{\texorpdfstring{\hyperref[#1]{#2}}{#2}}

\protected\def\verythinspace{%
  \ifmmode
    \mskip0.5\thinmuskip
  \else
    \ifhmode
      \kern0.083em
    \fi
  \fi
}
\newrefformat{eq}{\savehyperref{#1}{\textup{(\ref*{#1})}}}
\newrefformat{e}{\savehyperref{#1}{\textup{(\ref*{#1})}}}
\newrefformat{ineq}{\savehyperref{#1}{\textup{(\ref*{#1})}}}
\newrefformat{eqn}{\savehyperref{#1}{\textup{(\ref*{#1})}}}
\newrefformat{l}{\savehyperref{#1}{Lemma~\ref*{#1}}}
\newrefformat{lem}{\savehyperref{#1}{Lemma~\ref*{#1}}}
\newrefformat{def}{\savehyperref{#1}{Definition~\ref*{#1}}}
\newrefformat{d}{\savehyperref{#1}{Definition~\ref*{#1}}}
\newrefformat{t}{\savehyperref{#1}{Theorem~\ref*{#1}}}
\newrefformat{thm}{\savehyperref{#1}{Theorem~\ref*{#1}}}
\newrefformat{cor}{\savehyperref{#1}{Corollary~\ref*{#1}}}
\newrefformat{c}{\savehyperref{#1}{Corollary~\ref*{#1}}}
\newrefformat{cha}{\savehyperref{#1}{Chapter~\ref*{#1}}}
\newrefformat{sec}{\savehyperref{#1}{\S\verythinspace\ref*{#1}}}
\newrefformat{s}{\savehyperref{#1}{\S\verythinspace\ref*{#1}}}
\newrefformat{subsec}{\savehyperref{#1}{\S\verythinspace\ref*{#1}}}
\newrefformat{app}{\savehyperref{#1}{\S\verythinspace\ref*{#1}}}
\newrefformat{tab}{\savehyperref{#1}{Table~\ref*{#1}}}
\newrefformat{fig}{\savehyperref{#1}{Figure~\ref*{#1}}}
\newrefformat{hyp}{\savehyperref{#1}{Hypothesis~\ref*{#1}}}
\newrefformat{alg}{\savehyperref{#1}{Algorithm~\ref*{#1}}}
\newrefformat{a}{\savehyperref{#1}{Assumption~\ref*{#1}}}
\newrefformat{rem}{\savehyperref{#1}{Remark~\ref*{#1}}}
\newrefformat{item}{\savehyperref{#1}{Item~\ref*{#1}}}
\newrefformat{step}{\savehyperref{#1}{step~\ref*{#1}}}
\newrefformat{conj}{\savehyperref{#1}{Conjecture~\ref*{#1}}}
\newrefformat{fact}{\savehyperref{#1}{Fact~\ref*{#1}}}
\newrefformat{p}{\savehyperref{#1}{Proposition~\ref*{#1}}}
\newrefformat{prop}{\savehyperref{#1}{Proposition~\ref*{#1}}}
\newrefformat{prob}{\savehyperref{#1}{Problem~\ref*{#1}}}
\newrefformat{claim}{\savehyperref{#1}{Claim~\ref*{#1}}}
\newrefformat{clm}{\savehyperref{#1}{Claim~\ref*{#1}}}
\newrefformat{relax}{\savehyperref{#1}{Relaxation~\ref*{#1}}}
\newrefformat{rem}{\savehyperref{#1}{Remark~\ref*{#1}}}
\newrefformat{red}{\savehyperref{#1}{Reduction~\ref*{#1}}}
\newrefformat{part}{\savehyperref{#1}{Part~\ref*{#1}}}
\newrefformat{ex}{\savehyperref{#1}{Exercise~\ref*{#1}}}
\newrefformat{property}{\savehyperref{#1}{Property~\ref*{#1}}}
\newrefformat{type}{\savehyperref{#1}{Type~\ref*{#1}}}
\newrefformat{eg}{\savehyperref{#1}{Example~\ref*{#1}}}
\newrefformat{obs}{\savehyperref{#1}{Observation~\ref*{#1}}}
\newrefformat{que}{\savehyperref{#1}{Question~\ref*{#1}}}
\newrefformat{cond}{\savehyperref{#1}{Condition~\ref*{#1}}}
\newrefformat{ass}{\savehyperref{#1}{Assumption~\ref*{#1}}}
\newrefformat{not}{\savehyperref{#1}{Notation~\ref*{#1}}}
\newrefformat{cond}{\savehyperref{#1}{Condition~\ref*{#1}}}
\newrefformat{f}{\savehyperref{#1}{Figure~\ref*{#1}}}

\newlist{thmenum}{enumerate}{1}
\setlist[thmenum]{
    label=\textup{(\roman*)},
    ref=\theproposition\textup{(\roman*)}
}

\newrefformat{pitem}{\savehyperref{#1}{Proposition~\ref*{#1}}}
\newrefformat{titem}{\savehyperref{#1}{Theorem~\ref*{#1}}}
\newrefformat{litem}{\savehyperref{#1}{Lemma~\ref*{#1}}}
\newrefformat{ditem}{\savehyperref{#1}{Definition~\ref*{#1}}}
\newcommand{\Sref}[1]{\hyperref[#1]{\S\ref*{#1}}}

\let\pref=\prettyref
\let\Cref=\prettyref

\renewcommand{\eps}{\varepsilon}

\newcommand{\calA}{\mathcal A}
\newcommand{\calB}{\mathcal B}

\newcommand{\calD}{\mathcal D}
\newcommand{\calE}{\mathcal E}
\newcommand{\calF}{\mathcal F}

\newcommand{\calH}{\mathcal H}

\renewcommand{\calN}{\mathcal N}

\newcommand{\calR}{\mathcal R}
\newcommand{\calS}{\mathcal S}

\newcommand{\calW}{\mathcal W}

\newcommand{\Law}{\mathcal L}

\newcommand{\sL}{\mathscr L}

\newcommand{\sE}{\mathscr E}

\newcommand{\bbS}{\mathbb S}

\newcommand{\one}[0]{\mathds{1}}
\renewcommand{\set}[2]{\left\{{#1}:{#2}\right\}}

\renewcommand{\R}{\mathbb R}

\renewcommand{\N}{\mathbb N}
\renewcommand{\Z}{\mathbb Z}

\newcommand{\lr}[0]{\leftrightarrow}

\newcommand{\iy}{\infty}

\newcommand{\ball}[2]{B_{#1}(#2)}

\newcommand{\ppart}[1]{%
  \par\addvspace{\medskipamount}%
  \noindent\textit{#1.}\hspace{0.5em}\ignorespaces%
}

\renewcommand{\lr}[0]{\leftrightarrow}

\newcommand{\scr}[1]{\mathscr{#1}}

\newcommand{\Kprox}[1]{K^{\textup{prox}}_{#1}}
\newcommand{\prox}[0]{^{\textup{prox}}}

\newcommand{\loc}[0]{^{\textup{local}}}
\newcommand{\CM}[0]{_{\textup{CM}}}

\renewcommand{\Kprox}[1]{K^{\textup{prox}}_{#1}}
\renewcommand{\prox}[0]{^{\textup{prox}}}

\renewcommand{\loc}[0]{^{\textup{local}}}

\usepackage{algorithm}
\usepackage{algpseudocode}
\algnewcommand\algorithmicinput{\textbf{Input: }}
\algnewcommand\INPUT{\State\algorithmicinput}
\algnewcommand\algorithmicinitialize{\textbf{Initialize: }}
\algnewcommand\INIT{\State\algorithmicinitialize}
\algnewcommand\algorithmicrun{\textbf{Run: }}
\algnewcommand\RUN{\State\algorithmicrun}
\algnewcommand\algorithmicupdate{\textbf{Update: }}
\algnewcommand\UPDATE{\State\algorithmicupdate}
\algnewcommand\algorithmicset{\textbf{Set: }}
\algnewcommand\SET{\State\algorithmicset}
\algnewcommand\algorithmicquery{\textbf{Query: }}
\algnewcommand\QUERY{\State\algorithmicquery}
\algnewcommand\algorithmicoutput{\textbf{Output: }}
\algnewcommand\OUTPUT{\State\algorithmicoutput}

\makeatletter
\newcommand\appendix@section[1]{%
  \refstepcounter{section}%
  \orig@section*{\@Alph\c@section.\texorpdfstring{\,\,\,\;}{}#1}
}
\let\orig@section\section
\g@addto@macro\appendix{\let\section\appendix@section}
\makeatother

\renewcommand{\paragraph}[1]{\medskip\noindent{\bf #1{.}}}

\makeatletter
\newcommand{\saveequation}[2]{
  #2 \label{#1}
  \protected@write\@mainaux{}{\string\SAVEEQUATION{#1}{\unexpanded{\unexpanded{#2}}}}%
}
\newcommand{\savetagequation}[3]{
  #3 \label{#1} \tag{#2}
  \protected@write\@mainaux{}{\string\SAVEEQUATION{#1}{\unexpanded{\unexpanded{#3}}}}%
}
\newcommand{\SAVEEQUATION}[2]{%
  \global\@namedef{SAVEDEQUATION@#1}{#2}%
}
\newcommand{\repeatequation}[1]{%
  \ifcsname SAVEDEQUATION@#1\endcsname
    \@nameuse{SAVEDEQUATION@#1}\tag{\ref{#1}}%
  \else
    ?? \notag
  \fi
}
\makeatother

\setlist[itemize]{topsep=-4pt, partopsep=2pt}

\usepackage{setspace}

\usepackage{thmtools}
\declaretheoremstyle[%
  spaceabove=-2pt,%
  spacebelow=6pt,%
  headfont=\normalfont\itshape,%
  postheadspace=1em,%
  qed=\qedsymbol%
]{mystyle} 
\declaretheorem[name={Proof},style=mystyle,unnumbered,
]{prf}



%% file: body.tex
\section{Introduction}

Let $\mu$ be a probability measure on $\Om\subeq r\cdot \mathbb{S}^{n-1}\subeq \R^n$, for instance, $\Om = \{\pm 1\}^n$ or $\sqrt n\cdot \mathbb{S}^{n-1}$. Consider a Markov chain or process on $\Om$ with $\mu$ as its stationary distribution, for instance, Glauber dynamics or Langevin on the sphere, respectively. 
A standard problem is to prove functional inequalities for $\mu$, which would imply a number of useful properties such as concentration of Lipschitz functions and fast mixing of the associated dynamics. 
Simulating these dynamics would then give an efficient sampling algorithm.

Stochastic localization (SL) \cite{eldan2013thin} is a measure-valued stochastic process which almost surely converges to a point mass $\de_y$, where $y$ is a draw from the target measure $\mu$. In the framework of sampling algorithms, this has been used in two ways: first, as a framework for proving functional inequalities \cite{chen2022localization}, and second, as a sampling algorithm itself through a computational approximation; this is the framework of algorithmic stochastic localization (ASL) \cite{EMS25}. While the former gives stronger results---namely, efficient sampling from running the natural Markov chain or process---it also requires stronger conditions. When $\mu$ is a Gibbs distribution for a spin glass model, these conditions can be very challenging to show. 
The appeal of ASL has been that it gives an efficient sampling algorithm under weaker conditions.

This motivates the following question: 
\begin{center}
    \emph{Can we show functional inequalities under weaker conditions than required in the SL framework,\\ namely those that suffice for ASL?}
\end{center} 
We show an affirmative answer to this question for weak Poincar\'e inequalities (WPI) under natural conditions where an ASL process can track the SL process up to $O(1)$ error. In particular, this will apply for us to prove a weak Poincar\'e inequality for the Gibbs measure of the the Sherrington-Kirkpatrick (SK) model up to $\be<\rc2$, defined by 
\[
\mu_{\be A}(\si) \propto 
e^{\frac{\beta}{2}\an{\si,  A\si}}, \quad \si\in \{\pm 1\}^n, \quad A\sim \GOE(n).
\]
For the SK model, polynomial-time sampling with negligible total variation (TV) distance is known up to $\beta < 1/2$ \cite{DLSS26}, but functional inequalities were previously known only up to $\be \approx 0.295$ \cite{anari2024trickle}.

\subsection{Main result}
To state our results in generality, we consider the linear-tilt stochastic localization (SL) process associated with $\mu$ and an ``algorithmic'' stochastic localization (ASL) process that approximates it. Note that unlike prior works, we will only rely on the \emph{existence} of an ASL process with good properties, and not need to use it as an algorithm. Hence our ASL can be better thought of as ``approximate'' stochastic localization. We define the SL process and an ASL process by the following SDEs:
\begin{align}
    \savetagequation{e:SL}{SL}{dy_t &= \mg(y_t) dt +  dB_t&y_0&=0}\\
    \savetagequation{e:ASL}{ASL}{d\hat y_t &= \hat \mg(\hat y_t) dt+dB_t}&\hat y_0&=0,
\end{align}
where the magnetization $m$ is 
\[
\mg(y) := \an{\si}_{\mu_y} = \E_{\mu_y}\si , \quad \text{where} \quad \mu_y(\si) \propto \mu(\si) e^{\an{y,\si}}
\]
and $\hat m$ is an approximation of $m$ that is uniformly bounded. 
Let $p_t$, $\hat p_t$ denote the distribution of $y_t$, $\hat y_t$ respectively at time $t$. The $\mu_{y_t}$ are the random measures given by SL, and almost surely $\mu_{y_t}\rightarrow\de_\si$ for some $\si$ as $t\to \iy$; this $\si$ will be distributed as $\mu$.

We make the following assumptions, which are known to be satisfied with probability $\ge 1-\de$ for the SK model for $\be<\rc 2$ \cite{DLSS26}, and used to show efficient sampling with an ASL-based algorithm.
\begin{assumption}[Magnetization error]
\label{a:m-error}
For $t\in [0,T_1]$, for some uniformly bounded $\ep(t)$, 
\[
\E \| \mg (y_t) - \hat \mg (y_t) \|^2 \le \ep(t)^2.
\]
\end{assumption}
\begin{assumption}[Lipschitz drift for ASL]
\label{a:Lipschitz-drift}
For all $y,z\in \R^n$,
\[
\|\hat \mg (y) - \hat \mg (z)\| \le L
\|y-z\|.
\]
\end{assumption}
Notably, this is an assumption on the approximate, not the ideal, SL process. Not only can this be easier to prove, but also it can hold in cases where the true magnetization is not Lipschitz.
\begin{assumption}[Functional inequality for localized distribution]
\label{a:localized-fi}
    For $T=T_0<T_1$, with probability $1-\de'(n)$ over the draw of $y_T\sim p_{T}$, $\mu_{y_T}$ satisfies a $(c_T\loc,\de(n))$-weak Poincar\'e inequality (see \pref{d:wpi}) for Glauber dynamics (when $\mu$ is a distribution on $\{\pm 1\}^n$) or Langevin dynamics (when $\mu$ has a density on $r\cdot \bbS^n$). 
\end{assumption}
For the SK model with $\be<\rc 2$, this holds with $\de'(n) = \de(n)= e^{-c(\be)n}$ 
for large enough $n$ and appropriate constant $c(\be)$; see \pref{t:local-wpi}.

Our main theorem is the following.
\begin{theorem}[Weak Poincar\'e inequality from Approximate Stochastic Localization]
\label{t:main}
Let $\mu$ be a measure on $\{\pm 1\}^n$ or $r\cdot \bbS^{n-1}$.
    Under \pref{a:m-error}, \pref{a:Lipschitz-drift}, and \pref{a:localized-fi}, for any $\ep>0$, $\mu$ satisfies a $(c_{T_0}\loc e^{-O(1/\ep)},e^{O(1/\ep)}\cdot $ $(\de(n)+\de'(n)) + \ep)$-weak Poincar\'e inequality, where the constants depend on the constants in the assumptions. 
\end{theorem}
\begin{remark*}
    Examining the proof shows that the dependence of the $e^{O(1/\ep)}$ term on the bounds in \pref{a:m-error} and \pref{a:Lipschitz-drift} is $e^{O((E_{T_1}+1) e^{O(L)}/\ep)}$, where $E_T$ is defined in \eqref{e:Et}.
\end{remark*}
This result holds under significantly weaker assumptions compared to what is required in the SL framework, and hence provides a plausible strategy to proving functional inequalities for spin glass models, where due to significant challenges even in showing static properties of the distribution, few tight results are known. 
As our main application, we combine this with the results of \cite{DLSS26} to show the following.
\begin{corollary}[Weak Poincar\'e inequality for SK model for $\be<\rc 2$]
\label{c:sk}
For every $\beta<1/2$ and $\delta>0$, there exist constants
$C=C(\beta,\delta)$ and $N=N(\beta,\delta)$ such that for all
$n\ge N$, with probability at least $1-\delta$ over $A \sim \GOE(n)$, for every
$\epsilon\in(0,1)$, $\mu_{\beta A}$ satisfies a
\(\left(\frac1n e^{-C/\epsilon},\epsilon\right)\)-weak Poincar\'e inequality.
\end{corollary}

A Poincar\'e inequality shows fast mixing from any initial distribution, while a weak Poincar\'e inequality shows fast mixing from a warm start. Therefore, turning this into an algorithmic result requires adding a first phase of the algorithm to obtain that warm start (\pref{d:ws}). 
Compared to \cite{DLSS26}, which provided a complex algorithm via ASL, Jarzynski's equality with rejection sampling, and the polarized walk, we use the weak Poincar\'e inequality to show that the following substantially simpler algorithm suffices for sampling, using Glauber dynamics (\pref{alg:Langevin}) after obtaining a warm start with Langevin dynamics. Note that the quantitative strength of the weak Poincar\'e inequality is not enough for simulated annealing to work \cite{huang2025weak}, so we proceed by establishing that a natural 
annealed distribution 
$\rh_T$ is a warm start to the SL distribution, a result of independent interest. Establishing this only relies on one additional assumption, that $\rh_T$ is related to the ASL distribution through Jarzynski's equality whose weights satisfy Lipschitz drift (\pref{ass:Lipschitz-je}).

\begin{algorithm}[ht]
\caption{SK Sampler: Glauber dynamics from a warm start}
\label{alg:Langevin}
\begin{algorithmic}[1]
\INPUT Matrix $A\sim \GOE(n)$, inverse temperature $\be<\rc 2$, SL time $T=T(\be)$
\OUTPUT Approximate sample from Gibbs measure $\mu_{\be A}$.
\State \textbf{Part 1 (Obtaining a warm start):}
To obtain $y_T$,
\begin{itemize}
    \item \blu{run Langevin dynamics for $\rh_T$ for time $T_{\textup{Langevin}} = \td O(1)$ with step size $\eta = \td\Te\prc{n}$ 
    starting from $y_0\sim \calN(0,TI_n)$,} 
    \[
y_{t+\eta} = y_{t} + \eta \pa{\hat m(y_t) - \fc{y_t}{T}} + \sqrt{2\eta}\, \xi_t, \quad \xi_t\sim \calN(0,I_n),
    \]
    where $\hat m(y)$ is the solution\footnotemark{} to $
    -\be A m - y + \tanh^{-1}(\mg)
    + \be^2 \pa{1-\rc n \ve{m}_2^2}m=0$, 
    \emph{or}
    \item run ASL for time $T$ with step size $\eta = \td \Te\prc n$.
\end{itemize}
\State Starting from $\sign(y_T)$, for the localized distribution $\mu_{\be A, y_T}$,
\begin{itemize}
    \item \blu{run Glauber dynamics for $O(n^2)$ steps,} \emph{or} 
    \item run the polarized walk for $O(n)$ steps.
\end{itemize}
\State \textbf{Part 2 (Sampling from warm start):} Run Glauber dynamics for $ne^{(1/\ep)^{O(1)}}$ steps.
\end{algorithmic}
\end{algorithm}
\footnotetext{The solution can be computed to arbitrary accuracy using mirror descent; see \cite[\S D]{DLSS26}.}

\begin{theorem}[Sampling from SK using warm start and Glauber dynamics]
\label{t:alg}
    For the SK model for $\be<1/2$, for any $\de>0$, with probability $\ge 1-\de$, for any $\ep>0$, \pref{alg:Langevin} (with Langevin and Glauber dynamics) samples from a distribution that is $\ep$ close to $\mu_{\be A}$ in time $O(n^2 + ne^{(1/\ep)^{O(1)}})$, where the constants depend on $\de, \be$.
\end{theorem}
Note that the warm start portion of \pref{alg:Langevin} is independent of $\ep$, hence obtaining a more accurate sample only requires running Glauber dynamics for a longer time in part 2. 
We note that the dependence on $\ep$ is worse than in \cite{DLSS26}, which has runtime $O(\poly(n) e^{O(1/\ep)})$. Although we do not analyze it here, we expect that the method in \cite{DLSS26} (ASL and polarized walk) can also give a warm start.

\subsection{Discussion} Prior work on spherical spin glass models used uniform bounds on $D_ym(y) =\Cov(\mu_y)$ (which amounts to \pref{a:Lipschitz-drift} on the true magnetization) to prove functional inequalities via SL, with this being weakened to \emph{high-probability} bounds to prove a \emph{weak} Poincar\'e inequality \cite{huang2025weak}. This is far from what is currently known for the SK model for $0.295<\be<0.5$, namely second moment bounds on $\Cov(\mu_y)$ \cite{DLSS26}; such bounds were ``integrated'' via Gr\"onwall's inequality to show \pref{a:m-error}.

The parameters of our weak Poincar\'e inequality are substantially weaker than what is known for the spherical $p$-spin model at high temperature via SL \cite{huang2025weak}. A $(n^{-O(1)}, o(n^{-1}))$-WPI would allow using simulated annealing to obtain negligible guarantees in TV distance (using $O(n)$ temperatures); however, taking $\ep=O\prc{n}$ in \pref{t:main}  already gives a trivial weak Poincar\'e constant of $\exp(-O(n))$. 

Our weak Poincar\'e inequality allows showing mixing from a warm start in $L^\iy$ divergence, up to $\ep$ in time $n\exp(O(1/\ep))$. We obtain a higher complexity for the end-to-end algorithm, as we only approximately obtain a warm start in $L^\iy$ divergence. We find it surprising that a weak Poincar\'e inequality can be proved under such minimal assumptions, and expect that stronger assumptions---left to future work---will be necessary to strengthen it.


We note that while \pref{alg:Langevin} for the SK model simplifies the former algorithm of \cite{DLSS26}, the proof of its correctness relies on the same assumptions. Hence, 
for the SK model, our result shortcuts the algorithm but not the proof.
The bulk of \cite{DLSS26} is occupied by showing these assumptions (implied by the desiderata therein)\footnote{For the SK model, the proofs for \pref{a:m-error} and \pref{a:Lipschitz-drift} rely on two results in free probability and mathematical spin-glass theory established in \cite{DLSS26}---Lipschitz control over the diagonal sub-algebra of squared resolvents of deformed Wigner matrices with optimal error rates, and precise cavity estimates (stronger than overlap concentration) for the covariance matrix of the planted SK model with SL tilt.}. We find it remarkable that although the tools introduced there---ASL and Jarzynski's equality---are algorithmic in nature and intended as such, they can (1) function to prove a purely mathematical result, and (2) be absorbed into the analysis of an algorithm that does not rely on simulating these processes.

\subsection{Related work}
In the context of sampling, stochastic localization (originally developed in \cite{eldan2013thin}) has been used in two ways: first as a way of proving functional inequalities \cite{chen2022localization} and hence establishing rapid mixing of the standard Markov chain (e.g., Glauber dynamics), and second, algorithmically through approximating a SDE \cite{el2022sampling}. For spin glasses, this has been applied both to the SK model on the hypercube and spherical $p$-spin models.

Functional inequalities and efficient sampling algorithms for the SK model for $\be<1$ has been a longstanding open problem \cite{sompolinsky1981dynamic,mezard1988spin,bandeira2025randomstrasse101}. 
Using localization schemes, a modified log-Sobolev inequality has been proved up to $\be\approx 0.295$ \cite{eldan2022spectral,anari2022entropic,anari2024trickle}. Algorithmic stochastic localization has been used to give efficient sampling for the SK model for the entire regime $\be<1$, but with error measured under the weaker Wasserstein distance metric \cite{el2022sampling,celentano2024sudakov}. A recent result \ifconf{}{by the authors }\cite{DLSS26} combines ASL and Jarzynski's equality with rejection sampling to give efficient sampling with negligible TV distance for $\be<\rc2$. The analysis in  \cite{DLSS26} builds on the potential Hessian ascent (PHA) framework \ifconf{}{developed by Jekel and the last two authors }\cite{jekel2024pha, jekel2025pha2}. 

For spherical $p$-spin models, 
\cite{huang2024sampling} use ASL to sample up to the so-called SL threshold with negligible TV distance error. \cite{huang2025weak} use SL to show weak Poincar\'e inequalities up to the same threshold, thereby showing efficient sampling using annealed Langevin dynamics. However, in addition to \pref{a:localized-fi}, they require exponential tails on the operator norm of the covariance under the localization process.
Due to the difficulty of working on the discrete hypercube, it is not known how to obtain operator norm bounds of the right order for the SK model.
In contrast, we only need \pref{a:m-error} and \pref{a:Lipschitz-drift}, whose proofs in \cite{DLSS26} only suffice to imply second-tracial-moment bounds on the covariance.

Our main question is similar to a question posited by \cite{huang2025weak}: Is there a general reduction from a sampling guarantee for algorithmic stochastic localization to one for simulated annealing? We show that an ASL guarantee essentially implies a WPI---just not of the strength required for vanilla simulated annealing.

Given the difficulty of proving rapid mixing, various works have consider weaker notions of mixing.
\cite{liu2024locally} show that locally stationary distributions with respect to Glauber dynamics are also locally stationary with respect to restricted Gaussian dynamics (RGD). \cite{rajaraman2025markov} relate the behavior or RGD to approximate message passing (AMP). We note that their use of localization schemes to relate the behavior of Glauber dynamics and RGD bears resemblance to our decomposition relating Glauber dynamics with the proximal sampler; their RGD can be thought of as a rank-1 version of the proximal sampler.

We rely on a decomposition theorem for the weak Poincar\'e inequality, following
a line of work on decomposition theorems, including for the spectral gap or Poincar\'e inequality \cite{madras2002markov,andrieu2018uniform,ge2018simulated},  restricted spectral gap \cite{garg2025restricted}, and $s$-conductance \cite{zhou2025polynomial}. \cite{gao2026weak} give related comparison results for the weak Poincar\'e inequality. Decomposition theorems have been used to analyze sampling from multimodal distributions \cite{ge2018simulated}, as well as data augmentation and Gibbs samplers including hybrid (Metropolis-within-Gibbs) variants \cite{qin2025spectral}. They enable local-to-global results for Markov chains \cite{oppenheim2018local} and appear in another guise in localization schemes \cite{chen2022localization}. 
\cite{qin2025spectral} gives a unifying framework for many of these decomposition theorems. 

For distributions on $\R^n$, the proximal sampler was introduced in \cite{liang2022proximal}, with improved analyses given in \cite{chen2022improved,fan2023improved}. 
The (A)SL process is a reparameterization of a standard denoising diffusion model~\cite{Mon23}, and these models have been independently discovered by the machine learning community for use in generative modeling. A long line of work has established efficient sampling under access to an approximate score function~\cite{chen2022sampling,chen2023improved,benton2023nearly}. Jarzynski's equality \cite{Jar97} is an identity in statistical mechanics based on the Feynman-Kac formula \cite{kac1949distributions}, which has recently seen wide application in sampling and in machine learning \cite{vargas2023transport,albergo2024nets}, with theoretical guarantees given in~\cite{guo2025complexity}.

\subsection{Technical overview and organization of paper} 
We first give preliminaries in \pref{s:prelim}. 

In \pref{s:decomp}, we present the keystone of the proof, a decomposition theorem for the weak Poincar\'e inequality (WPI) which reduces the task of proving a WPI for Glauber to proving a WPI (1) for the time-$T_0$ localized distribution ($X\mid Y$) and (2) for the time-$T_0$ proximal sampler ($Y$). The proximal sampler is the sampler that iteratively alternates between sampling $X\mid Y$ and then the posterior $Y\mid X$, where $X$ is a Gaussian noised observation of $Y$ given by the SL process. 
The result follows as a corollary (\pref{c:wpi-gd-ps}) of a more general decomposition theorem for a joint distribution on $(X,Y)$ (\pref{t:wpi-decomp}), where a WPI for conditional distributions $X\mid Y$ and for the alternating Gibbs chain on $Y$ implies a WPI for the joint chain; projecting onto $X$ then gives the WPI for the marginal distribution. 

In \pref{s:wpi-SL}, we prove a WPI for the time-$T_0$ proximal sampler (\pref{l:prox-wpi}) by first showing it for Langevin for the SL distribution $p_{T_0}$ via an isoperimetric argument, namely $s$-conductance (\pref{t:wpi-pT}). The idea of the proof is given in \pref{f:1}. First, choose $T_1>T_0$; regularity of ASL gives a log-Sobolev inequality for $\hat p_{T_1}$. Drift error gives bounded KL divergence to $p_{T_1}$, which allows $p_{T_1}$ to have smaller density, but not too much larger density. Unfortunately, regions of smaller density can ``disconnect'' $p_{T_1}$ and inhibit proving a functional inequality for $p_{T_1}$. Instead, we note that by construction of the SL process, $p_{T_0}$ for $T_0<T_1$ is a Gaussian noising of $p_{T_1}$, and this smoothed-out distribution will satisfy a WPI. Formally, we argue that given any two not-too-small sets under $p_{T_1}$, we can couple a large portion of their mass under $p_{T_1}$ such that the distance between coupled points is small (\pref{l:couple-distance}). The conductance of a set under $p_{T_0}$ then cannot be too small, by considering its conductance under mixture-of-gaussians with bounded separation under these coupled points. A conductance profile gives a WPI by Cheeger's inequality; we then apply the argument of \cite{chen2022improved} to show that mixing for Langevin implies mixing for the proximal sampler. We conclude in \pref{s:proof} with a proof of the main theorem (\pref{t:main}) by combining \pref{s:wpi-SL} with the decomposition result for WPI (\pref{c:wpi-gd-ps}).

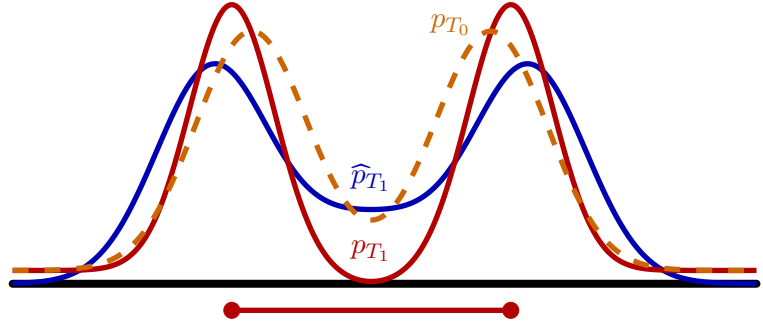
\begin{figure}[h!]
\centering
\begin{tikzpicture}[x=0.5pt,y=0.5pt]

\def\leftX{80}
\def\rightX{640}
\def\baseY{40}    
\def\W{720}
\def\H{250}

\draw[line width=3pt, black, line cap=round] (\leftX,\baseY) -- (\rightX,\baseY);

\draw[line width=2pt, blue!70!black]
  plot[smooth, samples=200, domain=\leftX:\rightX]
    ({\x},
     {%
       \baseY
       + 160*exp(-((\x-230)/60)^2)   
       + 160*exp(-((\x-470)/60)^2)   
       + 50*exp(-((\x-350)/80)^2)    
     } );

\draw[line width=2pt, red!70!black]
  plot[smooth, samples=200, domain=\leftX:\rightX]
    ({\x},
     {%
       \baseY
       + 200*exp(-((\x-245)/45)^2)   
       + 200*exp(-((\x-455)/45)^2)   
       - 0*exp(-((\x-350)/45)^2)    
       + 10*(1-exp(-((\x-350)/45)^2))                            
     } );

\draw[line width=2pt, orange!80!black, dashed, dash pattern=on 6pt off 6pt]
  plot[smooth, samples=200, domain=\leftX:\rightX]
    ({\x},
     {%
       \baseY
       + 180*exp(-((\x-260)/60)^2)   
       + 180*exp(-((\x-440)/60)^2)   
       - 0*exp(-((\x-350)/90)^2)    
       + 10                            
     } );

\node[blue!70!black,font=\large] at (350,120) {$\hat{p}_{T_1}$};
\node[red!70!black,font=\large] at (350,65) {$p_{T_1}$};
\node[orange!80!black,font=\large] at (410,235) {$p_{T_0}$};

\draw[line width=2pt, red!70!black, line cap=round]
  (245,20) -- (455,20);
\filldraw[red!70!black] (245,20) circle (3pt);
\filldraw[red!70!black] (455,20) circle (3pt);

\useasboundingbox (0,0) rectangle (\W,\H);

\end{tikzpicture}
\caption{
Proving a WPI for the SL distribution $p_{T_0}$ (orange, dashed) by ``scaffolding'' with the ASL process $\hat p_{T_1}$ (blue, solid).
Small drift error between $\hat p_{T_1}$ and $p_{T_1}$ (red, solid) means that that they are close together, but the density of $p_{T_1}$ can still go to zero where $\hat p_{T_1}$ is non-zero.
Gaussian noise and the small-distance coupling (red line segment) lift the corresponding areas of $p_{T_0}$ away from zero.
}
\label{f:1}
\end{figure}

To apply this to the SK model, we need to prove a WPI for the time-$T_0$ localized distribution. In \pref{s:wpi-local}, we show that the method in \cite{DLSS26}, written for a different Markov chain called the polarized walk, can be adapted for Glauber dynamics. We show the WPI holds whenever the mass is highly concentrated in a neighborhood, which holds with high probability for the localized distribution. This allows us to prove our main theorem for the SK model (\pref{c:sk}).

The remaining sections turn the weak Poincar\'e inequality for the SK model into an algorithmic result by establishing a warm start that allows efficient sampling, which may be of independent interest.

To motivate \pref{s:rho-t-lsi}, note that the ASL distribution satisfies a LSI and approximates the SL distribution; however its density is not explicit. \cite{DLSS26} define explicit distributions $\rh_t$ in terms of the TAP free energy which is a reweighting of the ASL distribution (via Jarzynski's equality), which also well-approximate the SL distribution. We will refer to $\rh_t$ as annealed distributions. 
We show that under the general \pref{a:Lipschitz-je} (in the special case of a constant initial distribution $y_0 = 0$)---that $\rh_t$ is a reweighting with weights defined by a ODE with Lipschitz drift---which holds for the SK model, $\rh_t$ also satisfies a log-Sobolev inequality. To show this, we first prove a perturbation result for finite dimensional Gaussian space (\pref{lem:lipschitz-lsi-measure-space}), and then extend it to the infinite-dimensional path space of Brownian motion via a limiting argument that uses the decomposition of the reweighed measure into a finite-dimensional cylindrical part and an infinite-dimensional Gaussian tail (\pref{t:lsi-CM-pert}). Then, using the fact that the end-point map to evaluate $\rh_T$ is Lipschitz, we ``transport'' this path-space LSI and show it holds for $\rh_T$ in standard Eucliden space (\pref{thm:lsi-rhT-lipschitz}). Hence, we can sample $\rh_t$ directly, and this is a suitable candidate for a warm start.

To obtain a warm start to the SK measure, we first obtain a warm start for the ASL distribution $\hat p_{T_0}$ and then approximately sample from the localized measure. In \pref{s:ws}, we show that $\rh_{T_0}$ is a warm start for $\hat p_{T_0}$ through a similar scaffolding approach as in \pref{s:wpi-SL}, thinking of these as approximate Gaussian noisings of $\rh_{T_1}$ and $\hat p_{T_1}$ for some $T_1>T_0$. 

Finally, in \pref{s:alg}, we put everything together to prove our main algorithmic result (\pref{t:alg}).

\section{Preliminaries}
\label{s:prelim}

\subsection{Markov processes, functional inequalities, and isoperimetry}

We assume knowledge of Markov chains and processes (see e.g., \cite{levin2026markov} and \cite{kallenberg1997foundations}). 

\begin{definition}
    Given a Markov chain on $\Om$ defined by a Markov kernel $P$ with stationary distribution $\mu$, define the Dirichlet form by $\sE(f,g) = \an{f,(I - P)g}$. 
    Given a Markov process on $\Om$ defined by a generator $\sL$ with stationary distribution $\mu$, define the Dirichlet form by $\sE(f,g) = -\an{f,\sL g}$.
\end{definition}
Note that any Markov chain can be associated with a Markov process with the same Dirichlet form by letting $\sL = P-I$; this corresponds to making transitions according to the discrete chain according to a Poisson clock with rate 1. We will always assume that the Markov chain/process is ergodic and reversible. 

We use Glauber dynamics as the canonical Markov chain on $\{\pm 1\}^n$ and Langevin dynamics as the canonical Markov process on $\R^n$ with a specified stationary measure $\mu$. 
Define Glauber dynamics with stationary distribution $\mu$ on $\{\pm 1\}^n$ as the Markov chain where at each step, if the current sample is $x$, we choose a coordinate $i\in [n]$ uniformly at random, and resample $x_i$ according to 
\[\mu(X_i|X_{[n]\bs \{i\}}=x_{[n]\bs \{i\}}).\]
The Dirichlet form of Glauber dynamics is 
\[\sE_\mu(f,g) = \rc{2n} \sum_{x\in \{\pm 1\}^n} \sumo in \fc{\mu(x)\mu(x^{\opl i})}{\mu(x) + \mu(x^{\opl i})}(f(x^{\opl i})-f(x))(g(x^{\opl i})-g(x)),\]
where $x^{\opl i}$ denotes $x$ with the $i$th coordinate flipped.
Langevin diffusion with stationary distribution $\mu$ is defined by the SDE 
\[dX_t = \gd \ln \mu(X_t)\,dt + \sqrt 2\, dB_t,\] with Dirichlet form $\sE_\mu(f,g) = \int_{\R^n} \an{\gd f,\gd g}\,d\mu$.

We now provide some background on 
functional inequalities.

\begin{definition}[Log-Sobolev inequality]
Consider a Markov chain or process on $\Om$ with stationary distribution $\pi$ and Dirichlet form $\mathscr E$. 
    For a function $f:\Om\to\R$ and measure $\mu$ on $\Om$ we write 
    \[\Ent_\mu(f) := \E_\mu f\log f - \E_\mu f \E_\mu \log f.\]
    We say that the Markov chain or process satisfies a \vocab{modified log-Sobolev inequality (MLSI)}\footnote{On $\R^n$, there is no distinction between the LSI and the modified LSI, and we will drop the modifier.} with constant $c$ if for all $f:\Om \to \R_{\ge 0}$ with $\Ent_\mu(f)>0$,  
    \[
\Ent_\mu(f) \le \rc{2c}\sE(f,\log f).
    \]
\end{definition}
For Langevin dynamics on $\R^n$, this has the equivalent form $\Ent_\mu(f^2)\le \fc{2}{c}\sE(f,f)$.

\begin{definition}[Weak Poincar\'e inequality]
\label{d:wpi}
Consider a Markov chain or process on $\Om$ with stationary distribution $\pi$ and Dirichlet form $\mathscr E$. 
    We say that the Markov chain or process satisfies a \vocab{$(c,\de)$-weak Poincar\'e inequality (WPI)} with error function $\mathrm{Err}$ if for all $f$,
    \[
\Var_\pi(f) \le \rc{c}\sE(f,f) + \de \Err(f).
    \]
    A Poincar\'e inequality (PI) with constant $c$ is a $(c,0)$-WPI. 
\end{definition}
For background on weak Poincar\'e inequalities, we refer to \cite[\S7.5]{bakry2013analysis} for Markov processes and \cite{andrieu2026weak} for Markov chains. 

We say a distribution $\mu$ satisfies a Poincaré or modified log-Sobolev inequality with constant $c$ if the canonical Markov chain or process (Glauber dynamics for $\{\pm 1\}^n$ or Langevin dynamics for $\R^n$ or $r\bbS^n$) satisfies a Poincaré or modified log-Sobolev inequality with constant $c$. We call the maximal $c$ for which this holds \emph{the} Poincaré or modified log-Sobolev constant, and denote it by $\rp(\mu)$ or $\rls(\mu)$. A MLSI implies a PI with the same constant. Here smaller is worse; we note that some texts use the reciprocal convention ($C=\rc c$ as the Poincaré or modified log-Sobolev constant). 

If not specified, we take $\Err(f) = \osc(f)^2$ where $\osc(f):=\sup f - \inf f$. Weak Poincar\'e inequalities allow mixing from a warm start.
\begin{lemma}[{\cite[Theorem 4.6]{huang2025weak}}]
\label{l:wpi-converge}
    Suppose that for a Markov process, $\pi$ satisfies a $(c,\de)$-weak Poincar\'e inequality with error function $\mathrm{Err}$. 
    Let $\La_T(s) = \fc{2c}{e^{2c T}-1} e^{2c s}$. 
    Then for all $\nu_0$,
    \[
\chi^2(\nu_T\|\pi) \le e^{-2c T} \chi^2(\nu_0\|\pi) + \de \E_{s\sim \La_T} \ba{\Err\pa{\dd{\nu_s}{\pi}}}
\le e^{-2c T} \chi^2(\nu_0\|\pi) + \de 
\sup_{s\in [0,T]}\ba{ \Err\pa{\dd{\nu_s}{\pi}}}
    \]
\end{lemma}
\begin{remark*}
    For a Markov chain, the same bound (on the right side) holds for $T\in \N$ with $e^{2c T}$ replaced by $(1-c)^{-2T}$.
\end{remark*}
For $\Err = \osc^2$, monotonicity gives that the sup can be bounded by $\Err\pa{\dd{\nu_0}{\pi}}$.

A useful consequence of functional inequalities is concentration; we will just need the result for the log-Sobolev inequality.
\begin{lemma}[{Lipschitz functions are sub-gaussian under log-Sobolev inequality, \cite[\S 5.4.2]{bakry2013analysis}}]
\label{l:conc-lsi}
    Suppose $\mu$ on $\R^n$ satisfies a log-Sobolev inequality with constant $c$. Let $f$ be a $L$-Lipschitz function. 
    Then 
    \[\mathbb{E}_{\mu}\left[e^{t(f-\mathbb{E}_\mu f)}\right] \leq e^\frac{L^2t^2}{2c}\]
    and
    \[
\mu\pa{f\ge \E_\mu f + r} \le e^{-\fc{c r^2}{2L^2}}, \qquad 
\mu\pa{f\le \E_\mu f - r} \le e^{-\fc{c r^2}{2L^2}}.
    \]
\end{lemma}
Functional inequalities are related to isoperimetry. We will need the relationship between $s$-conductance and the weak Poincar\'e inequality for distributions on $\R^n$.
The $s$-conductance is defined to be the minimum ratio of the boundary measure to the measure of a set, over sets with measure in $(s,\rc 2]$.
\begin{definition}
    Let $\mu$ be a probability measure on $\R^n$. 
    Given a measurable set $S$, define the $\ep$-neighborhood of $S$ to be
    \[
S_\ep := \set{x\in \R^n}{d(x,S)\le \ep}, \qquad 
\pl_\ep S = S_\ep\bs S.
    \]
    Define the (outer) boundary measure by 
    \[\mu^+(S):= 
    \liminf_{h\to 0^+} \fc{\mu(\pl_h S)}{h}.
    \]
    Define the \vocab{$s$-conductance} of $\mu$ to be 
    \begin{align*}
    \Phi_s(\mu)
    := \inf_{S:\mu(S)\in (s,\rc 2]} \fc{\mu^+(S)}{\mu(S)}. 
    \end{align*}
    Define the \vocab{conductance} of $\mu$ to be $\Phi(\mu) := \Phi_0(\mu)$. 
\end{definition}

\begin{theorem}[Cheeger's inequality for weak Poincar\'e inequality, {consequence of \cite[Theorem 4.1]{rockner2001weak}}] 
\label{t:cheeger}
Let $\mu$ be a probability measure on $\R^n$ and $s \in [0,1/2)$ with $\Phi_{s/2}(\mu)>0$.
Then $\mu$ satisfies a $(\rc 4\Phi_{s/2}^{2}(\mu),s)$-weak Poincar\'e inequality.
\end{theorem}
Note that \cite{rockner2001weak} actually show it suffices to take the infimum over smooth domains $S$.
The conductance profile is exactly known for Gaussians.
\begin{theorem}[Gaussian isoperimetric inequality, \cite{sudakov1978extremal,borell1975brunn}]
\label{t:g-isop}
    Let $\ga_{\si^2}$ be the the measure of the Gaussian $\calN(0,\si^2I)$. Then for any $A$, 
    \[
\ga_{\si^2}^+(A)
\ge 
\rc{\si}\ph (\Phi^{-1}(\ga_{\si^2}(A))),
    \]
    where $\ph(t) = \rc{\sqrt{2\pi}} e^{-\fc{t^2}2}$ and $\Phi(t) = \int_{-\iy}^t \ph(s)ds$ are the pdf and cdf of the standard Gaussian.
\end{theorem}

\subsection{Stochastic localization and the proximal sampler}

We focus on the linear tilt localization \cite{el2022sampling,chen2022localization} for a probability measure $\mu$ on\footnote{It is not necessary to restrict to $r\cdot \bbS^{n-1}$, but convenient in our presentation. In the general case, $\mu_t(\si) \propto \mu(\si) e^{-\fc{\ve{\si}^2}{2t} + \an{y_t,\si}}$; the quadratic term is constant when $\mu$ is supported on $r\cdot \bbS^{n-1}$.}  $\Om\subeq r\cdot \bbS^{n-1}$, which is a measure-valued stochastic process $(\mu_t)_{t\ge 0}$ defined by
\begin{align*}
\repeatequation{e:SL}\\
\mu_t(\si) &= \mu_{y_t}(\si) \propto \mu(\si) e^{\an{y_t,\si}}
\end{align*}
where $m(y) = \an{\si}_{\mu_y} = \E_{\mu_y}\si$. 
For any set $A\subeq \Om$, $\mu_t(A)$ is a martingale. The measure localizes in the sense that almost surely, $\mu_{y_t}\rightarrow \de_\si$ for some $\si$ as $t\to \iy$, and this $\si$ is distributed as $\mu$. Thus, this is a (non-algorithmic) process that samples from $\mu$. 
Stochastic localization is equivalent to the diffusion process
\begin{align*}
    dx_t &= dB_t,& x_0&\sim\mu,
\end{align*}
under the substitution $y_t = tx_{1/t}$ \cite{Mon23}; note that the time is reversed.
We hence have for $t_0<t_1$ that 
\begin{align}
\label{e:SL-time-change}
y_{t_0} = t_0x_{1/t_0} 
\stackrel d= t_0\pa{x_{1/t_1} + \sqrt{\rc{t_0}-\rc{t_1}}\,\xi} = \fc{t_0}{t_1} y_{t_1} + \sqrt{t_0\pa{1-\fc{t_0}{t_1}}}\,\xi, \quad \xi\sim\calN(0,I_n).
\end{align}
This motivates the proximal sampler. Note that in contrast to the literature, for convenience we define it in the SL rather than diffusion model coordinates.
\begin{definition}\label{d:ps}
    Given $t_0<t_1$, we define the following Markov kernels. 
\begin{thmenum}
    \item\label{ditem:ps-forward} $K_{t_0\to t_1}(x_0,\cdot)$: Given $x_0$, run \eqref{e:SL} starting from $y_{t_0}=x_0$, and let $x_1 = y_{t_1}$ be the output.
    \item\label{ditem:ps-backward} $K_{t_1\to t_0}(x_1,\cdot)$: Given $x_1$, draw $\xi\sim \calN(0,\Id_n)$ and let $x_0 = \fc{t_0}{t_1}x_1 +\sqrt{t_0\pa{1-\fc{t_0}{t_1}}}\xi$ be the output.
\end{thmenum}
We define the $(t_0,t_1)$-proximal sampler as the Markov chain with kernel $\Kprox{t_0,t_1}:=K_{t_0\to t_1}K_{t_1\to t_0}$.
\end{definition}
Note that \pref{ditem:ps-backward} is the reverse of the SL process by \eqref{e:SL-time-change}, in that the following define the same joint distribution:
\begin{align*}
    x_0&\sim p_{t_0}, & x_1 &\sim K_{t_0\to t_1}(x_0,\cdot),\\
    x_1&\sim p_{t_1}, & x_0 &\sim K_{t_1\to t_0}(x_1,\cdot).
\end{align*}
Hence, these kernels are adjoint with respect to the measures $ p_{t_0}$ and $ p_{t_1}$. $K_{t_0\to t_1}(x_0,\cdot)$ can be thought of as the posterior distribution of $x_1$ where $x_0$ is a noisy Gaussian observation.

Noting that SL converges in the sense that $\fc{y_t}{t}$ converges to a point in $\{\pm 1\}^n$, we define the above for $t_1=\iy$ by
\begin{enumerate}
    \item $K_{t_0\to \iy}(x_0,\cdot)$: Given $x_0$, run \eqref{e:SL} starting from $y_{t_0}=x_0$, and let $\si = \lim_{t\to \iy} \fc{y_t}{t}$ be the output.
    \item $K_{\iy\to t_0}(\si,\cdot)$: Given $\si$, draw $\xi\sim \calN(0,\Id_n)$ and let $y_{t_0} = t_0 \si + \sqrt{t_0}\xi$.
\end{enumerate}

We note that \cite{huang2025weak} provide a way to prove weak Poincar\'e inequalities via a measure decomposition given by stochastic localization.
A measure decomposition of $\pi$ on a measurable space $(\Om,\sF)$ is a decomposition 
$
\pi = \E_{z\sim \rh} \pi_z,
$
i.e., $\pi(A) = \E_{z\sim \rh} \pi_z(A)$ for any $A\in \sF$. More formally, $\pi_{z}(\cdot)$ is a probability kernel, and $\pi$ is the result of the kernel applied to $\rh$.

\begin{lemma}[{\cite[Lemma 4.11, Lemma A.6]{huang2025weak}}]
\label{l:wpi-convergence}
Let $\pi$ be a distribution over $\Om = \{\pm 1\}^n$ or $r\cdot \bbS^{n-1}$ and $\pi = \E_{z\sim \rh}\pi_z$ be a measure decomposition of $\pi$ such that
\begin{enumerate}
    \item For all functions $f$, $\E_{z\sim \rh} \Var_{\pi_z}[f] \ge C\Var_\pi[f]$.
    \item With probability $1-\eta$ over $z\sim \rh$, $\pi_z$ satisfies a $(c,\de)$-weak Poincar\'e inequality (with Glauber or Langevin dynamics).
\end{enumerate}
Then $\pi$ satisfies a $(c\, C, \fc{\de+\eta}{C})$-weak Poincar\'e inequality.
\end{lemma}
However, our assumptions are too weak to use this lemma. With only second moment control of $\Cov(\mu_t)$, by employing early stopping when the approximate variance conservation in the SL process is violated, 
$C$ would depend on $\eta$ as $e^{-\Theta(1/\eta)}$, which makes the WPI trivial.
We do, however, use this to show a WPI for the \emph{localized} measure, through \pref{t:sl-wpi}. 


\subsection{Consequences of assumptions}
We note some immediate consequences of the assumptions. In fact, everything we need from \Cref{a:m-error} and \Cref{a:Lipschitz-drift} is encapsulated in the following two consequences.

\begin{lemma}[KL bound between ideal and approximate process]
\label{l:KL}
    Given \Cref{a:m-error}, 
    for $t\in [0,T_1]$, 
    \[
\KL(p_t \| \hat p_t) \le \rc2 \int_0^t \ep(s)^2\,ds.
    \]
\end{lemma}
This follows from Girsanov's Theorem; see e.g., \cite[Corollary 3.7]{DLSS26}.
We define $E_t$ to be this bound:
\begin{align}
\label{e:Et}
    E_t &= \rc 2 \int_0^t \ep(s)^2\,ds.
\end{align}
Note that a reverse KL bound between $\hat p_t$ and $p_t$ allows us to use $\hat p_t$ as a proposal distribution to rejection sample from $p_t$, if an estimate of $\dd{p_t}{\hat p_t}$ were known; this is the approach taken in \cite{DLSS26}.

\begin{lemma}[Log-Sobolev inequality for time-$t$ ASL distribution]
\label{l:LSI-time-t}
Given \Cref{a:Lipschitz-drift}, $\hat p_t$ satisfies a log-Sobolev inequality with constant 
$\hat c_t = \fc{2L}{e^{2Lt}-1}$.
\end{lemma}
\begin{prf}
    This follows from the more general \pref{l:lsi-Lip-SDE}, noting that because $p_0$ is a point mass, $\al_0=\iy$.
\end{prf}

\subsection{Jarzynski's equality}
To obtain an explicit approximating distribution as a warm start, we use Jarzynski's equality. Note this is only needed for algorithmizing the WPI, in \pref{s:rho-t-lsi} onwards.
\begin{theorem}[Jarzynski's equality~\cite{vaikuntanathan2008escorted,vargas2023transport,albergo2024nets}]\label{t:je}
Let $\rh_t \propto e^{-U_t}$ be a sequence of probability densities where $U_t$ is twice differentiable with respect to $(x,t) \in [0,T]\times \R^n$.
Let $b_t$ be continuously differentiable, and suppose that uniformly for $t\in [0,T]$, $\gd U_t$ and $b_t$ are Lipschitz in $x$. Consider the coupled SDE-ODE system
    \begin{align*}
        dx_t & = \ba{-\rc2\gd U_t(x_t) + b_t(x_t)}dt + dB_t, &x_0&\sim \rh_0\\
        dw_t &= \ba{\gd \cdot b_t(x_t) - \an{\gd U_t(x_t), b_t(x_t)} - \pl_t U_t(x_t)}dt, & w_0&=0.
    \end{align*}
    Then $\fc{Z_t}{Z_0} = \fc{\int_{\R^n} e^{-U_t(x)}dx}{\int_{\R^n} e^{-U_0(x)}dx} = \E[e^{w_t}]
    $
    and for any bounded measurable function $h$,
    \[
\int_{\R^n} h(x)\rh_t(x) dx = \fc{\E[e^{w_t}h(x_t)]}{\E[e^{w_t}]}.
    \]
    Therefore, letting $\hat p_t$ be the distribution of $x_t$, we have
    \[
\dd{\rh_t}{\hat p_t}(x) = \fc{\E[e^{w_t}|x_t=x]}{\E[e^{w_t}]}.
    \]
\end{theorem}

\subsection{Gaussian analysis, Cameron--Martin spaces, and Sobolev domains}
\label{s:prelim-gaussian}

We collect elementary definitions and facts from Gaussian analysis and Malliavin calculus on path space; these are used in \pref{s:rho-t-lsi} to prove a log-Sobolev inequality for the annealed distribution $\rh_T$, and in the auxiliary lemmata of \S\,\ref{a:wiener}. For further background, see \cite[\S2--3]{hairer2016advanced} and \cite{tubaro2025introduction}. 
\begin{definition}[Gaussian measure on a Banach space]
\label{d:gaussian-banach}
    Let $(E,\norm{\cdot}_E)$ be a separable Banach space with Borel $\si$-algebra $\calB(E)$ and topological dual $E^*$. A probability measure $\ga$ on $(E,\calB(E))$ is a \vocab{centered Gaussian measure} if $\ell_\#\ga$ is a centered Gaussian distribution on $\R$ for every $\ell \in E^*$.
\end{definition}
We work concretely on the classical Wiener space over a finite time horizon.
\begin{definition}[Wiener space and Cameron--Martin space]
\label{d:wiener-space}
    Fix $0<T<\iy$. Let 
    \[
    E := C_0([0,T],\R^n) = \set{B\in C([0,T],\R^n)}{B_0 = 0}
    \]
    with the supremum norm, and let $\ga$ be the \vocab{Wiener measure} on $(E,\calB(E))$, i.e., the law of standard Brownian motion on $[0,T]$ started at $0$. The \vocab{Cameron--Martin space} of $\ga$ is
    \[
    \calH := H_0^1([0,T],\R^n) = \set{h\in E}{h \text{ absolutely continuous},\ \norm{h}\CM < \iy}, 
    \qquad 
    \norm{h}\CM^2 := \int_0^T \norm{\dot h_t}_2^2\,dt,
    \]
    a separable Hilbert space with inner product $\an{g,h}_\calH = \int_0^T \an{\dot g_t, \dot h_t}\,dt$. We write $i:\calH\hookrightarrow E$ for the inclusion and $i^*:E^*\to\calH$ for its adjoint (identifying $\calH$ with its dual).
\end{definition}

The triple $(E,\calH,\ga)$ is the prototypical example of an abstract Wiener space, and everything below extends to that generality \cite[\S8.1]{tubaro2025introduction}. The inclusion $i$ is continuous with dense image---indeed $\norm{h}_E \le \sqrt T \norm h\CM$ by Cauchy--Schwarz---but $\ga(\calH) = 0$ \cite[\S8.1.3]{tubaro2025introduction}.
Unlike in finite dimensions, $\ga$ should not be understood through a
Lebesgue density on $E$: there is no locally finite translation-invariant
Lebesgue measure on an infinite-dimensional Banach space. The
Cameron--Martin space replaces this missing density-based calculus by
identifying the directions along which translations of $\ga$ remain
absolutely continuous, and hence the directions in which gradients and
Sobolev norms are defined. Thus, $\calH$ is a $\ga$-null set of ``smooth directions'' which nevertheless governs the measure $\ga$ in the two senses recorded in \pref{l:paley-wiener} and \pref{t:cameron-martin} below.

\begin{lemma}[Paley--Wiener integrals, {\cite[\S2]{hairer2016advanced}, \cite[\S8.1.2]{tubaro2025introduction}}]
\label{l:paley-wiener}
    For $h\in \calH$, define $X_h : E\to \R$ by the It\^o integral
    \[
    X_h(B) := \int_0^T \an{\dot h_t, dB_t}.
    \]
    Then the following hold.
    \begin{enumerate}
        \item The map $h\mapsto X_h$ is a linear isometry from $\calH$ into $L^2(\ga)$: each $X_h$ is a centered Gaussian random variable, and $\E_\ga[X_g X_h] = \an{g,h}_\calH$. In particular, if $\{e_i\}_{i\in\N}$ is an orthonormal basis of $\calH$, then $\{X_i := X_{e_i}\}_{i\in\N}$ are i.i.d.\ $\calN(0,1)$.
        \item Every $\ell\in E^*$ agrees $\ga$-a.s.\ with a Paley--Wiener integral: $\ell(B) = X_{i^*\ell}(B)$ as elements of $L^2(\ga)$. For the coordinate evaluation functionals this is the identity $\E_\ga[X_h(B)\, B_t] = h_t$, which follows from the It\^o isometry; the general case follows by continuous linear extension.
    \end{enumerate}
\end{lemma}

Throughout \pref{s:rho-t-lsi} and \S\,\ref{a:wiener} we fix an orthonormal basis $\{e_i\}_{i\in\N}$ of $\calH$ and write 
\[
H_k := \operatorname{span}(e_1,\ldots,e_k)\cong \R^k,
\]
$H_k^\perp$ for its orthogonal complement in $\calH$, and $P_{H_k}, P_{H_k^\perp}$ for the corresponding orthogonal projections. We also write $\calF_k := \si(X_1,\ldots,X_k)$ for the $\si$-algebra generated by the first $k$ Gaussian coordinates (where $X_i$ is defined in \pref{l:paley-wiener}); conditioning on $\calF_k$ is the ``revelation'' of $k$ Gaussians used in the cylindrical approximation of \pref{t:lsi-CM-pert}.

\begin{theorem}[Cameron--Martin, {\cite[\S8.1.3]{tubaro2025introduction}}]
\label{t:cameron-martin}
    For $h\in E$, let $\ga_h := \ga(\cdot - h)$ denote the shifted measure. If $h\in \calH$, then $\ga_h$ and $\ga$ are mutually absolutely continuous, with density
    \[
    \dd{\ga_h}{\ga}(B) = \exp\pa{X_h(B) - \rc 2\norm{h}\CM^2}.
    \]
    If $h\notin\calH$, then $\ga_h$ and $\ga$ are mutually singular.
\end{theorem}

In particular, translation by any $h\in \calH$ preserves $\ga$-null sets, so for a function $V$ defined only $\ga$-a.s.\ the quantity $V(B+h)$ is well-defined for $\ga$-a.e.\ $B$, and the following definition makes sense.

\begin{definition}[Cameron--Martin Lipschitz functions]
\label{d:cm-lipschitz}
    A $\ga$-measurable, $\ga$-a.s.\ finite function $V: E\to\R$ is \vocab{$K$-Lipschitz in the Cameron--Martin norm} ($K$-CM-Lipschitz) if for $\ga$-a.e.\ $B\in E$ and every $h\in \calH$,
    \[
    |V(B+h) - V(B)| \le K \norm{h}\CM.
    \]
    Similarly, a map $\Phi: E\to F$ into a normed space $F$ (for us, $F = E$ with $\norm{\cdot}_E$ or $F=\R^n$ with $\norm{\cdot}_2$) is $K$-CM-Lipschitz if $\norm{\Phi(B+h)-\Phi(B)}_F \le K\norm h\CM$ for $\ga$-a.e.\ $B$ and all $h\in\calH$. We write $\lipnorm{\Phi}$ for the smallest such $K$.
\end{definition}

Note that a CM-Lipschitz function need not be continuous on $E$, as the CM norm is much stronger than the norm of $E$; conversely, an $\ell$-Lipschitz function on $(E, \norm\cdot_E)$ is $\ell\sqrt T$-CM-Lipschitz. Although $\ga(\calH)=0$, CM-regularity controls $\ga$-measurable sets of positive measure through the following isoperimetric inequality of Borell, the infinite-dimensional counterpart of \pref{t:g-isop}.

\begin{theorem}[Borell's inequality, {\cite[Theorem 3.53]{hairer2026advanced}}]
\label{t:borell}
    Let $A\in\calB(E)$ and $\al\in\R$ be such that $\ga(A) \ge \Phi(\al)$, where $\Phi$ is the standard Gaussian cdf as in \pref{t:g-isop}. Then for every $\ep>0$,
    \[
    \ga(A + B_\ep) \ge \Phi(\al+\ep), \qquad \text{where } B_\ep := \set{h\in\calH}{\norm h\CM\le\ep}
    \]
    and $A + B_\ep := \set{a+h}{a\in A,\ h\in B_\ep}$.
\end{theorem}

Combined with \pref{d:cm-lipschitz}, Borell's inequality yields sub-Gaussian concentration of CM-Lipschitz functions around their medians, and in particular finiteness of all their exponential moments; we prove the precise statement we need in \pref{lem:infinite-dim-gaussian-ui}. This is the path-space substitute for \pref{l:conc-lsi}.

We now turn to differential structure. Since $\ga$-generic paths are nowhere differentiable, derivatives on $(E,\ga)$ are taken only along Cameron--Martin directions; this is the starting point of Malliavin calculus \cite[\S3]{hairer2016advanced}, \cite[\S5, \S8.1.4]{tubaro2025introduction}.

\begin{definition}[Cylindrical functions and Malliavin derivative]
\label{d:malliavin}
    We write $C_b^1(E)$ for the class of bounded cylindrical functions\footnote{This notation is specific to the cylindrical
    path-space calculus and should not be confused with Fr\'echet
    $C^1$ functions on the Banach space $E$.}
    \[
        F(B)=f(X_{h_1}(B),\ldots,X_{h_k}(B)),
    \]
    where $k\in\N$, $h_1,\ldots,h_k\in\calH$, and
    $f\in C_b^1(\R^k)$.

    We say that $F$ is \vocab{smooth cylindrical} if the same representation
    holds with $f\in C_b^\infty(\R^k)$.

    For $F\in C_b^1(E)$, the \vocab{Malliavin derivative}
    $D_HF:E\to\calH$ is
    \[
        D_H F(B)
        :=
        \sum_{i=1}^k
        \pl_i f\pa{X_{h_1}(B),\ldots,X_{h_k}(B)}\, h_i .
    \]
    For $h\in\calH$, the directional derivative is
    \[
        D_hF
        :=
        \an{D_HF,h}_\calH
        =
        \ddd\ep\bigg|_{\ep=0} F(B+\ep h),
    \]
    consistently with \pref{t:cameron-martin}.
\end{definition}
Note that for maps $F:E\to\R^n$, the derivative is defined coordinate-wise; then $D_HF(B):\calH\to\R^n$ is a bounded linear operator and $\norm{D_HF(B)}\CM$ denotes its operator norm. The identity $D_h F = \ddd\ep|_{\ep=0}F(B+\ep h)$ shows in particular that if $F \in C_b^1(E)$ is $K$-CM-Lipschitz, then $\norm{D_H F}\CM \le K$ pointwise. The same conclusion holds without the cylindrical structure by a Rademacher-type theorem on Wiener space.

The operator $D_H$ is used through its closure, i.e., on Sobolev domains defined by completion under the graph norm.
Since we tilt the base measure, we state the definition for a general reference measure.
Recall that an operator \(T:\mathcal A\subset L^2(\nu)\to L^2(\nu;\calH)\)
is \vocab{closable} if the closure of its graph is again the graph of an operator.
Equivalently, it suffices to test sequences converging to zero: whenever
\(F_j\in\mathcal A\), \(F_j\to0\) in \(L^2(\nu)\), and \(TF_j\to G\) in \(L^2(\nu;\calH)\), one must have \(G=0\). This
criterion guarantees that the derivative of a graph-norm limit is independent of the approximating sequence.

\begin{definition}[Path-space Sobolev domains]
\label{d:sobolev-domains}
Let $\nu$ be a probability measure on $(E,\calB(E))$ for which
\[
    D_H:C_b^1(E)\subset L^2(\nu)\to L^2(\nu;\calH)
\]
is closable. The \vocab{Sobolev domain} $\calD^{1,2}(\nu)$ is the domain of the closure
of $D_H$, equipped with the graph norm
\[
    \|F\|_{\calD^{1,2}(\nu)}^2
    :=
    \int_E F^2\,d\nu
    +
    \int_E \|D_HF\|_{\rm CM}^2\,d\nu .
\]
Equivalently, $F\in\calD^{1,2}(\nu)$ if there exist
$F_j\in C_b^1(E)$ and $G\in L^2(\nu;\calH)$ such that
\[
    F_j\to F \quad\text{in }L^2(\nu),
    \qquad
    D_HF_j\to G \quad\text{in }L^2(\nu;\calH).
\]
By closability, $G$ is independent of the approximating sequence, and we
write $D_HF:=G$.
\end{definition}

\begin{lemma}[CM-Lipschitz functions are Malliavin differentiable, {consequence of \cite[Theorem 4]{ambrosio2009metric}}]
\label{l:cm-lip-malliavin}
Let \(1<p<\infty\). Let \(V:E\to\R\) be \(K\)-CM-Lipschitz and \(V\in L^p(\ga)\). Then \(V\in\calD^{1,p}(\ga)\) (defined in the special case that $p=2$ in \pref{d:sobolev-domains}, $p > 1$ is similar) and
\[
    \norm{D_HV}_{\CM}\le K
    \qquad \ga\text{-a.s.}
\]
\end{lemma}

\begin{remark*}
    We use \pref{l:cm-lip-malliavin} only through maps whose CM-Lipschitz constants we compute explicitly by Gr\"onwall estimates (\pref{lem:cm-lipscitz-one}, \pref{lem:cm-lipschitz-two}); for these, membership in the Sobolev domain and the derivative bound can also be verified directly along the cylindrical approximations in the proof of \pref{t:lsi-CM-pert}.
\end{remark*}

For \(\nu=\gamma\), closability is the standard Malliavin Sobolev construction.
We also use tilted measures \(d\nu^V=Z_V^{-1}e^Vd\gamma\), where \(V\) is finite and CM-Lipschitz; their closability is established in \pref{l:closability-cm-tilts}.

The standard closability statement is usually stated for smooth cylindrical functions; the same closed operator is obtained from the cylindrical $C_b^1$ core by finite-dimensional mollification in the Gaussian
coordinates.
For the classical Sobolev theory on $\R^n$ underlying $\calW^{1,2}(\mu)$---weak derivatives, approximation by smooth functions, and mollification---we refer to \cite[\S8--9]{brezis2011functional}; the properties of mollifiers used are collected in \pref{lem:mollifiers}.

We record the chain rule, which we use to transfer derivative bounds through the It\^o map.

\begin{lemma}[Chain rule, {\cite[\S5.1]{tubaro2025introduction}}]
\label{l:malliavin-chain-rule}
Let $\nu$ be a path-space reference measure for which
$\calD^{1,2}(\nu)$ is defined as in \pref{d:sobolev-domains}. Let
$\Phi=(\Phi^{(1)},\ldots,\Phi^{(n)}):E\to\R^n$ with
$\Phi^{(j)}\in\calD^{1,2}(\nu)$ for each $j$, and let
$f\in C_b^1(\R^n)$. Then $f\circ\Phi\in\calD^{1,2}(\nu)$ and
\[
    D_H(f\circ\Phi)(B)
    =
    D_H\Phi(B)^\sT\,\nabla f(\Phi(B)),
    \qquad \nu\text{-a.s.}
\]
Here $D_H\Phi(B):\calH\to\R^n$ denotes the operator whose $j$th component
is $D_H\Phi^{(j)}(B)$.
\end{lemma}

Moreover, when $\Psi$ is CM-Lipschitz, the chain rule (\pref{l:malliavin-chain-rule}) bounds the path-space Dirichlet form of $f\circ\Psi$ by $\lipnorm{\Psi}^2$ times the Euclidean Dirichlet form of $f$, so a LSI for $\nu$ on path space transfers to a LSI for $\Psi_\#\nu$ on $\R^n$ with constant degraded by $\lipnorm{\Psi}^2$; this is carried out in \pref{prop:lsi-rhT-bounded}.
 
The path-space Sobolev domain above should be distinguished from the ordinary Euclidean weighted Sobolev domain used after pushing forward to the endpoint law.

\begin{definition}[Euclidean weighted Sobolev domain]
\label{d:euclidean-sobolev-domain}
For a probability measure $\mu$ on $\R^n$, the Euclidean Sobolev domain
$\calW^{1,2}(\mu)$ is the closure of $C_b^1(\R^n)$ under the graph norm
\[
    \norm{f}_{\calW^{1,2}(\mu)}^2
    :=
    \int_{\R^n} f^2\,d\mu
    +
    \int_{\R^n} \norm{\nabla f}_2^2\,d\mu .
\]
\end{definition}

Finally, we state a critical fact about entropy on Gaussian spaces that drives \pref{s:rho-t-lsi}. This is the Gaussian log-Sobolev inequality of Gross, which holds with a dimension-free constant for any centered Gaussian measure on a separable Banach space, with the Dirichlet form measured in the Cameron--Martin norm; it is the infinite-dimensional counterpart of the fact that $\calN(0,\Id_n)$ satisfies a LSI with constant $1$.

\begin{theorem}[Gaussian log-Sobolev inequality, {\cite[(1)]{capitaine1997martingale}, \cite{3f89ed44-5628-3112-93f5-de6a2b2ab43a}}]
\label{t:gross-lsi}
    Let $\ga$ be a centered Gaussian measure on a separable Banach space $E$ with Cameron--Martin space $\calH$. Then for every $F\in\calD^{1,2}(\ga)$,
    \[
    \Ent_\ga[F^2] \le 2\int_E \norm{D_H F}\CM^2 \,d\ga.
    \]
\end{theorem}

We apply \pref{t:gross-lsi} to the ``tail'' Gaussian measures $\ga_{>k} = (R_k)_\#\ga$ obtained by projecting out the first $k$ coordinates, which are centered Gaussian measures on $E$ with Cameron--Martin space $H_k^\perp$; see \pref{lem:gamma-tail-gaussian}. 

\input{decomp}

\section{Weak Poincaré inequality from Approximate Stochastic Localization}
\label{s:wpi-SL}
We obtain a weak Poincar\'e inequality for the time-$t$ SL distribution, for both Langevin and the proximal sampler. We first show a $s$-conductance bound for Langevin, then obtain a WPI by Cheeger's inequality, and then transfer this to the proximal sampler.
Finally, we combine this with \pref{c:wpi-gd-ps} to prove \pref{t:main} in \pref{s:proof}.

\subsection{Coupling with small distance}
Under a log-Sobolev inequality, large subsets cannot be too far away from each other.
\begin{lemma}[Distance between sets under LSI]
\label{l:set-separation-lsi}
    Let $\mu$ be a measure on $\R^n$ satisfying a log-Sobolev inequality with constant $\ge c$. Let $A, B$ be disjoint subsets with $\mu(A)=a$ and $\mu(B)=b$. Then 
    $d(A,B)\le \sfc{2}{c} \pa{\sqrt{\ln \prc a}+ \sqrt{\ln \prc b}}$. 
\end{lemma}
To show this, we use the fact that the distance function is Lipschitz, and concentration of Lipschitz functions.
\begin{prf}
    Consider the 1-Lipschitz function $d(A,x)$. 
    Let $\bar d = \E_\mu d(A,x)$. 
    By \pref{l:conc-lsi}, 
    \[
a \le  \mu(d(A,x) = 0) \le e^{-\fc{c \bar d^2}2} 
\implies 
\bar d \le \sfc{2\ln \prc a}{c}. 
    \]
    We then have
    \[
    b \le 
\mu(d(A,x) \ge d(A,B)) \le 
e^{-\fc{c (d(A,B)-\bar d)^2}{2}} \implies 
d(A,B) \le \sfc{2}{c} \pa{\sqrt{\ln \prc a}+ \sqrt{\ln \prc b}}.\qedhere
    \]
\end{prf}
A greedy matching then shows we can couple most of the mass of $A$ and $B$.
\begin{lemma}[Coupling sets with small distance]
\label{l:couple-distance}
    Keep the assumptions of \Cref{l:set-separation-lsi}. Suppose $\ep\le a\le b$, and assume moreover that $\mu$ is non-atomic. Then for any $d> \sfc{2}{c} \pa{\sqrt{\ln \prc {\ep}}+ \sqrt{\ln \prc {\ep + b-a}}}$, there exists a coupling $(X,Y)$ of $\mu$ with itself such that
    \[
\P(X\in A, \, Y\in B, \,d(X,Y)\le d) \ge a-\ep.
    \]
\end{lemma}
\begin{prf}
    Let $0<\de<\ep$. Let $K=[-N\de, N\de]^n$ be a large enough box so that $\mu(K)\ge 1-\de$. Now divide it into boxes of side length $\de$. 
    We inductively construct the coupling. 
    Let $A_0=A\cap K$, $B_0=B\cap K$. We will maintain $\mu(B_t)-\mu(A_t) = \mu(B_0)-\mu(A_0) \ge b-a-\de$.
    At step $t$, while $\mu(A_t)>\ep - \de$, we have 
    by \Cref{l:set-separation-lsi}, there exist points $x\in A_t$, $y\in B_t$ such that $d(x,y) \le \sfc{2}{c} \pa{\sqrt{\ln \prc{\ep-\de}} + \sqrt{\ln \prc{\ep + b-a -2\de}}}$, and $x$, $y$ are in the boxes $K_x$, $K_y$ of side length $\delta$, respectively, which contain nonzero mass in $A_t$ and $B_t$, respectively. Consider 2 cases.
    \begin{enumerate}
        \item If $\mu(A_t \cap K_x)\le \mu(B_t\cap K_y)$, then choose a measurable subset $K_y'\subset B_t \cap K_y$ with $\mu(K_y') = \mu(A_t \cap K_x)$, couple the points in $K_y'$ and $A_t \cap K_x$, let $A_{t+1} = A_t \bs K_x$, and let $B_{t+1} = B_t \bs K_y'$. 
        \item If $\mu(A_t \cap K_x)\ge \mu(B_t\cap K_y)$, then choose a measurable subset $K_x'\subset A_t \cap K_x$, and do the reverse of the above.
    \end{enumerate}
    Note this process must terminate with $\mu(A_t)\le \ep-\de$, since one box is emptied each time. Arbitrarily couple the rest of $\mu$. We have that the coupling satisfies 
    \[
\P\pa{X\in A, \, Y\in B, \, d(X,Y)\le 
\sfc{2}{c} \pa{\sqrt{\ln \prc{\ep-\de}} + \sqrt{\ln \prc{\ep + b-a -2\de}}} + 
2 \sqrt{n}\,\de
} \ge a-\ep.
    \]
    Since $d$ was specified by a strict inequality, choose $\de>0$
small enough that the constructed coupling has good-pair distance at most
$d$ and failure probability at most $\ep$.
\end{prf}

\subsection{Langevin diffusion}


The following lemma will tell us that given a set with nontrivial measure under $p_t$, we can find a not-too-small subset, measured under $\hat p_t$, where the density under $\hat p_t$ is not too much smaller.
\begin{lemma}
\label{l:p-subset-hat-p}
    Suppose that $p, \hat p$ are probability measures such that $\KL(p\|\hat p)\le E$. Let $0<a\le 1$ and suppose that $p(A)\ge a$. Define
    \[
        A^\sharp := A \cap \bc{x:\dd p{\hat p}(x)\ge \fc a2}.
    \]
    Then
    \[
        \hat p(A^\sharp) \ge \fc{a}{2e^{\fc{2(E+1)}{a}}}.
    \]
\end{lemma}
\begin{prf}
Let $r=\dd{p}{\hat p}$, which exists $p$-a.s. since $\KL(p\|\hat p)<\iy$.  By definition of $A^\sharp$,
\begin{align*}
    p(A\bs A^\sharp)
    = \int_{A\bs A^\sharp} r\,d\hat p
    \le \fc a2 \hat p(A\bs A^\sharp)
    \le \fc a2.
\end{align*}
Thus, if $b:=p(A^\sharp)$, then $b\ge p(A)-a/2\ge a/2$.  By data processing for the partition $\{A^\sharp,(A^\sharp)^c\}$,
\begin{align*}
    E &\ge \KL(p\|\hat p) \\
    &\ge b\ln\fc{b}{\hat p(A^\sharp)}
    +(1-b)\ln\fc{1-b}{1-\hat p(A^\sharp)} \\
    &\ge b\ln\fc{b}{\hat p(A^\sharp)}-1,
\end{align*}
where the last line uses $x\ln(x/y)\ge -1$ for $x,y\in[0,1]$. Hence
\[
    \hat p(A^\sharp)\ge b\exp\pa{-\fc{E+1}{b}}.
\]
Since $b\mapsto b\exp(-(E+1)/b)$ is increasing on $(0,1]$ and $b\ge a/2$, this gives
\[
    \hat p(A^\sharp)\ge \fc a2\exp\pa{-\fc{2(E+1)}a}.\qedhere
\]
\end{prf}
We are now ready to prove a WPI for $p_{T_0}$.
\begin{theorem}[$s$-conductance and WPI for  $p_{T_0}$]
\label{t:wpi-pT}
    Let $s\le \rc 2$. 
    Suppose \pref{a:m-error} and \pref{a:Lipschitz-drift} hold.
    We have 
    \begin{align*}
        \Phi_s(p_{T_0})  =\exp\pa{
-O\pa{\fc{E_{T_1}+1}{s} \cdot \fc{T_0}{T_1(T_1-T_0)\hat c_{T_1}}}
}
    \end{align*}
    and $p_{T_0}$ satisfies a $\pa{\exp\pa{
-O\pa{\fc{E_{T_1}+1}{s} \cdot \fc{T_0}{T_1(T_1-T_0)\hat c_{T_1}}}
}, s}$-WPI.
\end{theorem}
Here, $E_t$ is given in \pref{l:KL} and $\hat c_T$ is given in \pref{l:LSI-time-t}.
\begin{prf}
    Let $S \subseteq \R^n$ be a subset with $p_{T_0}(S)= s$.
    We would like to show $S$ has large boundary measure. 
    Note that
    \begin{align*}
        p_{T_0}(S) = 
        p_{T_1}K_{T_1\to T_0}(S) = 
        \int_{\R^n} K_{T_1\to T_0}(x,S) p_{T_1}(dx)
    \end{align*}
    Let 
    \begin{align*}
        A_1 :&= \set{x}{K_{T_1\to T_0}(x,S)<\fc s4}\\
        B_1 :&= \set{x}{K_{T_1\to T_0}(x,S)> 1-\fc s4}.
    \end{align*}
    We consider 2 cases, based on whether $A_1\cup B_1$ or $(A_1\cup B_1)^c$ is large under $p_{T_1}$. If $(A_1\cup B_1)^c$ is large, then considering these $x$ gives significant expansion of $S$ since $K_{T_1\to T_0}(x,S)$ is bounded away from 1. If $A_1\cup B_1$ is large, then we show $A_1,B_1$ both must have significant mass, and use a coupling with small distance given by \pref{l:couple-distance} to show good expansion after Gaussian noising.

    \ppart{Case 1: $p_{T_1}((A_1\cup B_1)^c)>\fc s4$}
    Let $\si = \si_{T_0,T_1} = \sqrt{T_0\pa{1-\fc{T_0}{T_1}}}$. Note that $K_{T_1\to T_0}(x,\cdot)$ is Gaussian with standard deviation $\si$. 
    Let $f(x) = \ph(\Phi^{-1}(x))$.
    In this case, we have by Fatou's Lemma and the Gaussian isoperimetric inequality (\pref{t:g-isop}) that 
    \begin{align}
    \label{e:mu-plus-kernel}
        p_{T_0}^+(S) = 
        \liminf_{h\to 0^+} \rc{h}
        p_{T_0}(\pl_h S)
        &= \liminf_{h\to 0^+} \int_{\R^n} \rc{h} K_{T_1\to T_0}(x,\pl_h S) p_{T_1}(dx)\\
        \nonumber
        &\ge_{\textup{Fatou}} \int_{\R^n} \liminf_{h\to 0^+} \rc{h} K_{T_1\to T_0}(x,\pl_h S) p_{T_1}(dx)\\
        \nonumber
        &\ge_{\textup{\pref{t:g-isop}}} \int_{(A_1\cup B_1)^c} \rc{\si} 
        \ph\pa{\Phi^{-1}\pa{\fc s4}}
        p_{T_1}(dx)\\
        \nonumber 
        &\ge p_{T_1}((A_1\cup B_1)^c) \rc{\si}
        \ph\pa{\Phi^{-1}\pa{\fc s4}}
        > \fc{s}{4\si}\cdot \fc s4, 
    \end{align}
    where in the last inequality the fact that for $u\ge 1$, 
    \begin{align}
    \Phi(-u)\le \rc{\sqrt{2\pi}u}e^{-\rc 2 u^2} \le \ph(u)
    \label{e:Phi-tail-ub}
    \end{align}
    by \cite[Proposition 2.1.2]{vershynin2018high}, so for $x\le \rc 8$, taking $u=-\Phi^{-1}(x)$ gives $x\le \ph(\Phi^{-1}(x))$.

\ppart{Case 2: $p_{T_1}((A_1\cup B_1)^c)\le \fc s4$}
    We first record that both $A_1$ and $B_1$ have $p_{T_1}$-mass at least $s/2$.  For $B_1$,
    \begin{align*}
        s = p_{T_0}(S) &= \int_{\R^n} K_{T_1\to T_0}(x,S) p_{T_1}(dx)\\
        &\le p_{T_1}(A_1) \fc s4 +
        \ub{p_{T_1}((A_1\cup B_1)^c )}{\le s/4} \pa{1-\fc s4} + 
        p_{T_1}(B_1)
        \le \fc s2 + p_{T_1}(B_1),
    \end{align*}
    so $p_{T_1}(B_1)\ge s/2$.  For $A_1$, since
    \[
        s = p_{T_0}(S) \ge \pa{1-\fc s4}p_{T_1}(B_1),
    \]
    we have $p_{T_1}(B_1)\le s/(1-s/4)$, and therefore, using $s\le 1/2$,
    \begin{align*}
        p_{T_1}(A_1)
        &=1-p_{T_1}(B_1)-p_{T_1}((A_1\cup B_1)^c)\\
        &\ge 1-\fc{s}{1-s/4}-\fc s4
        \ge \fc s2.
    \end{align*}

    Apply \pref{l:p-subset-hat-p} with $a=s/2$, $p=p_{T_1}$, $\hat p=\hat p_{T_1}$, and the KL bound from \pref{l:KL}.  Let
    \[
        \hat s := \fc{s}{4}e^{-4(E_{T_1}+1)/s}.
    \]
    Then the sets
    \[
        A^\sharp := A_1\cap\bc{x:\dd{p_{T_1}}{\hat p_{T_1}}(x)\ge \fc s4},
        \qquad
        B^\sharp := B_1\cap\bc{x:\dd{p_{T_1}}{\hat p_{T_1}}(x)\ge \fc s4}
    \]
    satisfy $\hat p_{T_1}(A^\sharp),\hat p_{T_1}(B^\sharp)\ge \hat s$.  Since $\hat p_{T_1}$ has no atoms, replace them by measurable subsets $\hat A\subeq A^\sharp$ and $\hat B\subeq B^\sharp$ with $\hat p_{T_1}(\hat A)=\hat p_{T_1}(\hat B)=\hat s$.

     By \Cref{l:couple-distance} with $\ep=\fc{\hat s}2$, there exists a coupling $\hat \P$ of $\hat p_{T_1}$ with itself such that
    \[
\hat \P \Bigg(X\in \hat A ,\, Y\in \hat B,\, d(X,Y)\le \ub{4\sfc{\ln \pf{2}{\hat s}}{\hat c_{T_1}}}{=:\bar d}\Bigg)\ge \fc{\hat s}2.
    \]
    Let $G$ denote the event inside the probability above.  On $\hat A\cup\hat B$ we have $\dd{p_{T_1}}{\hat p_{T_1}}\ge \fc{s}{4}$.  Hence the sub-probability measure $\fc{s}{4}\hat\P|_G$ has both marginals dominated by $p_{T_1}$.  Extending this sub-coupling by any coupling of the two residual marginals gives a coupling $\P$ of $p_{T_1}$ with itself such that
    \begin{align}\label{e:couple-prob}
\P \pa{X\in \hat A ,\, Y\in \hat B,\, d(X,Y)\le \bar d}\ge \fc{\hat s}2 \cdot \fc s4 = \fc{s^2}{32} e^{-4(E_{T_1}+1)/s}.
    \end{align}
Now given $x_1,y_1$ with $\|x_1-y_1\|\le \bar d$, we have that we can couple the distributions
\[
x_0 \sim 
K_{T_1\to T_0}(x_1,\cdot) = 
\calN\pa{\fc{T_0}{T_1} x_1, T_0 \pa{1-\fc{T_0}{T_1}}\Id_n}, \qquad 
y_0 \sim 
K_{T_1\to T_0}(y_1,\cdot) = 
\calN\pa{\fc{T_0}{T_1} y_1, T_0 \pa{1-\fc{T_0}{T_1}}\Id_n}
\]
with overlap $\ge 2\Phi\pf{-\fc{T_0}{2T_1}\bar d}{\sqrt{T_0 \pa{1-\fc{T_0}{T_1}}}} = 2\Phi\pa{-\fc d2\sfc{T_0}{T_1(T_1-T_0)}}$, where $\Phi$ is the cdf of the standard normal. 
By \cite[Proposition 2.1.2]{{vershynin2018high}},  for $x\ge 1$, 
\begin{align}
\label{e:cdf-lb}
    \Phi(-x) \ge\fc{x}{x^2+1}\rc{\sqrt{2\pi}} e^{-\rc 2 x^2}  \ge \rc{6x}e^{-\rc 2x^2}.
    \end{align}
Hence 
\begin{align}
\nonumber
\Phi\pa{-\fc{\bar d}2\sfc{T_0}{T_1(T_1-T_0)}}
&\ge \rc{6} \pa{2\sqrt{\fc{\ln \pf{2}{\hat s}}{\hat c_{T_1}}\fc{T_0}{T_1(T_1-T_0)}}}^{-1}
\exp\pa{
- 2 \fc{\ln \pf{2}{\hat s}}{\hat c_{T_1}}\fc{T_0}{T_1(T_1-T_0)} 
}\\
& =\exp\pa{
-O\pa{\fc{E_{T_1}+1}{s} \cdot \fc{T_0}{T_1(T_1-T_0)\hat c_{T_1}}}
}.
\label{e:Phi-calc}
\end{align} 
    Hence 
    \begin{align*}
\ab{K_{T_1\to T_0}(x_1,S) - K_{T_1\to T_0}(y_1,S)}
\le \TV(K_{T_1\to T_0}(x_1,\cdot), K_{T_1\to T_0}(y_1,\cdot))
= 1-2\Phi\pa{-\fc{\bar d}2\sfc{T_0}{T_1(T_1-T_0)}}
    \end{align*}
    and 
    \begin{align}
    \label{e:max-AB}
\max\bc{K_{T_1\to T_0}(x,S), 
 1-K_{T_1\to T_0}(y,S)} &\ge \Phi\pa{-\fc {\bar d}2\sfc{T_0}{T_1(T_1-T_0)}}.
    \end{align}
    Hence following \eqref{e:mu-plus-kernel}, 
        \begin{align*}
        \liminf_{h\to 0^+} \rc{h}  p_{T_0}(\pl_h S) &\ge  
        \int_{\set{(x,y)\in A_1\times B_1}{d(x,y)\le \bar d}} \liminf_{h\to 0^+} \rc{h} \pa{K_{T_1\to T_0}(x,\pl_hS) + K_{T_1\to T_0}(y,\pl_h S)} d\P(x,y) \\
        &\ge \int_{\set{(x,y)\in A_1\times B_1}{d(x,y)\le \bar d}} \rc{\si}\sfc{2}{\pi} 
        \max\bc{K_{T_1\to T_0}(x,S), 
 1-K_{T_1\to T_0}(y,S)} d\P(x,y)\\
 &\ge_{\eqref{e:couple-prob}, \eqref{e:max-AB}} \fc{s^2}{32} e^{-\fc{4(E_{T_1}+1)}s}\cdot  \rc{\si}\sfc{2}{\pi} \Phi\pa{-\fc d2\sfc{T_0}{T_1(T_1-T_0)}}\\
 &=_{\eqref{e:Phi-calc}} \exp\pa{
-O\pa{\fc{E_{T_1}+1}{s} \cdot \fc{T_0}{T_1(T_1-T_0)\hat c_{T_1}}}
}
    \end{align*}
    Combining the 2 cases gives the inequality for $\Phi_s$. The weak Poincar\'e inequality follows from Cheeger's inequality, \pref{t:cheeger}.
\end{prf}

\subsection{From Langevin to the proximal sampler}

We now transfer the WPI from Langevin dynamics to the proximal sampler.
\begin{lemma}[WPI for proximal sampler]
\label{l:prox-wpi}
    Let $0<t_0<t_1 \le \iy$.
    If $p_t$ satisfies a $(c_t,\de)$-WPI for $t\in [t_0,t_1]$, then the proximal sampler $\Kprox{t_0,t_1}$ satisfies a 
    $( \rc2[1-e^{-2\int_{t_0}^{t_1}c_{t}dt}], \de)$-WPI. In particular, under \pref{a:m-error} and \pref{a:Lipschitz-drift}, for all $\ep>0$, $\Kprox{T_0,\iy}$ satisfies a 
    $\pa{\exp\pa{
-O\pa{\fc{E_{T_1}+1}{\ep(T_1-T_0)\hat{c}_{T_1}} }
}, \ep}$-WPI.
\end{lemma}
\begin{prf}
In order to apply the results of \cite{chen2022improved}, we let $q_t$ be the distribution of $\fc{y_t}{t}$ when $y_t\sim p_t$.
Suppose that $p_t$ satisfies a $(c_t,\de)$-WPI. Then $q_t$ satisfies a $(t^2c_t,\de_t)$-WPI. Let $M_c(y) =cy$ denote multiplication by $c$, and let $K'_{s\to s'}$ be the kernel such that $M_{1/t^{\prime*}} K^{*}_{t\to t'}\mu =K^{\prime *}_{1/t\to 1/t'} M_{1/t*}\mu$ for any $t,t'$.

Let $s_i=1/t_i$. 
Note $K'_{s_1\to s_0}(x_1,\cdot)$ is heat flow (forward diffusion) given by $\calN(x_1,(s_0-s_1)I_n)$, and $K'_{s_0\to s_1}$ is the backward diffusion process.

\ppart{Forward step}
By \cite[Lemma 12]{chen2022improved}, for $s\ge s_1$, for $\nu_{s} := K^{\prime *}_{s_1\to s}\nu_{s_1}$, 
\begin{align*}
    \ddd s \chi^2(\nu_s\|q_s) &= -\E\ba{\ve{\gd \fc{\nu_s}{q_s}}^2}
\end{align*}
Equivalently, for $g_s := K^{\prime}_{s_1\to s}g$,
\begin{align*}
    \ddd s\Var_{q_s}(g_s) &= 
    -\sE_{q_s}(g_s,g_s)\le  -\ba{(1/s)^2c_{1/s} [(\Var_{q_s}(g_s) - \de \osc(g_s)^2)\vee 0]}.
\end{align*}
Noting that $\osc(g_s)$ is non-increasing, we have by Gr\"onwall that
\begin{align*}
    \Var_{q_{s_0}}(g_{s_0}) & \le 
    \de \osc(g_s)^2 + 
    e^{-\int_{s_1}^{s_0}(1/s)^2 c_{1/s}ds}
    \pa{\Var_{q_{s_1}}(g) - \de \osc(g)^2}.
\end{align*}

\ppart{Backward step} By \cite[Lemma 15]{chen2022improved} and using entirely analogous arguments, letting $g_s:= K'_{s_0\to s}g$,
\begin{align*}
    \Var_{q_{s_1}}(g_{s_1}) & \le 
    \de \osc(g)^2 + 
    e^{-\int_{s_1}^{s_0}(1/s)^2 c_{1/s}ds}
    \pa{\Var_{q_{s_0}}(g) - \de \osc(g)^2}.
\end{align*}
Combining the backward and forward steps,
\begin{align*}
    \Var_{q_{s_0}}(K'_{s_1\to s_0} K'_{s_0\to s_1} g)
    &\le \de \osc(g)^2 + e^{-2\int_{s_1}^{s_0}(1/s)^2 c_{1/s}ds}
    \pa{\Var_{q_{s_0}}(g) - \de \osc(g)^2}
\end{align*}
Note that $\int_{s_1}^{s_0}(1/s)^2 c_{1/s}\,ds = \int_{t_0}^{t_1} c_t\,dt.$
The result then follows from \pref{l:var-to-wpi} below and noting that the same inequalities hold for $\Kprox{t_1,t_0}$ by bijectivity. 

In our case, plugging in the bound from \pref{t:wpi-pT}, we have
\begin{align*}
    \int_{T_0}^{\iy} c_t\,dt
    &\ge \int_{T_0}^{\fc{T_0+T_1}2}
    \exp\pa{
-O\pa{\fc{E_{T_1}+1}{\ep} \cdot \fc{t}{T_1(T_1-t)\hat{c}_{T_1}}}
} \,dt   
= \exp\pa{
-O\pa{\fc{E_{T_1}+1}{\ep(T_1-T_0)\hat{c}_{T_1}} }
},
\end{align*}
and note that 
$\rc2[1-e^{-2\int_{t_0}^{t_1}c_{t}dt}] = \Om(\int_{t_0}^{t_1}c_{t}dt \wedge 1)$.
\end{prf}

\begin{lemma}
\label{l:var-to-wpi}
    Let $P$ be a reversible Markov kernel with stationary distribution $\mu$, and suppose $\Var_\mu(Pg) \le (1-c) \Var_\mu(g) + c \de \osc(g)^2$. Then $P$ satisfies a $\pa{\fc c2, \de}$-WPI.
\end{lemma}
\begin{prf}
    We have
    \begin{align}
\nonumber
\sE_\mu(g,g)
&= \an{(g-\E_\mu g),(\Id-P)(g-\E_\mu g)}_\mu \\
&\ge \rc 2 \an{(g-\E_\mu g),(\Id-P)(\Id+P)(g-\E_\mu g)}
\label{e:I2-P2}
\\
\nonumber
&= \rc 2 \ba{\Var_\mu(g) - \Var_\mu(Pg)}
\ge \fc{c}{2}\Var_\mu(g)-\fc{c}2\de \osc(g)^2. 
    \end{align}
where \eqref{e:I2-P2} follows from expanding $g$ in the eigenbasis of $P$ and noting that $\Id+P$ has all eigenvalues in $[0,2]$. Rearranging gives the result.
\end{prf}


\subsection{Proof of main theorem}
\label{s:proof}


We can now prove the main theorem using \pref{c:wpi-gd-ps}, which upgrades a WPI for the proximal sampler to a WPI for the target distribution, given WPIs for the localized distributions.

\begin{prf}[Proof of \pref{t:main}]
By \pref{a:localized-fi}, with probability $1-\de'(n)$ over the draw of $y_T\sim p_{T}$, $\mu_{y_T}$ satisfies a $(c_T\loc,\de(n))$-weak Poincar\'e inequality.
    By \pref{l:prox-wpi}, $\Kprox{T_0,\iy}$ satisfies a 
    $\pa{c\prox_{T_0}(\ep), \ep}$-WPI, where $c\prox_{T_0}(\ep) = \exp\pa{
-O\pa{\fc{E_{T_1}+1}{\ep(T_1-T_0)\hat c_{T_1}} }
}$.
By \pref{l:LSI-time-t}, $\hat c_{T_1} = e^{-O(LT_1)}$, so 
\[
c\prox_{T_0}(\ep) = 
\exp\pa{-O\pa{\fc{(E_{T_1}+1) e^{O(L T_1)}}{\ep(T_1-T_0)} }
} = \exp\pa{-O\prc{\ep}}
\]
keeping just the dependence on $\ep$. 
By \pref{c:wpi-gd-ps}, $\mu$ satisfies a 
\[\pa{\Om(c_{T_0}\loc c_{T_0}\prox(\ep)), O(c\prox_{T_0}(\ep))(\de(n)+\de'(n)) + \ep}\textup{-WPI.}\qedhere\]
\end{prf}

\section{Weak Poincar\'e inequality for the localized distribution on the SK model}
\label{s:wpi-local}

We prove a weak Poincar\'e inequality for the localized distribution for the SK model after large enough constant time. In the following, we let $\mu_{A,y}$ denote the Ising model 
\[
\mu_{A,y}(\si)\propto e^{\rc 2 \an{\si, A\si} + \an{y,\si}},\quad \si\in \{\pm 1\}^n.
\]
\begin{theorem}[WPI for localized distribution for SK]
\label{t:wpi-localized}
\label{t:local-wpi}
Fix $\ep, C>0$. 
    Consider the stochastic localization process for the Gibbs distribution of the SK model with inverse temperature $\be<\rc 2$. 
    There exist constants $a(\be)$, $T(\beta)$, and $c(\be)$, such that 
    with probability $\ge 1-e^{-a(\be)n}$ over the SK model, for $T\ge T(\beta)$, with probability $\ge 1-e^{-Cn}$ over $y_T$, 
    $\mu_{\be A, y_T}$ satisfies a $\pa{\fc{c(\be)}{n},  e^{-C n}}$-WPI.
\end{theorem}

\begin{theorem}[Concentration for localized distribution, {\cite[Lemma 7.43]{DLSS26}}]
\label{t:loc-conc}
Given a measure $\mu$ on $\{\pm 1\}^n$, consider the SL process for $\mu$, \eqref{e:SL}.
    Given $\ep,C>0$ with $\ep < 1/2$, 
there exists $T_0$ such that for any fixed $T\ge T_0$, with probability $\ge 1-e^{-C n}$,
    $\mu_{y_T}(\ball{\ep n}{\sign(y_T)}) \ge 1-e^{-C n}$.
\end{theorem}
\begin{prf}
    The probability given in \cite[Lemma 7.43]{DLSS26} is $1-2e^{-\fc n2D_{\KL}(\ep\|\Phi(-\sqrt T))}$. Choosing $T$ large enough makes this $\ge 1-e^{-Cn}$.
\end{prf}

We use a combined form of the EKZ needle decomposition. 

\begin{theorem}[Needle decomposition for Ising models, combined form of
{\cite[Theorem 7.24]{DLSS26}}, {\cite[Theorem 12]{eldan2022spectral}},
and {\cite[Theorem 29]{anari2022entropic}}]
\label{t:needle}
    Given a test function $\ph:\{\pm 1\}^n\to \R^m$ and an Ising model $\mu_{A,h}$ where $A$ is PSD, there exists a measure $\nu_{\ph}$ on $\R^{2n}$ such that 
    \[
\mu_{A,h} = \int_{\R^{2n}} \mu_{uu^{\top} , v}
d\nu_\ph (u,v)
    \]
    where $\nu_\ph$-a.s., 
    $uu^{\top} \preceq A$, and $\E_{\mu_{A,h}}\ph = \E_{\mu_{uu^{\top} , v}} \ph$. 
    Moreover, for $x,y$ adjacent vertices of $\{\pm 1\}^n$, 
    \[
        \E_{\nu_\ph} \ba{\fc{\mu_{uu^\top, v}(x)\mu_{uu^\top,v}(y)}{\mu_{uu^\top, v}(x)+\mu_{uu^\top,v}(y)}} 
    \le \fc{\mu_{A, h}(x)\mu_{A, h}(y)}{\mu_{A, h}(x)+\mu_{A, h}(y)}\,.
    \]
    Therefore, 
    \begin{align}
    \label{e:needle-dir-form}
\E_{\nu_\ph}[\sE_{\mu_{uu^\top, v}}(f,f)] \le \sE_{\mu_{A,h}}(f,f).
    \end{align}
\end{theorem}
\begin{prf}
The decomposition, expectation preservation, and Loewner domination are
\cite[Theorem 7.24]{DLSS26}. The conductance comparison is the supermartingale
property from the EKZ localization proof, stated for the full hypercube in
\cite[Theorem 12]{eldan2022spectral} and \cite[Theorem 29]{anari2022entropic}.
The displayed Dirichlet-form comparison follows by summing the conductance
comparison over hypercube edges.
\end{prf}

\begin{theorem}[Hubbard--Stratonovich transform, \cite{hubbard1959calculation,koehler2022sampling}]
\label{t:hs}
Let $A=XX^\top$ with $X\in \R^{m\times n}$, and let $\Omega$ be any finite subset of $\set{v\in \R^n}{\ve{v}=R}$. 
Then we have the following decomposition of $\mu_A$ into ``product measures'' restricted to $\Om$, where $\lm_{H}$ denotes Lebesgue measure on the hyperplane $H$:
    \begin{align*}
\mu_{A,h}|_{\Omega}(\si) 
&=\int_{\im(X)} p(u) 
\mu_{X^{\top} u + h}|_{\Om}(\si) 
d\lambda_{\im(X)}(u) \\
\text{where }p(u) &\propto \exp\pa{-\fc 12 \ve{u}^2}
\sum_{\si\in \Omega} \exp\pa{\an{
X^\top u+h, \si}
}.
\end{align*}
\end{theorem}
The following is a consequence of \pref{l:wpi-convergence}.
\begin{theorem}[{Weak Poincar\'e inequality from covariance bound under localization, consequence of {\cite[Lemmas 6.8 and A.6]{huang2025weak}}}]
\label{t:sl-wpi}
    Let $\nu$ be a measure on $\{\pm 1\}^n$ 
    and $J$ be a PSD $n\times n$ matrix. 
    Consider stochastic localization with $\mu_0=\nu$ and driving matrix $C_t=J^{1/2}$, so that 
    \[
\dd{\mu_{t}}{\nu}(x) \propto \exp\pa{-\fc{t}2 \an{x,Jx}+\an{v_t,x}}
    \]
    Suppose that the following hold.
    \begin{enumerate}
    \item With probability $\ge 1-\eta_1$, for all $t\in [0,1]$, 
    \[
\opnorm{\Cov(\mu_{t})}\le \al. 
    \]
    \item 
    With probability $\ge 1-\eta_2$, $\mu_1$ satisfies a $(\rp,\de)$-WPI.
    \end{enumerate}
    Then $\mu$ satisfies a $\pa{\rp e^{-\opnorm{J}^{1/2}\al}, e^{\opnorm{J}^{1/2}\al}(\de + \eta_1 + \eta_2)}$-WPI.  
\end{theorem}
Note that while \cite[Lemma 6.8]{huang2025weak} considers measures on $\sqrt n\cdot \bbS^{n-1}$, the exact statement follows for $\{\pm 1\}^n$ by their Lemma A.6 (\pref{l:wpi-convergence}). 
\begin{lemma}[Covariance bound for rank-1 Ising models on wedges]
\label{l:cov-rank-1}
    Let $A\sim \GOE(n)$ and $\be<\rc2$. There are constants $\ep(\be)$, $c(\be)$, $C(\be)$ such that with probability $\ge 1-e^{-c(\be)n}$, for any $x$ such that $xx^{\top}\preceq \be A + (\be+\rc 2)I$ and any $x_0\in \{\pm 1\}^n$, 
    \[
\Cov(\mu_{xx^\top, v}|_{\ball{\ep(\be)n}{x_0}}) \preceq C(\be) \cdot I_n.
    \]
\end{lemma}
This follows from the proof of \cite[Lemma 7.42]{DLSS26} using \Cref{t:hs}.
\begin{prf}
First note that by \cite[Corollary 7.3.2]{vershynin2018high}, 
\begin{align}\label{e:opnorm-A}
\P_{A\sim \GOE(n)}\ba{\opnorm{\be A + \pa{\be + \rc2}I_n}\le 2} \ge 1-e^{-\Om(n)}.
\end{align}
Under this event, $\ve{x}\le \sqrt2$.
Also with probability $1-e^{-\Om(n)}$, by \cite[Lemma 7.23]{DLSS26}, for some constant $c_1$, for $|S|\le 4\ep n$,
\begin{align}
\label{e:opnorm-A-submats}
\opnorm{(\be A+\ga I)_{S\times S}} \le \ga + \be c_1\sqrt{h(4\ep)} <1
\end{align}
where $h(\ep) = -\ep \log \ep - (1-\ep)\log (1-\ep)$, 
when $\ep$ is small enough.
    Suppose these events hold. Let $\Om = B_{\ep(\be)n}(x_0)$. The Hubbard--Stratonovich transform (\pref{t:hs}) gives
    \begin{align*}
        \mu_{xx^\top,v}|_{\Omega}(\si) 
&=\int_{\R} p(u) 
\mu_{ux + v}|_{\Om}(\si) 
du \\
\text{where }p(u) &\propto \exp\pa{-\fc 12 \ve{u}^2}
\sum_{\si\in \Omega} \exp\pa{\an{
xu+h, \si}
}.
    \end{align*}
By the proof of \cite[Lemma 7.42]{DLSS26}, using \eqref{e:opnorm-A} and \eqref{e:opnorm-A-submats}, $\Var(p)$ is bounded by a constant $V(\be)$. Note 
\[\ddd u \E_{\mu_{ux+v}}\si = \Var_{\mu_{ux+v}}\pa{\an{\si,x}}\le 2\ve{x}^2
\]
by \cite[Lemma 7.21]{DLSS26}, so $\E_{\mu_{ux+v}}\si$ is $2\ve{x}^2$-Lipschitz.
Therefore, $\E_{\mu_{ux+v}}\si$ has covariance at most $2\ve{x}^2V(\be)$. 
By covariance decomposition,
\begin{align}
\nonumber
    \Cov_{\mu_{xx^{\sT} ,v}|_{\Om}}(\si) &= 
    \Cov_p(\E_{\mu_{ux + v}|_{\Om}}\si)
    + 
    \E_p \Cov_{\mu_{ux + v}|_{\Om}}(\si)\\
    &\preceq \pa{2\ve{x}^2V(\be) + 2}I_n,
\nonumber
\end{align}
where the second covariance bound again follows by \cite[Lemma 7.21]{DLSS26}.
\end{prf}

\begin{prf}[Proof of \Cref{t:wpi-localized}]
    Consider the needle decomposition (\Cref{t:needle}) applied to $\mu_{\be A + (\be + \rc 2)I, y_T}$, $
\mu_{A,h} = \int_{\R^{2n}} \mu_{uu^{\top} , v}
d\nu_f (u,v)$.
Let $\Om_0 = \ball{\ep(\be) n}{\sign(y_T)}$. 
Given $C>0$, by \pref{t:loc-conc}, we can choose $T(\be)$ such that for $T\ge T(\be)$, with probability $\ge 1-e^{-Cn}$ over $y_T$, we have  $\mu_{\be A,y_T}(\Om_0) \ge 1-e^{-C n}$. 
    By Markov's inequality,
    with probability $\ge 1-e^{-\fc{Cn}2}$ over $(u,v)\sim \nu_f$, 
    \[
\mu_{uu^\top, v}(\Om_0) \ge 1-e^{-\fc{Cn}2}.
    \]
    Let $S$ be the subset of $(u,v)$ for which this holds. Then
    \begin{align} \label{e:var-needle-whp}
\Var_{\mu_{A,h}}(f) \le \int_S \Var_{\mu_{uu^\top, v}}(f)  \,d\nu_f(u,v) + e^{-\fc{Cn}2} \cdot \osc(f)^2.
    \end{align}
    Consider stochastic localization for $\mu_{uu^\top, v}$ with driving matrix $(uu^\top)^{1/2}$. 
    Let $\tau$ be a stopping time for the stochastic localization process defined by 
    \[
\tau = \inf\set{t}{(\mu_{uu^\top, v})_t(\Om_0^c) \ge e^{-\fc{Cn}4}}.
    \]
    By the martingale property of stochastic localization, by the optional stopping theorem,
    \[
e^{-\fc{Cn}2}
\ge
\mu_{uu^\top, v}(\Om_0^c)
= 
\E 
\ba{(\mu_{uu^\top, v})_{1\wedge \tau}(\Om_0^c)}
\ge e^{-\fc{C}4 n} \P[\tau \le 1].
    \]
    Therefore,
    $\P[\tau\le 1]\le e^{-\fc{Cn}{4}}$,
or equivalently,
\[
\P\ba{\forall t\in [0,1],\,
(\mu_{uu^\top, v})_t(\Om_0^c)<e^{-\fc{Cn}{4}}}
=
\P[\tau>1]
\ge
1-e^{-\fc{Cn}{4}}.
\]
    Under this event, we have by decomposition of covariance that
    \begin{align*}
\Cov((\mu_{uu^\top, v})_t) &= 
(\mu_{uu^\top, v})_t(\Om_0) \Cov((\mu_{uu^\top, v})_t|_{\Om_0}) + 
(\mu_{uu^\top, v})_t(\Om_0^c) \Cov((\mu_{uu^\top, v})_t|_{\Om_0^c}) + 
\Cov(\E[\si|\one_{\si\in \Om_0}]) \\
&\preceq  \Cov((\mu_{uu^\top, v})_t|_{\Om_0}) + 2
(\mu_{uu^\top, v})_t(\Om_0^c)\cdot 
n I_n\\
&\preceq  \ba{C + 2e^{-\fc{Cn}{4}} n} I_n
\preceq C_2(\be) I_n
    \end{align*}
by \Cref{l:cov-rank-1}. 
    Now, $(\mu_{uu^\top, v})_1$ is a product distribution and so has Poincar\'e constant $\rc n$ (with probability 1). 
    Noting \eqref{e:opnorm-A}, by \Cref{t:sl-wpi}, for $(u,v)\in S$, $\mu_{uu^\top ,v}$ satisfies a WPI with parameters
    \[
\pa{c_1:=\rc n e^{-2C_2(\be)} ,\de_1:= e^{2C_2(\be)} e^{-\fc{Cn}4}}.
    \]
    Continuing from \eqref{e:var-needle-whp}, we have that
    \begin{align*}
    \Var_{\mu_{A,h}}(f)
    &\le \int_S \ba{\rc{c_1} \sE_{\mu_{uu^\top, v}}(f,f)
    + \de_1 \osc(f)^2}
    \,d\nu_f(u,v) + e^{-\fc{Cn}2} \cdot \osc(f)^2\\
    &\le_{\eqref{e:var-needle-whp}}
    \rc{c_1} \sE_{\mu_{A,h}}(f,f) + \pa{\de_1 + e^{-\fc{Cn}2}} \osc(f)^2.
    \end{align*}
    The theorem follows, after adjusting constants as needed.
\end{prf}

We can now prove a WPI for the SK model.
\begin{prf}[Proof of \pref{c:sk}]
Fix $\be<1/2$ and $\de>0$. We work on an event of probability at
least $1-\de$, after increasing $n$ if necessary, on which the following
hold: $\opnorm A\le 3$; \pref{a:m-error} and
\pref{a:Lipschitz-drift} hold by \cite{DLSS26}; and, by
\pref{t:wpi-localized} with $T_0$ large enough, \pref{a:localized-fi}
holds with
\[
    c_{T_0}\loc=\frac{c(\be)}{n},
    \qquad
    \de(n)+\de'(n)\le e^{-c n},
\]
where all constants may depend on $\be$ and $\de$.

By \pref{t:main}, there are constants $C_1,C_2<\infty$ such that, for
every $\eta\in(0,1)$, $\mu_{\be A}$ satisfies a
\[
    \left(
        \frac1n e^{-C_1/\eta},
        e^{C_2/\eta}e^{-c n}+\eta
    \right)
\]
weak Poincar\'e inequality, where we have absorbed the constant
$c(\be)$ into $C_1$. On the other hand, on the event $\opnorm A\le 3$,
Holley--Stroock gives a Poincar\'e inequality, hence a weak Poincar\'e
inequality, with parameters
\(
    \left(\frac1n e^{-C_3 n},0\right)
\).

Choose a constant $K$ sufficiently large compared to $C_2/c$, and then
choose the hidden constant $C$ in the desired bound sufficiently large
compared to $C_1$ and $C_3K$. We claim that, for all sufficiently large
$n$, this gives the desired WPI for every $\ep\in(0,1)$.

If $\ep\ge K/n$, apply \pref{t:main} with $\eta=\ep/2$. By the choice
of $K$,
\(e^{2C_2/\ep}e^{-c n}\le \ep/2\) for all sufficiently large $n$, so $\mu_{\be A}$ satisfies an \(\left(\frac1n e^{-C/\ep},\ep\right)\)-weak Poincar\'e inequality.

If instead $\ep<K/n$, the Holley--Stroock bound gives the same
conclusion, since the choice of $C$ ensures \(\frac1n e^{-C_3 n}
    \ge
    \frac1n e^{-C/\ep}\).
Thus, on the same event of probability at least $1-\de$, for every
$\ep\in(0,1)$, $\mu_{\be A}$ satisfies a
\(\left(\frac1n e^{-O_{\be,\de}(1/\ep)},\ep\right)\)-weak Poincar\'e inequality.
\end{prf}

\section{Log-Sobolev inequality for the annealed distribution \texorpdfstring{$\rh_t$}{rho}}
\label{s:rho-t-lsi}

\Cref{l:LSI-time-t} shows that the $\hat p_t$ satisfy a log-Sobolev inequality. 
In this section, we further assume that there exists a sequence of distributions $\rh_t$ that approximate the ASL distribution $\hat p_t$, in the sense that their Radon-Nikodym derivative to the ASL distribution can be controlled, and we show that the $\rh_t$ also satisfy a log-Sobolev inequality. 
The relevance is when the $\rh_t$ has an explicit density, this provides an explicit distribution approximating the SL distribution $p_t$ which satisfies a LSI, so that we can sample from it directly. We call $\rh_t$ the annealed distributions. In \pref{s:ws}, we will quantify how close $\rh_t$ is to $p_t$, in the sense of being a warm start.

\begin{assumption}[Lipschitz drift for Jarzynski weights]
\label{a:Lipschitz-je}\label{ass:Lipschitz-je}
    With $y_t$ denoting the ASL tilt\footnote{Previously referred to as $\hat{y}_t$ in \pref{e:ASL}.}, assume that 
    \begin{align}
\savetagequation{e:JE}{JE}{dw_t &= \om(y_t) \, dt}
& y_0&\sim \hat p_0
    \end{align}
    where $\om$ is $L_\om$-Lipschitz, $\hat p_0$ satisfies a log-Sobolev inequality with constant $c_0$, and $\rh_t$ is a distribution such that
    \begin{align}
    \label{e:je-weights}
\dd{\rh_t}{\hat p_t}(x) \propto \E[e^{w_t} | y_t = x].
    \end{align}
\end{assumption}
We will be interested in the following particular case. Suppose that 
\begin{align}
\label{e:rh-t-F}
\rh_t(y)\propto e^{\calF(y) - \fc{\ve{y}^2}{2t}}\end{align}
where $\gd \calF(y) = \hat m(y)$; here $\rh_0$ is a point mass at 0. Applying Jarzynski's equality with $-U_t(y) = \calF(y) - \fc{\ve{y}^2}{2t}$ and $b_t(y) = \rc2 \pa{\hat m(y) + \fc{y}{t}}$ gives \eqref{e:JE} for 
\begin{align}
\label{e:om}
\om(y) = \rc 2 \ba{\gd \cdot \hat m(y) + \ve{\hat m(y)}^2} 
\end{align}
up to a constant not depending on $y$, which does not affect the result.\footnote{Technically, $\rh_t$ approaches a point mass as $t\to 0^+$, so Jarzynski's equality needs to be justified with a limiting argument; see \cite[\S3.2]{DLSS26}.} In particular, for the SK model, we take $\calF(y) = \FT(\hat\mg(y), y)$; see \cite[\S3.2]{DLSS26} for the calculations.

We will be able to use a LSI on the ``tilted'' path-space measure $\nu$, and then transport it along these rough paths to $\rh_T$ (which can be seen as a Lipschitz push-forward of $\nu$ in the path space).

First, we show that the $w_t$ appearing in the change-of-measure is Lipschitz in the Cameron-Martin norm in path space. Simple Gr\"onwall estimates with the It\^o integral yield the requisite Cameron-Martin Lipschitz constants. The result then follows from a path-space perturbation theorem (\pref{t:lsi-CM-pert}), which says that multiplicative perturbations to the Wiener measure that are log-Lipschitz in the Cameron-Martin sense cause the reweighed measure to also satisfy a LSI. The proof for finite-dimensional Gaussian space (\pref{lem:lipschitz-lsi-measure-space}) is elementary, but passage to the infinite-dimensional setting requires much more analytic control of the corresponding cylindrical approximations and their limiting behavior.

We will prove the following main theorem in \pref{s:rho-t-lsi-proof} after establishing a suite of analytical tools for regularity and LSIs in Wiener space.

\begin{theorem}[LSI for annealed distribution]
\label{t:rho-ls}
Assume \pref{a:Lipschitz-drift} and \pref{a:Lipschitz-je} in the special case that $\hat{p}_0 = \delta_0$.
Then $\rh_T$ satisfies a log-Sobolev inequality with constant $c_T^\rh$ depending only on the constants in the assumptions.
\end{theorem}
We obtain a log-Sobolev inequality of the annealing distribution for SK as an immediate corollary.
\begin{corollary}[LSI for annealed distribution for SK]
\label{c:lsi-anneal-sk}
    For the SK model for $\be<\rc2$, with probability $1-\de$, $\rh_T(y)\propto e^{\FT(\hat\mg(y), y) - \fc{\ve{y}^2}{2T}}$ satisfies a log-Sobolev inequality with constant depending only on $T$ and $\de$.
\end{corollary}
\begin{prf}
    For the SK model for $\be<\rc2$, with probability $1-\de$, \pref{a:Lipschitz-drift} holds, and \pref{a:Lipschitz-je} holds in the special case that $\hat{p}_0 = \delta_0$ by the proof of Corollary 6.14 in \cite{DLSS26}.
    More precisely, the proof of Corollary 6.14(d5) in \cite{DLSS26}
shows that, on the same high-probability event, the Jarzynski drift
\[
\omega(y)=\frac12\left(\nabla\cdot \hat m(y)+\|\hat m(y)\|^2\right)
\]
is \(O_{T,\delta}(1)\)-Lipschitz in \(y\). Thus the constant \(L_\omega\)
in \pref{a:Lipschitz-je} is dimension-free, and the constant in
\pref{t:rho-ls} depends only on \(T\) and \(\delta\).
\end{prf}

\subsection{Two Cameron-Martin Lipschitz estimates} We prove two key Lipschitz estimates. The first estimate shows that the Cameron--Martin norm for the ASL It\^o map is constant at finite constant time $T(\beta)$. 

\begin{lemma}[ASL It\^o map is Cameron-Martin Lipschitz]\label{lem:cm-lipscitz-one}
Fix $T>0$. 
Assume that $\{\hat{y}_t\}_{0 \le t \le T}$ is the strong solution to the ASL SDE $d\hat{y}_t = \hat{m}(\hat{y}_t)dt + dB_t$ with $\hat{y}_0 = 0$. Then, under~\pref{a:Lipschitz-drift}, for $0\le t\le T$, the map $\Phi: E \to E$ given as $\Phi_t(B) =\hat{y}_t$ satisfies the following Lipschitz bound in the Cameron-Martin norm,
    \[
        \lipnorm{\Phi_t}\le \sqrt{t}e^{Lt}\, . 
    \]
    Considering the dependence on the initial condition for $\hat y_0$, 
     with respect to the norm on $\R^n\times E$ 
     \begin{align}
     \label{e:norm-y-h}
     \ve{(y,h)} = \sqrt{c\ve{y}^2 + \ve{h}\CM^2},\end{align} 
    $\Phi_t: \R^n \times E \to E$ satisfies the following Lipschitz bound,
    \[
\lipnorm{\Phi_t} \le \sqrt{t+c^{-1}}\cdot e^{Lt}\, .
    \]
\end{lemma}
\begin{prf}
    Note by the definition of the It\^o integral that
    \[
        \hat{y}_t = B_t + \int_0^t \hat{m}(\hat{y}_s)\,ds\,.
    \]
    For any path $h \in \calH$, let
    \[
        \delta_h\Phi_t(B) := \Phi_t(B + h)\, ,
    \]
    and then observe that
    \allowdisplaybreaks
    \begin{align*}
        \norm{\delta_h\Phi_t(B) - \Phi_t(B)}_2 &= \norm{h_t + \int_0^t \left(\hat{m}(\delta_h
        \Phi_s(B))-\hat{m}(
        \Phi_s(B))\right)\,ds}_2 \\
        &\le \norm{h_t}_2 + \int_0^t \norm{\left(\hat{m}(\delta_h
        \Phi_s(B))-\hat{m}(
        \Phi_s(B))\right)}_2\,ds \\
        &\le_{\text{\pref{a:Lipschitz-drift}}} \norm{h_t}_2 + L\int_0^t \norm{\delta_h
        \Phi_s(B)- 
        \Phi_s(B)}_2\,ds \\
        &\le_{\text{Gr\"onwall}} e^{Lt}\sup_{0 \le s \le t}\norm{h_s}_2\,.
    \end{align*}
    Note that 
    \[
        \norm{h_s}_2 = \norm{\int_0^s \ddd u h_u \,du}_2 \le \sqrt{s} \sqrt{\int_0^s\norm{\ddd u h_u}_2^2du} =  \sqrt{s}\norm{h}_{\text{CM}}\,,
    \]
    and so 
    \[
         \norm{\Phi_{t}(B+h) - \Phi_{t}(B)}_2 \le \sqrt{t}e^{Lt}\norm{h}\CM\,.
    \]
    Including the initial condition, we have
    \[
\norm{\delta_h\Phi_t(B, y_0') - \Phi_t(B,y_0)}_2
\le e^{Lt} \pa{\norm{y_0'-y_0} + \sqrt{t}\norm{h}\CM }.
    \]
    Hence 
    \[
\fc{\norm{\delta_h\Phi_t(B, y_0') - \Phi_t(B,y_0)}_2}{\sqrt{c\norm{y_0'-y_0}^2 + \norm{h}\CM^2}} \le \sqrt{t+c^{-1}} \cdot e^{Lt}.\qedhere
    \]
\end{prf}

Having established the estimate above, we show the second key estimate which implies that the (as functions of $\hat{y}_t$) Lipschitz Jarzynski weights $w_t$ are also Lipschitz in the Cameron-Martin space when viewed as push-forwards of $\Phi$.

\begin{lemma}[Jarzynski functional is Cameron-Martin Lipschitz]\label{lem:cm-lipschitz-two}
    Assume \pref{a:Lipschitz-drift}.
    Let $W: E \to \R$ be defined as
    \[
        W(B) := \int_0^T \omega(\hat{y}_t)\,dt = \int_0^T \omega(\Phi_t(B))\,dt\,,
    \]
    and assume $\om:\R^n\to \R$ is $L_\om$-Lipschitz. 
    Then $W$ satisfies the following Lipschitz bound in the Cameron-Martin norm:
    \[
        \lipnorm{W} \le L_\omega T^{3/2}e^{LT}\,.
    \]
    Including the initial conditions, letting $W:\R^n\times E\to \R$, $W$ satisfies the following Lipschitz bound with respect to the norm \eqref{e:norm-y-h}, 
    \[
        \lipnorm{W} \le L_\omega T \sqrt{T+c^{-1}} \cdot e^{LT}\,.
    \]
\end{lemma}
\begin{prf}
    For any $h \in \calH$, we again have that
    \allowdisplaybreaks
    \begin{align*}
        \left|W(B+h) - W(B)\right|&= \left|\int_0^T\left(\omega(\Phi_t(B+h)) - \omega(\Phi_t(B))\right)\,dt\right| \\
        &\le \int_0^T \left|\omega(\Phi_t(B+h)) - \omega(\Phi_t(B))\right|\,dt \\
        &\le_{\om \text{ Lipschitz}}
        L_\omega \int_0^T \norm{\Phi_t(B+h) - \Phi_t(B)}_2\,dt \\
        &\le_{\text{\pref{lem:cm-lipscitz-one}}} L_\omega \sqrt{T} e^{LT} \norm{h}_{\text{CM}}\int_0^T \,dt \le L_\omega T^{3/2}e^{LT} \norm{h}_{\text{CM}}\,.
    \end{align*}
    The result including the initial conditions follows similarly from the corresponding part of \pref{lem:cm-lipscitz-one}.
\end{prf}

\subsection{Path-space perturbation via cylindrical approximation}  We will use the following elementary fact about log-Lipschitz perturbations to the multivariate Gaussian measure $\gamma \sim \calN(0,\Id_n)$ preserving the LSI constant (up to Lipschitz-dependent factors).

\begin{lemma}[Finite-dimensional log-Lipschitz perturbation to Gaussians]\label{lem:lipschitz-lsi-measure-space}
    Let $\gamma \sim \calN(0,\Id_n)$ and $V: \R^n \to \R$ be $K$-Lipschitz in $\norm{\cdot}_2$-norm for some $K > 0$. Then, for $d\nu := \frac{e^V}{Z}d\gamma$, the probability measure $\nu$ satisfies a LSI with constant $c_\nu \ge \fc{e^{-K^2}}2$.
\end{lemma}
\begin{prf}
We first show the following claim: \emph{If $f$ is differentiable and
$a$-strongly convex, $g$ is $L$-Lipschitz, and $h$ is the convex envelope of
$f+g$, then
\[
0 \le f(x)+g(x)-h(x)\le \frac{L^2}{2a}.
\]
}
The lower bound follows from $h\le f+g$. For the upper bound, by translating
we may assume $x=0$. By the convex-envelope formula,
\[
h(0)=\inf_{\substack{w_i\ge 0,\ \sum_i w_i=1\\ \sum_i w_i x_i=0}}
\sum_i w_i [f(x_i)+g(x_i)] .
\]
For any such convex combination, strong convexity of $f$ and Lipschitzness of
$g$ give
\begin{align*}
\sum_i w_i [f(x_i)+g(x_i)]
&\ge
\sum_i w_i\left[
f(0)+\langle \nabla f(0),x_i\rangle+\frac a2\|x_i\|^2
+g(0)-L\|x_i\|
\right] \\
&=
f(0)+g(0)+
\sum_i w_i\left[\frac a2\|x_i\|^2-L\|x_i\|\right] \\
&\ge
f(0)+g(0)-\frac{L^2}{2a}.
\end{align*}
Taking the infimum over such convex combinations gives
\[
h(0)\ge f(0)+g(0)-\frac{L^2}{2a},
\]
and therefore
\[
f(0)+g(0)-h(0)\le \frac{L^2}{2a}.
\]
This proves the claim.
    
    Let $a_1+a_2=1$, $a_1,a_2>0$ and write the density of $\nu$ as 
    \[
\nu(x) \propto e^{-\fc{a_1}2\ve{x}^2 + V(x)}e^{- \fc{a_2}2\ve{x}^2 }.
    \]
    Let $h$ be the convex envelope of $\fc{a_1}{2}\ve{x}^2 - V(x)$. Consider
    \[
\td \nu(x) \propto e^{-h(x)}e^{- \fc{a_2}2\ve{x}^2 }
    \]
    The lemma gives that $\osc[\fc{a_1}2\ve{x}^2 - V(x) - h(x)] \le \sup_x  [\fc{a_1}2\ve{x}^2 - V(x) - h(x)] \le \fc{K^2}{2a_1}$. 
    Since $\td\nu$ is $a_2$-strongly convex, it satisfies a log-Sobolev inequality with constant $a_2$. 
    Thus by the Holley-Stroock perturbation lemma \cite{holley1987logarithmic}, $\nu$ satisfies a log-Sobolev inequality with constant $a_2e^{-\fc{K^2}{2a_1}}$. Taking $a_1=a_2=\rc 2$ gives the result.
 \end{prf}
With this finite-dimensional perturbation LSI in hand, we can now prove that Lipschitz perturbations to the Wiener measure in the Cameron-Martin space preserve LSIs up to dimension-free factors.

\begin{theorem}[LSI under Cameron-Martin Lipschitz perturbations in path-space]
\label{t:lsi-CM-pert}
    Let $\calH := H^1_0\left([0,T(\beta)], \R^n\right)$ be the Cameron-Martin space for the Wiener measure $\gamma$ on $C_0([0,T],\R^n)$. Assume that $V: E \to \R$ is $\gamma$-measurable, finite and $K$-Lipschitz in the $\norm{\cdot}\CM$-norm. Then, for $d\nu(B) := \frac{e^{V(B)}}{Z}d\gamma(B)$ for $B \in E$, the measure $\nu$ satisfies
    \[
        \Ent_\nu\left[f^2\right] \le \frac{4}{e^{-CK^2}} \int_E \norm{D_Hf}^2\CM d\nu\, ,
    \]
    for some absolute constant $C >0$ and any $f \in \calD^{1,2}(\nu)$.
\end{theorem}
\begin{prf}
    The proof follows by a limiting argument.
    First, we take the cylindrical functions $V_k$ obtained by revealing $k$ Gaussians and averaging $V$ over the ``tail''.
    We then decompose the path-space perturbed measure $\nu_k$ on this revelation into a finite-dimensional Gaussian reweighed by $V_k$ tensorized with an infinite-dimensional Gaussian over path-space $E$.
    Combining a LSI for the former via~\pref{lem:lipschitz-lsi-measure-space} and a LSI for the latter due to a result of Gross (\pref{t:gross-lsi}) via entropy tensorization yields uniform path-space LSIs for $\{\nu_k\}_k$.
    A $L^1$ limiting argument using the fact that $V_k$ is a martingale in conjunction with uniform integrability allows the LSIs to pass to the limit $\nu$.
    The final step is a closure argument to boost the LSI for $\nu$ to hold for $f \in \calD^{1,2}(\nu)$.

    \ppart{Cylindrical approximation for $V$} Fix $B \in E$ and define a sequence of i.i.d. $\calN(0,1)$ random variables using an orthonormal basis $\{e_j\}_{j \ge 1}$ for $\calH$ as
    \[
        X_j(B) := \int_0^T \an{\frac{d}{dt}e_j, dB_t}\,.
    \]
    Denote the $\sigma$-algebra generated by the first $k$ elements in the basis as $\calF_k := \sigma\left(X_1,\dots,X_k\right)$. Then, the conditional expectation of $V$ under the finite-dimensional shadow $\calF_k$ is $V_k := \E_\gamma\left[V \mid \calF_k\right] = \widehat V_k\left(X_1(\cdot),\dots,X_k(\cdot)\right)$ for some Borel function $\widehat V_k:\R^k\to\R$ by the Doob-Dynkin lemma.

Define
\[
    \pi_k(B):=(X_1(B),\ldots,X_k(B)),\qquad
    R_k(B):=B-\sum_{i=1}^k X_i(B)e_i .
\]
    
    We show that $\widehat V_k$ is $K$-Lipschitz in the $\norm{\cdot}_2$-norm.
    For $a\in \R^k$, write $h_a:=\sum_{i=1}^k a_i e_i$. By the
Gaussian product decomposition in the coordinates $(\pi_k,R_k)$,

\[\widehat V_k(a)=\int_E V(h_a+r)\,d\gamma_{>k}(r),\]
where $\gamma_{>k}:=(R_k)_\#\gamma$. Hence for all $a,b\in\R^k$,
\begin{align*}
    \left|\widehat V_k(a)-\widehat V_k(b)\right|
    &= \left|\int_E \left[V(h_a+r)-V(h_b+r)\right]\,d\gamma_{>k}(r)\right| \\
    &\le \int_E \left|V(h_a+r)-V(h_b+r)\right|\,d\gamma_{>k}(r)\\
    &\le K\|h_a-h_b\|_{\calH}
     = K\|a-b\|_2,
 \end{align*}
where the last equality uses the orthonormality of $\{e_i\}_i$ in $\calH$.

    \ppart{Path-space LSI for Gaussian measures tilted via $V_k$} Define the following sequence of tilted measures, for any $B \in E$,
    \[
        d\nu_k(B) := \frac{e^{V_k(B)}}{Z_k}d\gamma(B)\,,
    \]
    where $\gamma$ is the Wiener measure on $(E,\calB(E))$.
    
    We next identify the finite-dimensional weak gradient of $\widehat V_k$ with the $H_k$-component of the Malliavin derivative of $V_k$. First note that for every $B \in E$, $\eps > 0$ and $h \in \calH$,
    \allowdisplaybreaks
    \begin{align*}
        X_i(B + \eps h) &= \int_0^T \an{\frac{d}{ds}e_i, d\left(B + \eps h\right)_s} = \int_0^T \an{\frac{d}{ds}e_i, dB_s} + \eps\int_0^T \an{\frac{d}{ds}e_i, \frac{d}{ds}h}\,ds \\
         &= X_i(B) + \eps \an{e_i, h}_H\,.
     \end{align*}
    Since $\calH$ is a separable Hilbert space, any $h \in \calH$ is given as $h = \sum_i h_i e_i$ with $h_i = \an{e_i,h}_H$.
    At points where $\widehat V_k$ is differentiable ($\gamma_k$-a.e. by Rademacher's theorem), the definition of the Malliavin derivative gives
    \[
        D_h V_k(B) = \sum_{i=1}^k \partial_i \widehat V_k(\pi_k(B))\,\an{e_i,h}_H .
     \]
     Hence, by the Riesz representation theorem,
    \[ 
        D_HV_k(B)=\sum_{i=1}^k \partial_i\widehat V_k(\pi_k(B))e_i\,,
    \]
    and $\norm{D_HV_k(B)}_{\calH}^2 = \norm{\nabla \widehat V_k(\pi_k(B))}_2^2$. Furthermore, letting $\gamma_k := (\pi_k)_\#\gamma=\calN(0,\Id_k)$ and $\tilde Z_k:=\int_{\R^k}e^{\widehat V_k(a)}\,d\gamma_k(a)$, \pref{lem:lipschitz-lsi-measure-space} applied to the measure
    \[d\tilde{\nu}_k(a) := \frac{e^{\widehat V_k(a)}}{\tilde Z_k}\,d\gamma_k(a),
    \]
    gives, for smooth test functions and then by the usual finite-dimensional Sobolev closure, for every $\varphi \in C_b^1(\R^k)$,
    \[
        \Ent_{\tilde{\nu}_k}[\varphi^2]
        \le
        \frac{4}{e^{-C\,K^2}}
        \int_{\R^k}\norm{\nabla \varphi(a)}_2^2\,d\tilde{\nu}_k(a)\,. 
     \]
    It is straightforward to see that, under the inclusion $\calH \subset E$, for every bounded and measurable $f: E \to \R$,
    \begin{align}\label{e:gaussian-prod-decomp}
        \E_{\nu_k}[f] = \int_{B \in E} f(B) d\nu_k(B) = \int_{a \in \R^k}\int_{r \in E}f(h_a + r)d\tilde{\nu}_k(a)d\gamma_{> k}(r)\, ,
    \end{align} where $\gamma_{> k} := (R_k)_{\#}\gamma$. Moreover,
    \[
        (\pi_k,R_k)_{\#}\gamma = \gamma_k \otimes \gamma_{> k},
        \qquad
        (\pi_k,R_k)_{\#}\nu_k = \tilde{\nu}_k \otimes \gamma_{> k},
    \]
    with the tensorization proved in~\pref{lem:wiener-lipschitz-reweigh-tail-decomposition}. Consequently,
    \[
        \Ent_{\nu_k}[f^2] = \Ent_{\tilde{\nu}_k \otimes \gamma_{>k}}\left[g^2\right]\, ,
    \]
    where $g(a,r) = f(h_a + r) = f(B)$ for $h_a = \sum_{i=1}^k a_i e_i$ with $a =\pi_k(B) \in \R^k$ and $r = R_k(B)$.  The push-forward $\gamma_{> k}$ is a centered Gaussian process\footnote{A proof of the $E$-Gaussianity of $\gamma_{>k}$ is given in~\pref{lem:gamma-tail-gaussian}.} on $E$ with a Cameron-Martin space $H^\perp_{k}$. Therefore, by tensorization of entropy under product decompositions we have
    \allowdisplaybreaks
    \begin{align*}
        \Ent_{\nu_k}[f^2] &= \Ent_{\tilde{\nu}_k \otimes \gamma_{>k}}\left[g^2\right] = \int_{y \in E}\Ent_{\tilde{\nu}_k}[g^2] d\gamma_{>k}(y) + \Ent_{\gamma_{>k}}\left[\E_{\tilde{\nu}_k}[g^2]\right] 
    \end{align*}
    We now apply the dimension-free LSI for Gaussians to $\gamma_{>k}$ on $E$~(\pref{t:gross-lsi}) and the finite-dimensional LSI for $\tilde{\nu}_k$ achieved in~\pref{lem:lipschitz-lsi-measure-space} to obtain the following LSI subadditivity in path space,
    \begin{align*}
        &\Ent_{\nu_k}[f^2] = \Ent_{\tilde{\nu}_k \otimes \gamma_{>k}}\left[g^2\right] = \int_{y \in E}\Ent_{\tilde{\nu}_k}[g^2] d\gamma_{>k}(y) + \Ent_{\gamma_{>k}}\left[\E_{\tilde{\nu}_k}[g^2]\right] \\
        &\le_{\text{\pref{lem:lipschitz-lsi-measure-space},\,\pref{t:gross-lsi}}} \int_{y \in E}\int_{a \in \R^k}\left(\frac{4}{e^{-CK^2}}\norm{\nabla_a g(a,y)}^2_2+ 2\norm{D_{H^\perp_k}g(a,y)}^2_{\text{CM}}\right)d\tilde{\nu}_k(a)d\gamma_{>k}(y) \\
        &\le \frac{4}{e^{-CK^2}}\int_{y \in E}\int_{a \in \R^k}\left(\norm{\nabla_a g(a,y)}^2_2+\norm{D_{H^\perp_k}g(a,y)}^2_{\text{CM}}\right)d\tilde{\nu}_k(a)d\gamma_{>k}(y) \\
        &=  \frac{4}{e^{-CK^2}}\int_{y \in E}\int_{a \in \R^k}\left(\norm{P_{H_k}D_Hg(a,y)}^2_{\text{CM}} + \norm{D_{H^\perp_k}g(a,y)}^2_{\text{CM}}\right)d\tilde{\nu}_k(a)d\gamma_{>k}(y) \\
        &= \frac{4}{e^{-CK^2}}\int_{y \in E}\int_{a \in \R^k}\norm{D_H f}^2_{\text{CM}} d\tilde{\nu}_k(a)d\gamma_{>k}(y) = \frac{4}{e^{-CK^2}}\int_{x \in E} \norm{D_H f}^2_{\text{CM}}d\nu_k(x)\,.
    \end{align*}
        
    \ppart{Limiting LSI for $C_b^1(E)$} We now combine the uniform LSIs for $\nu_k$ along with $L^1$ convergence of $\nu_k \to \nu$ to obtain the final LSI for $\nu$. For this, we first argue that $V_k \to V$, then use the integrability of $e^V$ to obtain convergence of $e^{V_k} \to e^V$, and finally use the convergences of $Z_k \to Z$ and $\nu_k \to \nu$ in total variation distance to obtain convergence of the Dirichlet form for functions in $C_b^1(E)$.
    
    Note that $\{V_k\}_{k \in \N}$ is a Doob martingale and $\calF_{\infty} := \sigma\left(\bigcup_{k \in \N} \calF_k\right) = \calB(E)$ ignoring Borel sets $A \in \calB(E)$ where $\gamma(A) = 0$. This implies, by L\'evy's upward theorem for martingale convergence~\cite[Theorem 14.2]{williams1991probability}, that
    \[
        V_k \to \E[V \mid \calF_\infty] = V\,,
    \]
    $\gamma$-a.s. and in $L^1(\gamma)$. By the continuous mapping theorem, this further yields
    \[
        e^{V_k} \to e^V\,,
    \]
    $\gamma$-a.s.. For $L^1$ convergence, first observe that since $V$ is $K$-Lipschitz in the CM norm, the sub-Gaussian tails implied by~\pref{t:borell} give that, for any finite $p > 0$, $\E_\gamma[e^{pV}] < \infty$\footnote{A proof of this is provided in~\pref{lem:infinite-dim-gaussian-ui}.}. Then, by Jensen's inequality applied to $e^{x}$,
    \[
        e^{pV_k} = e^{p\E[V\mid \calF_k]} \le \E[e^{pV} \mid \calF_k]\, ,
    \]
    and taking expectations over $\gamma$ on both sides while using the tower property of conditional expectations gives
    \[
        \sup_k \E_\gamma[e^{pV_k}] \le \E_\gamma[e^{pV}] < \infty\,. 
    \]
    Therefore, $\{e^{V_k}\}_k$ is uniformly integrable \cite[Theorem 4.6.2]{durrett2019probability}. Combining the $\gamma$-a.s.~convergence of $e^{V_k} \to e^V$ with uniform integrability and Vitali's theorem~\cite[Theorem 4.5.4]{bogachev2007measure} gives
    \[
        e^{V_k} \to_{L^1(\gamma)} e^V\,.
    \]
    This immediately implies
    \[
    Z_k := \int_E e^{V_k} d\gamma \to \int_E e^{V}\,d\gamma = Z.
    \]
    It now remains to show that the reweighed density $e^{V_k}/Z_k$ converges in $L^1$ to $e^V/Z$, as that immediately implies 
    \[
        \TV(\nu_k,\nu) = \frac{1}{2}\norm{\frac{e^{V_k}}{Z_k}-\frac{e^V}{Z}}_{L^1(\gamma)} \to 0\,.
    \]
    Since $Z > 0$, the convergence above implies that there exists a large enough $k$, such that $Z_k > 0$. So, 
    \allowdisplaybreaks
    \begin{align*}
        \norm{\frac{e^{V_k}}{Z_k} - \frac{e^V}{Z}}_{L^1} &= \norm{\frac{e^{V_k}}{Z_k} - \frac{e^V}{Z_k} + \frac{e^V}{Z_k} - \frac{e^V}{Z}}_{L^1} \\
        &\le \frac{1}{Z_k}\norm{e^{V_k} - e^V}_{L^1} + \left|\frac{1}{Z_k} - \frac{1}{Z}\right|\norm{e^V}_{L^1} \to 0\, ,
    \end{align*}
    by the $L^1$ convergence of $e^{V_k} \to e^V$, the convergence $Z_k \to Z$, and the integrability of $e^V$. This gives convergence of the Dirichlet form for any $f \in C_b^1(E)$ as $F(e) := \norm{D_H f(e)}^2_{\text{CM}}$ is bounded and measurable, giving
    \[
    \left|\int_E F(e)\,d\nu_k(e) - \int_E F(e)\,d\nu(e) \right|
        =
        \left|\int_E F(e)\left(\frac{e^{V_k(e)}}{Z_k} - \frac{e^{V(e)}}{Z}\right)d\gamma(e)\right|
        \le
        \norm{F}_\infty
        \norm{\frac{e^{V_k}}{Z_k} - \frac{e^V}{Z}}_{L^1(\gamma)}
        \to 0\,
        \]
    since $f$ is Lipschitz, and the densities converge in $L^1$ by the argument above. This immediately gives
    \[
        \int_{e \in E}\norm{D_H f}^2_{\text{CM}}d\nu_k(e) \to \int_{e \in E}\norm{D_H f}^2_{\text{CM}}d\nu(e)\,.
    \]
    Finally, note that the function $u \mapsto u\log(u)$ on $[0,\infty)$ with $0\log 0 := 0$ is continuous. Since $f$ is bounded and the density for $\nu_k$ converges to $\nu$ in $L^1(\gamma)$, we have that
    \[
        \int_E f(e)^2 d\nu_k(e) \to \int_E f(e)^2 d\nu(e)\, ,
    \]
    and, by the same $L^1$ density convergence applied to the bounded measurable function $e\mapsto f(e)^2\log(f(e)^2)$,
    \[
        \int_E f(e)^2\log(f(e)^2) d\nu_k(e) \to \int_E f(e)^2\log(f(e)^2)d\nu(e)\,.
    \]
    Plugging the convergences into the definition of the entropy functional gives
    \[
        \Ent_{\nu_k}[f^2] \to \Ent_{\nu}[f^2]\,.
    \]
    Combining the convergence of the entropy functional and Dirichlet form gives
    \[
        \Ent_{\nu}[f^2] \le \frac{4}{e^{-CK^2}}\int_{x\in E}\norm{D_H f}^2_{\text{CM}}d\nu(x)\, ,
    \]
    for every $f \in C_b^1(E)$. 
    
    \ppart{Closure under the Sobolev domain $\calD^{1,2}(\nu)$}  The final step is a closure argument to go from the cylindrical core $C_b^1(E)$ to the closed Sobolev
domain $\calD^{1,2}(\nu)$ defined in \pref{d:sobolev-domains}. Let $f \in \calD^{1,2}(\nu)$ and set $\{f_i\}_i$ be a Cauchy sequence of functions in $C_b^1(E)$ that converge to $f$ in $\norm{\cdot}_{\calD^{1,2}}$, so that $\int \norm{D_Hf_i}\CM^2d\nu \to \int \norm{D_Hf}\CM^2d\nu$. Then
    \[
         \norm{f^2 - f^2_i}_{L^1(\nu)} \le_{\text{Cauchy-Schwarz}} \norm{f - f_i}_{L^2(\nu)}\left(\norm{f}_{L^2(\nu)} + \norm{f_i}_{L^2(\nu)}\right)\,,
    \]
    and the lower semi-continuity of entropy gives that $\Ent_{\nu}[f^2] \le \liminf_{i \to \infty}\Ent_{\nu}[f^2_i]$. Combining these two with the LSI established above for $C_b^1(E)$ gives that, for $i \to \infty$,
    \[
        \Ent_{\nu}[f^2] \le \frac{4}{e^{-CK^2}}\int\norm{D_H f}\CM^2\,d\nu\,. \qedhere
    \]
\end{prf}

\subsection{LSI via Lipschitz end-point evaluation} We now use the Cameron--Martin Lipschitz bounds for the Jarzynski functional and the ASL drift in conjunction with the path-space LSI established in~\pref{t:lsi-CM-pert} to conclude that a Lipscshitz end-point evaluation will give a LSI for $\rh_T$. Note that the LSI will apply to all test functions $g: \R^n \to \R$ that are globally Lipschitz and (possibly) unbounded \emph{provided} $g \in L^2(\rh_T)$. 

\begin{lemma}[CM-Lipschitz maps under CM-Lipschitz Gaussian tilts]
\label{l:cm-lip-map-tilted}
Let \(V:E\to\R\) be finite and CM-Lipschitz, and define $d\nu^V=Z_V^{-1}e^V\,d\ga\,$. Let \(\Phi:E\to\R^m\) be \(K\)-CM-Lipschitz and suppose \(\Phi\in L^p(\ga;\R^m)\) for every finite \(p\). Then, for every \(f\in C_b^1(\R^m)\),
\[
    f\circ\Phi\in\calD^{1,2}(\nu^V),
    \qquad
    D_H(f\circ\Phi)
    =
    D_H\Phi^\sT\nabla f(\Phi)
    \quad \nu^V\text{-a.s.}
\]
\end{lemma}
\begin{prf}
Denote \(w=Z_V^{-1}e^V\) and observe that \pref{lem:infinite-dim-gaussian-ui} along with the fact that $0 < Z <\infty$ $\gamma$-a.s. imply \(w^q\in L^1(\ga)\) for every finite \(q>0\). Fix \(u\in\R^m\) and set \(F_u=\an{u,\Phi}\). Then \(F_u\) is \(K\norm u_2\)-CM-Lipschitz and belongs to \(L^p(\ga)\) for every finite \(p\). Choose \(p>2\). By \pref{l:cm-lip-malliavin},
\[
    F_u\in\calD^{1,p}(\ga)\,,
\]
and
\[
    \norm{D_HF_u}_{\CM}\le K\norm u_2
\]
$\gamma$-a.s. Choose a sequence \(\left\{F_{u,j}\in C_b^1(E)\right\}_j\) that converges to \(F_u\) in the \(\calD^{1,p}(\ga)\) graph norm. With \(q=p/(p-2)\), a ($p/2,p/(p-2)$)-H\"older's inequality gives
\[
    \norm{F_{u,j}-F_u}_{L^2(\nu^V)}^2 \le \norm{F_{u,j}-F_u}_{L^p(\ga)}^2 \norm{w}_{L^q(\ga)}\,,
\]
and the same estimate applies to the Malliavin derivatives $\left\{D_H F_{u,j}\right\}_j$. Hence, \(F_u\in\calD^{1,2}(\nu^V)\) with
\[
    \norm{D_HF_u}_{\CM}\le K\norm u_2
\]
$\nu^V\text{-a.s.}$. Taking \(u \in \left\{e_1,\ldots,e_m\right\}\) gives \(\Phi^{(j)}\in\calD^{1,2}(\nu^V)\) for every $j \in [m]$. Choosing \(u\) in a countable dense subset of the unit sphere \(\calS^{m-1}(1)\) and using linearity gives
\[
    \norm{D_H\Phi}_{\op}\le K
\]
$\nu^V\text{-a.s.}$. The final assertion then follows from \pref{l:malliavin-chain-rule}.
\end{prf}

\begin{lemma}[Endpoint map belongs to the tilted Sobolev domain]
\label{l:endpoint-map-sobolev}
Let \(d\nu^V\) be defined as in \pref{l:cm-lip-map-tilted} and let \(\Psi(B)=\hat y_T(B)\). 
Consequently, for every \(f\in C_b^1(\R^n)\),
\[
    f\circ\Psi\in\calD^{1,2}(\nu)\,
\]
and
\[
    D_H(f\circ\Psi) = D_H\Psi^\sT\nabla f(\Psi)
\]
$\nu\text{-a.s.}$.
\end{lemma}

\begin{prf}
Let
\[
    V(B):=\int_0^T \omega(\hat y_t(B))\,dt\,.
\]
By \pref{lem:cm-lipschitz-two}, \(V\) is finite and CM-Lipschitz, so \pref{l:cm-lip-map-tilted} applies to \(\nu^V\). \pref{lem:cm-lipscitz-one} implies that \(\Psi\) is \(\sqrt T e^{LT}\)-CM-Lipschitz. Moreover, \(\Psi\in L^p(\ga;\R^n)\) for every finite \(p\), by the standard moment estimate for Lipschitz SDEs. Applying
\pref{l:cm-lip-map-tilted} with \(\Phi=\Psi\) therefore immediately gives, for every \(f\in C_b^1(\R^n)\),
\[
    f\circ\Psi\in\calD^{1,2}(\nu)\,,
\]
and
\[
    D_H(f\circ\Psi) = D_H\Psi^\sT\nabla f(\Psi)\,.
\]
$\nu\text{-a.s.}$. \qedhere
\end{prf}

We now prove the LSI for $\rh_T$ over the Sobolev space $W^{1,2}(\rh_T)$ by combining \pref{t:lsi-CM-pert} and \pref{l:endpoint-map-sobolev} to ``pushforward'' the path-space LSI back to Euclidean space.

\begin{proposition}[Annealed distribution LSI over the Sobolev space $\calW^{1,2}(\rh_T)$]\label{prop:lsi-rhT-bounded}
    Fix $T \in [0, \infty)$ independent of $n$. Let $\rh_T$ be defined as in~\pref{e:je-weights} in the special case that $\hat{p}_0 = \delta_0$. Then, for every $f \in \calW^{1,2}(\rh_T)$, the following LSI holds for the measure $\rh_T$,
    \[
        \Ent_{\rh_T}[f^2] \le \frac{4T e^{2LT}}{e^{-CK^2}}\int_{x \in \R^n}\norm{\nabla f(x)}^2_2 d\rh_T(x)\,,
    \]
    where $L$ and $K$ are as defined in~\pref{a:Lipschitz-drift} and \pref{t:lsi-CM-pert}, and $C > 0$ is an absolute constant.
\end{proposition}
\begin{prf}
    The proof uses the push-forward map $\rh_T = \Psi_{\#}\nu$ where $\Psi: E \to \R^n$ and $\Psi(B) = \hat{y}_T(B)$, and combines it with the fact that the Malliavin derivative of $\Psi$ is bounded $\nu$-a.s. (a direct consequence of \pref{lem:cm-lipscitz-one} and \pref{l:endpoint-map-sobolev}) so we can invoke~\pref{t:lsi-CM-pert} after an application of the (weak) chain rule in conjunction with the stability of the Malliavin derivative. The final step is a closure argument to the entire Sobolev domain $\calW^{1,2}(\rh_T)$.

    \ppart{Annealed LSI for $C^1_b(\R^n)$}
    Let $f \in C^1_b(\R^n)$ and define its composition with the ASL drift as $F:= f \circ \Psi$.
    By \pref{lem:cm-lipscitz-one}, the endpoint map
$\Psi(B)=\hat y_T(B)$ is $K_\Psi$-CM-Lipschitz with
\[
    K_\Psi=\sqrt T e^{LT}.
\]
Applying \pref{l:endpoint-map-sobolev} to $\Psi$ gives
\[
    f\circ\Psi\in\calD^{1,2}(\nu),
    \qquad
    D_H(f\circ\Psi)(B)
    =
    D_H\Psi(B)^\sT\nabla f(\Psi(B))
    \quad \nu\text{-a.s.},
\]
which implies
    \[
        \norm{D_H(f \circ \Psi)(x)}_{\text{CM}}^2 \le_{\text{Cauchy-Schwarz}} \norm{D_H \Psi(x)}^2_{\text{CM}}\norm{\nabla f(\Psi(x))}^2_2 \le_{\text{\pref{l:endpoint-map-sobolev}}} Te^{2LT}\norm{\nabla f(\Psi(x))}^2_2\,.  
    \]
    This further gives
    \allowdisplaybreaks
    \begin{align*}
        \Ent_{\rh_T}[f^2] &=_{\rh_T = \Psi_{\#}\nu} \Ent_{\nu}[(f \circ \Psi)^2] \le_{\text{\pref{t:lsi-CM-pert}}} \frac{4}{e^{-CK^2}} \int_{x \in E}\norm{D_H(f\circ \Psi)(x)}^2_{\text{CM}}d\nu(x) \\
        &\le \frac{4Te^{2LT}}{e^{-CK^2}}\int_{x \in E}\norm{\nabla f(\Psi(x))}^2_2 d\nu(x) = \frac{4Te^{2LT}}{e^{-CK^2}} \int_{a \in \R^n} \norm{\nabla f(a)}^2_2 d\rh_T(a)\,.
    \end{align*}

    \ppart{Closure under the Sobolev domain $\calW^{1,2}(\rh_T)$} By \pref{d:euclidean-sobolev-domain},
$\calW^{1,2}(\rh_T)$ is the Euclidean graph-norm closure of
$C_b^1(\R^n)$. Let $f\in\calW^{1,2}(\rh_T)$ and choose
$f_j\in C_b^1(\R^n)$ such that $f_j\to f$ is a Cauchy sequence in the
$\calW^{1,2}(\rh_T)$ graph norm $\norm{f}_{\calW^{1,2}}^2 := \int_{x \in \R^n} f^2d\rh_T(x) + \int_{x \in \R^n}\norm{\nabla f(x)}^2_2 d\rh_T(x)$. An exactly analogous argument to the one in [Closure under the Sobolev domain $\calD^{1,2}(\nu)$,~\pref{t:lsi-CM-pert}] yields that
    \[
        \Ent_{\rh_T}[f^2] \le \frac{4Te^{2LT}}{e^{-CK^2}}\int_{x \in \R^n}\norm{\nabla f(x)}^2_2d\rh_T(x)\,. \qedhere
    \]
\end{prf}

For our use case, it will be critical that the LSI established in~\pref{prop:lsi-rhT-bounded} can be extended to hold for \emph{unbounded} Lipschitz test functions. This is established below by an explicit, smoothed truncation argument that makes use of the fact that the test functions have finite second moment with respect to $\rh_T$.

\begin{theorem}[Annealed distribution LSI for unbounded Lipschitz test functions]\label{thm:lsi-rhT-lipschitz}
Assume the same conditions as in \pref{t:lsi-CM-pert} and \pref{prop:lsi-rhT-bounded}.
    Fix $0 < A, T <\infty$. Let $f \in C^1(\R^n)$, such that $f$ is $A$-Lipschitz and $f \in L^2(\rh_T)$. Then, 
    \[
        \Ent_{\rh_T}[f^2] \le \frac{4T e^{2LT}}{e^{-CK^2}}\int_{x \in \R^n}\norm{\nabla f(x)}^2_2 d\rh_T(x)\,,
    \]
    where $L$ and $K$ are as defined in \pref{a:Lipschitz-drift} and \pref{t:lsi-CM-pert}, and $C > 0$ is an absolute constant.
\end{theorem}
\begin{prf}
    Let $f$ be $A$-Lipschitz and $f \in L^2(\rh_T)$. Fix increasing radii $R > 0$ and choose a family of truncation functions $\{ \tau_R : \R \mapsto \R\}_{R \ge 0}$ such that
    \[
        \tau_R(s) = s\, ,
    \]
    whenever $|s| \le R$, $|\tau_R(s)| \le |s|$ and $\norm{\tau_R}_{\text{Lip}} \le 1$\footnote{An explicit choice for such a truncator is $\tau_R(s) = \max(-R, \min(s,R))$.}. Denote the radius-$R$ truncation of $f$ as $f_R := \tau_R \circ f$. To smoothen this truncation so it can be approximated within $\calW^{1,2}(\rh_T)$, we set $g_{R,\eps} := \eta_\eps * f_R = \int_{x \in \R^n} \eta_\eps(y)f_R(x-y)\,dy$ where $\eta_\eps: \R^n \to \R$ is a mollifier\footnote{Mollifiers are standard in distribution theory, and their key properties used in the proof are collected in \pref{lem:mollifiers}.}. 

    \ppart{$L^2(\rh_T)$-convergence of hard truncations} It is straightforward to see that $f_R \to f$ pointwise as $R \to \infty$. Since $f \in L^2(\rh_T)$ and $|f_R(x)| \le |f(x)|$, the dominated convergence theorem implies that
    \[
        f_R \rightarrow_{L^2(\rh_T)} f\,.
    \]
    Also, since the composition of two Lipschitz continuous functions is continuous, $\norm{f_R}_{\text{Lip}} \le A$. Furthermore, $f_R$ is differentiable for every $x \not\in \{|f| = R\}$. Therefore, we choose a sequence of truncation radii $\{ R_j\}_{j}$ such that $\lim_{j \to \infty} R_j \to \infty$ and $\rh_T\{|f(x)| = R_j\} = 0$ for every $j$ since the set of atoms in the support of $\rh_T$ is at most countable. This immediately implies that, for $x \in \R \setminus \{-R,R\}$, 
    \[
        \nabla f_R(x) = \1_{|f(x)| < R}\nabla f(x)\,,
    \]
    which holds $\rh_T$-a.e.. So, since $\norm{ \1_{|f(x)| \le R}\nabla f(x)}_2 \le A$ and $\1_{|f(x)| \le R}\nabla f(x) \to \nabla f(x)$ pointwise as $R \to \infty$, the dominated convergence theorem implies that
    \[
        \1_{|f(x)| < R}\nabla f \rightarrow_{L_2(\rh_T)} \nabla f\,. 
    \]
    So, $f_R$ and $\nabla f_R$ already converge to $f$ and $\nabla f$, but $f_R \notin C^1_b(\R^n)$. To deal with this, we use the mollified $g_{R,\eps}$ evaluated at a sequence of truncation radii $\{R_j\}_j$ which are null-sets under $\rh_T$, and show that the properties of the mollifier allow us to conclude that $g_{R,\eps} \to f_R$ and $\nabla g_{R,\eps} \to \nabla f_R$ in $L^2(\rh_T)$ with $g_{R,\eps} \in C^1_b(\R^n)$. At this point, the closure of the Cauchy sequence $\{g_{R_j,\eps}\}_j$ under the $\norm{\cdot}_{\calW^{1,2}(\rh_T)}$-norm gives a $L^2(\rh_T)$ convergence to $f$ and shows that $f \in \calW^{1,2}(\rh_T)$. Invoking~\pref{prop:lsi-rhT-bounded} immediately yields the desired LSI for $f$.     

    \ppart{$L^2(\rh_T)$-convergence of mollified Cauchy sequence} Fix the sequence $R_j\to\infty$ chosen above with
    $\rh_T\{|f(x)|=R_j\}=0$. For each $R_j$ in the sequence, \pref{lem:mollifiers} applied to $F=f_{R_j}$ yields that
    \[
        g_{R_j,\eps}:=\eta_\eps*f_{R_j}\rightarrow_{L^2(\rh_T)} f_{R_j},\,
    \]
    and
    \[
        \nabla g_{R_j,\eps}\rightarrow_{L^2(\rh_T;\R^n)} \mathbf 1_{\{|f|<R_j\}}\nabla f
    \]
    Fix a sequence $\{\eps_j > 0\}_j$ such that $\lim_{j \to \infty} \eps_j \to 0$, and the following graph-norm convergence
    \[
        \|g_{R_j,\eps_j}-f_{R_j}\|_{L^2(\rh_T)} + \|\nabla g_{R_j,\eps_j} -\mathbf 1_{\{|f|<R_j\}}\nabla f\|_{L^2(\rh_T)}
        \le \frac1j\,,
    \]
    holds. Set $u_j:=g_{R_j,\eps_j}$ and observe that $u_j\in C_b^\infty(\mathbb R^n)\subset C_b^1(\mathbb R^n)$ by \pref{lem:mollifiers}. Moreover, since $f_{R_j}\to f$ in $L^2(\rh_T)$ and $\mathbf 1_{\{|f(x)|<R_j\}}\nabla f\to\nabla f$ in
    $L^2(\rh_T;\mathbb R^n)$, we obtain
    \[
        u_j\rightarrow_{L^2(\rh_T)} f
    \]
    and
    \[
    \nabla u_j\rightarrow_{L^2(\rh_T;\R^n)} \nabla f\,,
    \]
    by the triangle inequality. This means $f\in\calW^{1,2}(\rh_T)$ by
\pref{d:euclidean-sobolev-domain}.
    \pref{prop:lsi-rhT-bounded} then gives
    \[
        \Ent_{\rh_T}(f^2) \le \frac{4Te^{2LT}}{e^{-C K^2}} \int_{\R^n}\|\nabla f(x)\|_2^2\,d\rh_T(x)\,. \qedhere
    \]

\end{prf}

\subsection{Proof of LSI for annealed distributions}
\label{s:rho-t-lsi-proof}
\begin{prf}[Proof of \pref{t:rho-ls}]
Let
\[
    W(B)=\int_0^T\omega(\hat y_t(B))\,dt .
\]
By \pref{lem:cm-lipschitz-two}, \(W\) is \(K\)-CM-Lipschitz with
\(K\le L_\omega T^{3/2}e^{LT}\). Hence \pref{t:lsi-CM-pert} applies to
\[
    d\nu=Z^{-1}e^W\,d\gamma
\]
and gives a LSI for \(\nu\). By the Jarzynski representation
\eqref{e:je-weights}, \(\rh_T=\Psi_\#\nu\), where
\(\Psi(B)=\hat y_T(B)\). Since \(\Psi\) is \(\sqrt T e^{LT}\)-CM-Lipschitz
by \pref{lem:cm-lipscitz-one}, \pref{prop:lsi-rhT-bounded} gives the
claimed LSI for \(\rh_T\), with constants depending only on
\(T,L,L_\omega\) and the universal constants in \pref{t:lsi-CM-pert}.
\end{prf}

\section{Warm starts to the SL distribution}
\label{s:ws}
We expand the notion of a $L^\iy$ warm start to allow $\ep$ error in total variation.
\begin{definition}
\label{d:ws}
    Let $\hat P$ and $P$ be probability measures on the same space. 
    We say that $\hat P$ is a \vocab{$(B,\ep)$-warm start} to $P$ if there exists a probability measure $Q$ such that 
    \begin{align*}
        \TV(\hat P, Q)&\le \ep, & 
        \ve{\dd{Q}{P}}_{L^\iy(P)} &\le B.
    \end{align*}
    Equivalently, $\E_{\hat P} \ba{1-\pa{B \dd{P}{\hat P}\wedge 1}} \le \ep$ or $\E_P[\dd{\hat P}{P}\wedge B]\ge 1-\ep$. 
\end{definition}
For short, we will write $\hat P\xra{B,\ep} P$.
It is easy to see that if $\hat P\xra{B,\ep} P$, then $\hat P(A) \le B\cdot P(A)+\ep$.

Our main theorem in this section is the following.
\begin{theorem}[Annealed and ASL distributions are warm starts for SL distribution]
\label{t:annealed-ws}
Suppose \pref{a:m-error}, \pref{a:Lipschitz-drift}, and \pref{a:Lipschitz-je} hold for \eqref{e:rh-t-F}.
    Then $\rh_t$ and $\hat p_t$ are both $((1/\ep)^{O(1)}, \ep)$-warm starts to $p_t$.
\end{theorem}

First, we give some basic facts about when warm starts are preserved. 
\begin{proposition}[Preservation of warm starts]\label{p:ws-preserve}
~\vspace{-\baselineskip}
\begin{thmenum}
    \item\label{pitem:ws-preserve-data-processing} (Data processing) Let $T$ be a measurable map. If $\hat P$ is a $(B,\ep)$-warm start for $P$, then $T_\# \hat P$ is a $(B,\ep)$-warm start for $T_\# P$.
    \item\label{pitem:ws-preserve-warm-start-transitive} (Transitivity of warm starts) Suppose $P_1$ is a $(B_1,\ep_1)$-warm start to $P_2$, and $P_2$ is a $(B_2,\ep_2)$-warm start to $P_3$. Then $P_1$ is a $(B_1B_2, \ep_1+B_1\ep_2)$-warm start to $P_3$.
    \item\label{pitem:ws-preserve-warm-start-conditional} (Warm start from conditional distributions) Let $\hat P$ and $P$ be joint distributions on $(X,Y)$. 
    Suppose there is $A$ such that $A\subeq \set{x}{\dd{\hat P_X}{P_X}(x) \le B_1}$ (where $P_X$ denotes the marginal distribution of $X$) with $\hat P(X\in A)\ge 1-\ep_1$\footnote{Necessarily, for the marginal distributions of $X$, $\hat P_X$ is a $(B_1,\ep_1)$-warm start to $P_X$.}, and for $x\in A$, 
    $\hat P(\cdot \mid X=x)$ is a $(B_2,\ep_2)$-warm start for $P(\cdot \mid X=x)$. Then $\hat P$ is a $(B_1B_2,\ep_1+\ep_2)$-warm start for $P$.
\end{thmenum}
\end{proposition}
We note as a corollary of \pref{pitem:ws-preserve-warm-start-transitive} that if we have the power-law tails that for any $\ep>0$,
\[
P_1 \xra{O(1/\ep)^{C_1}, \ep} P_2 \xra{O(1/\ep)^{C_2}, \ep} P_3,
\]
then $P_1\xra{O(1/\ep_1)^{C_1} O(1/\ep_2)^{C_2}, \ep_1 + O(1/\ep_1)^{C_1}\ep_2} P_3$. Taking $\ep_1=\fc \ep2$, $\ep_2 = O(\pf{\ep}2^{C_1+1})$ gives $P_1\xra{O(1/\ep)^{C_1C_2+C_1+C_2},\ep} P_3$.
\begin{prf}
For part i, take $T_\# Q$ as the witness in the first definition.

For part ii, we have
\begin{align*}
    \E_{P_1}\ba{1-\pa{B_1B_2 \dd{P_3}{P_1}\wedge 1}} &\le 
    \E_{P_1}\ba{1-\pa{B_1 \dd{P_2}{P_1}\wedge 1} + \pa{B_1 \dd{P_2}{P_1}\wedge 1}\pa{1-B_2 \dd{P_3}{P_2}\wedge 1}}\\
    &\le \ep_1 + B_1 \E_{P_2}\pa{1-B_2 \dd{P_3}{P_2}\wedge 1} \le \ep_1+B_1\ep_2.
\end{align*}

For part iii, 
\begin{align*}
    \E_{\hat P}\ba{1-\pa{B_1B_2 \dd{P}{\hat P}\wedge 1}} &\le \ep_1+
    \E_{\hat P}\ba{\one_{x\in A}\pa{1-\pa{B_1B_2 \dd{P}{\hat P}\wedge 1}}} \\
    &\le \ep_1+\E_{\hat P}\ba{\one_A \E_{\hat P}\ba{1-\pa{B_1B_2 \dd{P_X}{\hat P_X}\dd{P_{Y\mid X}}{\hat P_{Y\mid X}}\wedge 1}} \mid X} \\
    &\le \ep_1+\E_{\hat P}\ba{\one_A \E_{\hat P}\ba{1-\pa{B_2 \dd{P_{Y\mid X}}{\hat P_{Y\mid X}}\wedge 1}} \mid X}\le \ep_1+\ep_2.\qedhere
\end{align*}
\end{prf}

The following is a simple consequence of \Cref{l:wpi-converge}.
\begin{lemma}[Mixing from an approximate $L^\iy$ warm start]
\label{l:ws-mix}
    Suppose that for a Markov process, $\pi$ satisfies a $(c,\de)$-weak Poincar\'e inequality.
    Suppose that $\nu_0$ is a $(B,\ep)$-warm start to $\pi$. Then
    \[
\TV(\nu_T,\pi)
\le \rc 2\sqrt{e^{-2c T}+ \de }\cdot B + \ep.
    \]
    For a Markov chain, the same bounds hold for $T\in \N$ with $e^{-2c T}$ replaced by $(1-c)^{2T}$.
\end{lemma}
\begin{prf}
Let $\td \nu_0$ be such that
\[
    \TV(\nu_0,\td\nu_0)\le \ep,
    \qquad
    \left\|\dd{\td\nu_0}{\pi}\right\|_\infty\le B.
\]
Let $\td \nu_T$ be the distribution at time $T$ with initial distribution $\td\nu_0$. Then by the triangle inequality and data processing, and \pref{l:wpi-converge} applied to $\td \nu_T$,
    \[\TV(\nu_T,\pi)\le \TV(\td \nu_T,\pi) +  \ep
    \le \rc2\sqrt{\chi^2(\td \nu_T\| \pi)} +  \ep \le\rc 2 \sqrt{e^{-2c T}  + \de }\cdot B + \ep.\qedhere
    \]
\end{prf}

To set up a coupling argument, we record the standard fact that non-atomic probability measures are all isomorphic to each other.
\begin{lemma}[Atomless matching]\label{l:atomless-matching}
Let $E,F$ be Borel subsets of $\R^n$, and let $\mu_E,\mu_F$ be finite
non-atomic Borel measures on $E$ and $F$ with
\[
    \mu_E(E)=\mu_F(F)<\infty.
\]
Then there exist full-measure measurable subsets $E_0\subseteq E$ and
$F_0\subseteq F$ and a measurable bijection
\[
    S:E_0\to F_0
\]
with measurable inverse such that
\[
    S_\#(\mu_E|_{E_0})=\mu_F|_{F_0}.
\]
\end{lemma}

\begin{prf}
If the common mass is zero, this is trivial. Otherwise normalize both
measures to be probability measures. The result is the standard
isomorphism theorem for atomless standard probability spaces; see, e.g.,
\cite[Theorem A.7]{janson2013graphons}. Rescaling gives the stated
finite-measure form.
\end{prf}
We use the following variation of the coupling argument of \Cref{l:couple-distance}. 
\begin{lemma}[Coupling sets with small distance, II]
\label{l:couple-distance-2}
    Keep the assumptions of \Cref{l:set-separation-lsi}. Suppose $a\le b$ and $0<\ep<1$, and assume moreover that $\mu$ is non-atomic.
    Let $\mu_A:=\mu(\cdot\cap A)/a$ and
$\mu_B:=\mu(\cdot\cap B)/b$.
    Then for any $d> \sfc{2}{c} \pa{\sqrt{\ln \prc {a\ep}}+ \sqrt{\ln \prc {b\ep}}}$,
there exist full-measure sets $B'\subseteq B$ and $A'\subseteq A$ and a
measurable bijection $T:B'\to A'$ with measurable inverse such that
$T_\#\mu_B=\mu_A$ and
\[
    \mu_B\{y\in B':\|T(y)-y\|\le d\}\ge 1-\ep.
\]
\end{lemma}
\begin{prf}
Fix $\td\ep\in(0,\ep)$ to be chosen at the end, and let
$0<\de<a\td\ep$. Let $K=[-N\de,N\de]^n$ be large enough that
$\mu(K)\ge 1-\de$, and divide it into boxes of side length $\de$.

We inductively construct a transport map on a large subset of $B$. Let
$A_0=A\cap K$ and $B_0=B\cap K$. We maintain the invariant
\[
    \mu(B_0 \setminus B_t)=\frac ba\, \mu(A_0 \setminus A_t).
\]
At step $t$, while $\mu(A_t)>\td\ep a$, and discarding the parts of the measure falling in boxes that contain zero measure, \Cref{l:set-separation-lsi}
gives points $x\in A_t$, $y\in B_t$ such that
\[
    \|x-y\|
    \le
    \sqrt{\frac{2}{c}}
    \left(
        \sqrt{\ln\frac1{a\td\ep}}
        +
        \sqrt{\ln\frac1{b\td\ep-\de}}
    \right).
\]
Let $K_x$ and $K_y$ be the side-length-$\de$ boxes containing $x$ and $y$, which have non-zero mass because we discarded zero-mass boxes before applying \Cref{l:set-separation-lsi}. Choose measurable
sets
\[
    K_x'\subseteq A_t\cap K_x,\qquad
    K_y'\subseteq B_t\cap K_y
\]
such that either $K_x'=A_t\cap K_x$ or $K_y'=B_t\cap K_y$, and
\[
    \frac{\mu(K_y')}{\mu(K_x')}=\frac ba .
\]
Equivalently,
\[
    \mu_B(K_y')=\mu_A(K_x').
\]
By \Cref{l:atomless-matching} applied to the finite measures
$\mu_B|_{K_y'}$ and $\mu_A|_{K_x'}$, there are full-measure subsets of
$K_y'$ and $K_x'$ and a measurable bijection between them, with measurable
inverse, pushing $\mu_B|_{K_y'}$ to $\mu_A|_{K_x'}$. Define $T$ on this
matched source piece by this bijection.

Set
\[
    A_{t+1}=A_t\setminus K_x',
    \qquad
    B_{t+1}=B_t\setminus K_y'.
\]
The process terminates after finitely many steps, since one box is emptied
at each step. Let $B_{\rm g}$ and $A_{\rm g}$ be the unions of the matched
source and target pieces, excluding the null sets discarded by
\Cref{l:atomless-matching}. Then
\[
    \mu_B(B_{\rm g})=\mu_A(A_{\rm g}),
\]
and for every $y\in B_{\rm g}$,
\[
    \|T(y)-y\|
    \le
    \sqrt{\frac{2}{c}}
    \left(
        \sqrt{\ln\frac1{a\td\ep}}
        +
        \sqrt{\ln\frac1{b\td\ep-\de}}
    \right)
    +2\sqrt n\,\de .
\]
Moreover, since $\mu(A_0)\ge a-\de$ and the procedure stops with
$\mu(A_t)\le a\td\ep$,
\[
    \mu_A(A_{\rm g})\ge 1-\td\ep-\frac{\de}{a},
\]
and hence also
\[
    \mu_B(B_{\rm g})\ge 1-\td\ep-\frac{\de}{a}.
\]

The residual measures
    $\mu_B|_{B\setminus B_{\rm g}}$
    and
    $\mu_A|_{A\setminus A_{\rm g}}$
have the same total mass and are non-atomic. Applying
\Cref{l:atomless-matching} once more, define $T$ on full-measure subsets of
the residual pieces by a measurable bijection pushing the first residual
measure to the second. Patching the finitely many matched-piece maps with
this residual map gives full-measure sets $B'\subseteq B$ and $A'\subseteq A$
and a measurable bijection $T:B'\to A'$ with measurable inverse such that
\[
    T_\#(\mu_B|_{B'})=\mu_A|_{A'}.
\]
Redefining $T$ arbitrarily on a $\mu_B$-null set obtains
    $T_\#(\mu_B)=\mu_A$.

Finally, since $d$ is strictly larger than
\[
    \sqrt{\frac{2}{c}}
    \left(
        \sqrt{\ln\frac1{a\ep}}
        +
        \sqrt{\ln\frac1{b\ep}}
    \right),
\]
we may choose $\td\ep<\ep$ sufficiently close to $\ep$ so that
\[
    d>
    \sqrt{\frac{2}{c}}
    \left(
        \sqrt{\ln\frac1{a\td\ep}}
        +
        \sqrt{\ln\frac1{b\td\ep}}
    \right).
\]
Then choose $\de>0$ sufficiently small that
    $\de<a\td\ep$ and
    $\de/a\le \ep-\td\ep$
and
\[
    \sqrt{\frac{2}{c}}
    \left(
        \sqrt{\ln\frac1{a\td\ep}}
        +
        \sqrt{\ln\frac1{b\td\ep-\de}}
    \right)
    +2\sqrt n\,\de
    \;\le\; d.
\]
With these choices,
\[
    \mu_B\{y\in B':\|T(y)-y\|\le d\}
    \;\ge\;
    \mu_B(B_{\rm g})
    \;\ge\;
    1-\td\ep-\frac{\de}{a}
    \;\ge\;
    1-\ep.
\]
This proves the claim.
\end{prf}

If we have a KL bound $\KL(p\|\hat p)$, then a priori $\hat p$ could be large where $p$ has small mass, but at least a constant portion of $\hat p$ is such that $p$ not too much smaller. This is weaker than $\hat p$ being a warm start; for the latter, we need $1-\ep$ mass under $\hat p$, and we will bootstrap to obtain this in the next lemma.
\begin{lemma}[Reverse KL bound implies capturing constant portion of distribution]
\label{l:rev-KL-constant-portion}
Let $\mu,\nu$ be probability measures with $\KL(\mu\|\nu)<\infty$.
Assume moreover that $\nu\ll\mu$.
    Let $L, L_2>1$. Then
    \[
\nu\pa{\bc{\dd{\nu}{\mu}\le L_2}} \ge \rc L \pa{
1- \fc{\KL(\mu\|\nu)+1}{\log L} - \rc{L_2}
}.
    \]
\end{lemma}
\begin{prf}
    First note that if $\mu\ll \nu$, then 
    \begin{align}
    \nonumber
        \KL(\mu\|\nu)& = \E_\mu \pa{\log \dd{\mu}{\nu}}\one_{\dd{\mu}{\nu}\ge 1} + \E_\mu\pa{\log \dd{\mu}{\nu}} \one_{\dd{\mu}{\nu}< 1} 
        \ge  \E_\mu \pa{\log \dd{\mu}{\nu}}\one_{\dd{\mu}{\nu}\ge 1} - 1
        \ge \mu\pa{\dd{\mu}{\nu}\ge L}(\log L)-1\\
        \implies
        \mu\pa{\dd{\mu}{\nu}\ge L} & \le \fc{\KL(\mu\|\nu) + 1}{\log L}.
        \label{e:kl-markov}
    \end{align}
    Hence
    \begin{align*}
        \nu\pa{\bc{\dd{\nu}{\mu}\le L_2}}
        &\ge\E_{\mu}\ba{\dd{\nu}{\mu} \one_{\bc{\dd{\nu}{\mu}\le L_2 \text{ and }\dd{\mu}{\nu}\le L}}}\ge \rc{L} \mu\pa{
\bc{\dd{\nu}{\mu}\le L_2}\cap \bc{\dd{\nu}{\mu}\ge \rc L}
        }\\
        &\ge \rc L \ba{\mu\pa{
\bc{\dd{\mu}{\nu}\le L }}- \mu\pa{\bc{\dd{\nu}{\mu}>L_2}}}\ge_{\eqref{e:kl-markov},\,\text{Markov}} \rc L \pa{1 - \fc{\KL(\mu\|\nu) + 1}{\log L} - \rc{L_2}}.\qedhere
    \end{align*}
\end{prf}

We first establish that a convolved distribution is a warm start.
\begin{lemma}[If $q$ satisfies LSI and captures a constant portion of the SL distribution, then the noising process turns it into a warm start]
\label{l:constant-portion-noise-ws}
    Suppose that $q_{t_1}$ is non-atomic, satisfies a log-Sobolev inequality with constant $c_{t_1}$ and $q_{t_1}\pa{\bc{\dd{q_{t_1}}{p_{t_1}}\le L}}\ge c$. 
    Then $q_{t_1}K_{t_1\to t_0}$ is a $\pa{\fc{L}{c}\exp\pa{O\pa{\pa{\fc{t_0}{t_1(t_1-t_0)c_{t_1}} \vee 1}\ln \prc{c\ep}}},\ep}$-warm start to $p_{t_0}$.
\end{lemma}
\begin{prf}
    It suffices to prove the claim with error $\eta+\eta'$ and then take
    $\eta=\eta'=\ep/2$. Replacing $c$ by $c\wedge \frac12$ only changes
    constants that can be absorbed by the $O(\cdot)$, so assume $c\le \frac12$. Let
    \[
        A\subeq \bc{\dd{q_{t_1}}{p_{t_1}}\le L},
        \qquad q_{t_1}(A)=c,
    \]
    and write
    \[
        q_A:=q_{t_1}(\cdot\mid A),
        \qquad
        q_B:=q_{t_1}(\cdot\mid A^c).
    \]
    By \pref{l:couple-distance-2}, there is a
    measurable bijection, modulo null sets, $T:A^c\to A$ with
    $T_\#q_B=q_A$ such that, for
    \[
        \bar d:=\fc{2}{\sqrt{c_{t_1}}}
        \pa{\sqrt{\ln\prc{c\eta}}+\sqrt{\ln\prc{(1-c)\eta}}},
    \]
    the set
    \[
        B_1:=\set{y\in A^c}{\|T(y)-y\|\le \bar d}
    \]
    satisfies $q_B(B_1)\ge 1-\eta$. Let $A_1:=T(B_1)$. Since $T$ is
    bijective and measure-preserving, $q_A(A_1)\ge 1-\eta$.

    Extend $T:A^c \to A$ by the identity on $A$ to get the map $\td T: \R^n \to A$. Then
    \[
        \td T_\#q_{t_1}=q_A,
        \qquad
        \dd{q_A}{p_{t_1}}\le \fc Lc .
    \]
    Consider the joint distributions
    \begin{align*}
        X_{t_1}&\sim p_{t_1} & X_{t_0}&\sim K_{t_1\to t_0}(X_{t_1}, \cdot)\\
        \hat X_{t_1} & \sim q_{t_1} & \hat  X_{t_0} &\sim K_{t_1\to t_0}(\hat X_{t_1},\cdot).
    \end{align*}
    We claim that $\Law(\td T(\hat X_{t_1}), \hat X_{t_0})$ is a warm start to $\Law(X_{t_1},X_{t_0})$.
    Fix $x\in A_1$. Let $y=T^{-1}(x)$ and $v=y-x$. Then
    $\|v\|\le \bar d$, and, since
    $q_{t_1}=cq_A+(1-c)q_B$ and $T$ is bijective, either $\hat X_{t_1} \in A$ in which case $\td T(\hat X_{t_1}) = \hat X_{t_1}$ or $\hat X_{t_1} \in A^c$ in which case $\td T(\hat X_{t_1}) = T(\hat X_{t_1})$ and so
    \[
        \Law(\hat X_{t_1}\mid \td T(\hat X_{t_1})=x)
        =
        c\delta_x+(1-c)\delta_y.
    \]
    Thus, writing $K(z,\cdot):=K_{t_1\to t_0}(z,\cdot)$,
    \[
        \Law(\hat X_{t_0}\mid \td T(\hat X_{t_1})=x)
        =
        cK(x,\cdot)+(1-c)K(y,\cdot),
        \qquad
        \Law(X_{t_0}\mid X_{t_1}=x)=K(x,\cdot).
    \]

    Let $\nu:=K(x,\cdot)$ and $\mu:=K(y,\cdot)$. These are Gaussians with
    common variance
    \[
        \sigma^2=t_0\pa{1-\frac{t_0}{t_1}}
        =
        \frac{t_0(t_1-t_0)}{t_1}
    \]
    and mean difference $\Delta=\frac{t_0}{t_1}v$. A standard Gaussian
    density-ratio calculation gives that, outside a set of $\mu$-mass at
    most $\eta'$,
    \[
        \dd{\mu}{\nu}
        \le
        B:=
        \exp\pa{
            \fc{t_0\bar d^2}{2t_1(t_1-t_0)}
            +
            \bar d
            \sqrt{\fc{2t_0\ln(1/\eta')}{t_1(t_1-t_0)}}
        }.
    \]
    Hence
    \[
        cK(x,\cdot)+(1-c)K(y,\cdot)
        \xra{B,\,\eta'} K(x,\cdot).
    \]

    The first marginal of $(\td T(\hat X_{t_1}),\hat X_{t_0})$ is $q_A$, and on $A_1$ it has
    density at most $L/c$ with respect to $p_{t_1}$, while
    $q_A(A_1)\ge 1-\eta$. By \pref{pitem:ws-preserve-warm-start-conditional},
    \[
        \Law(\td T(\hat X_{t_1}),\hat X_{t_0})
        \xra{LB/c,\;\eta+\eta'}
        \Law(X_{t_1},X_{t_0}).
    \]
    Projecting to the second coordinate gives
    \[
        q_{t_1}K_{t_1\to t_0}
        \xra{LB/c,\;\eta+\eta'}
        p_{t_0}.
    \]
    Taking $\eta=\eta'=\ep/2$ and using the definition of $\bar d$ gives
    \[
        B\le
        \exp\pa{
            O\pa{
                \pa{\fc{t_0}{t_1(t_1-t_0)c_{t_1}}\vee 1}
                \ln\prc{c\ep}
            }
        },
    \]
    which is the claimed bound.
\end{prf}

This lemma will give that $\rh_{t_1}K_{t_1\to t_0}$ is a warm start to $p_{t_0}$. We now relate this to $\rh_{t_0}$ and $\hat p_{t_0}$. 

If $\rh_t$ is an approximation of the SL process, then we expect that $\rh_{t_1}K_{t_1\to t_0} \approx \rh_{t_0}$. 
The following key calculation shows reweighting $\rh_{t_1}K_{t_1\to t_0}$ via Jarzynski's equality gives $\rh_{t_0}$, and the ODE for the weights is the same as for the SL process. As the weights are concentrated for the same reason, this makes the approximation precise. We will see in \Cref{l:ws-je} that concentration of the Jarzynski weights will give that each side is a warm start for the other.
\begin{lemma}
\label{l:je-reverse}
    Suppose that $\rh_s$ has the explicit form \eqref{e:rh-t-F} for
    $s\in[t,t_1]$, with $\nabla\calF=\hat m$, and that the regularity
    hypotheses of \pref{t:je} hold on this interval. This is the reverse-time
    analogue of \pref{a:Lipschitz-je}.
    Then we have that for $0<t<t_1$,
    \begin{align*}
        \dd{\rh_{t}}{\rh_{t_1}K_{t_1\to t}}(x)
        &= \fc{\E[e^{w_t} | y_t = x]}{\E[e^{w_t}]},
    \end{align*}
    where $y_t$, $w_t$ are defined by
    \begin{align*}
        dy_t &= \fc{y_t}{t}dt + dB_t^{\leftarrow}\\
        dw_t &= \rc 2 \ba{\gd \cdot \hat m(y) + \ve{\hat m(y)}^2} = \om(y),
    \end{align*}
    and $B_t^{\leftarrow}$ is reverse Brownian motion.
\end{lemma}
Here, the reverse Brownian motion can be interpreted as follows: if $dy_t = f(y_t,t)\,dt + G(y_t,t)\,dB_t^{\leftarrow}$, then the time-reversed process $y_t^{\leftarrow} = y_{-t}$ satisfies $dy_t^{\leftarrow} = -f(y_t^{\leftarrow},-t) \,dt+ G(y_t^{\leftarrow},-t)\,dB_t$.
\begin{prf}
    By \Cref{t:je} in reverse time, if
    \begin{align*}
        dy_t &= \ba{\rc 2 \gd U_t(y_t) - b_t(y_t)}dt + dB_t^{\leftarrow}, & y_{t_1} &\sim \rh_{t_1}\\
        dw_t &= \ba{-\gd \cdot b_t(x_t) + \an{\gd U_t(x_t), b_t(x_t)} - \pl_t U_t(x_t)}dt, & w_{t_1}&=0,
    \end{align*}
    and $\rh_t(y) \propto e^{-U_t(y)}$, and $\hat \rh_t$ is the density of $y_t$, then we have
    \[
\dd{\rh_t}{\hat \rh_t}(x) = \fc{\E[e^{w_t} | y_t = x]}{\E[e^{w_t}]}.
    \]
    In our case, we take $-U_t(y) = \calF(y) - \fc{\ve{y}^2}{2t}$ and $b_t = -\rc 2 \ba{\hat m(y) + \fc{y}t}$, so that $\gd U_t(y) = -\hat m(y) + \fc yt$ and 
    \begin{align*}
        dy_t &= \fc{y_t}{t}dt + dB_t^{\leftarrow}\\
        dw_t &= \ba{\rc 2 \gd \cdot \hat m(y_t) + \fc{n}{2t} - \rc 2\an{-\hat m(y_t)+\fc{y_t}{t},\hat m(y_t)+\fc{y_t}{t}} + \fc{\ve{y_t}^2}{2t^2}}dt 
        = [\om(y_t) + c(t)]\,dt,
    \end{align*}
    where $c(t)$ is some function depending only on $t$. Solving the SDE for $y_t$ gives that $\hat \rh_t = \rh_{t_1}K_{t_1\to t}$. Removing $c(t)$ has no effect on the normalized weights $\fc{e^{w_t}}{\E[e^{w_t}]}$, so the result follows. 
\end{prf}

We obtain a tail bound for the Jarzysnki weights using the following adaptation of \pref{l:conc-lsi} for path space. 
Note this bound is loose by a factor of some absolute constant $C>0$ to enable a simpler proof.

\begin{lemma}[Lipschitz functions in path space are sub-gaussian]
\label{l:herbst-quick}
Let $\mu=\rho\otimes\gamma$, where $\rho$ is a measure on $\R^n$ satisfying
a log-Sobolev inequality with constant $\rls$, and $\gamma$ is the Wiener measure on $(E,\calB(E))$. Let $f:\R^n\times E\to\R$ be $L$-Lipschitz with respect to shifts by $(v,h) \in \R^n \times \calH$ with norm
\[
    \|(v,h)\|^2=\rls\|v\|_2^2+\|h\|_{\calH}^2 .
\]
Then, for every $t\in\R$ such that $|t| \ge 1$,
\[
    \E_\mu\exp\left(t(f-\E_\mu f)\right) \le 2\exp\left(CL^2t^2\right)\,,
\]
for some sufficiently large absolute constant $C>0$.
\end{lemma}
\begin{prf}
Let $X\sim\rho$ and $B\sim\gamma$ be independent. Define
\[
    g(B):=\E_\rho f(X,B).
\]
For fixed $B$, the map $x\mapsto f(x,B)$ is $L\sqrt{\rls}$-Lipschitz in the Euclidean norm. Hence, by \pref{l:conc-lsi} we have
\[
    \E_\rho\exp\left(t(f(X,B)-g(B))\right) \le \exp\left(\frac{L^2t^2}{2}\right)\,.
\]
Moreover, $g$ is $L$-Lipschitz in Cameron--Martin directions, since for $h\in\calH$,
\[
    |g(B+h)-g(B)| \le \E_\rho |f(X,B+h)-f(X,B)| \le L\|h\|_{\calH}.
\]
Therefore, by \pref{rem:lemma-ui},
\[
    \E_\gamma\exp\left(t(g-\E_\gamma g)\right) \le 2\exp\left(C' L^2t^2\right)\,,
\]
for some sufficiently large absolute constant $C'>0$. Combining the two conditional bounds gives
\begin{align*}
    \E_{\rho\otimes\gamma}e^{t(f-\E f)}
    &=\E_\gamma\left[e^{t(g-\E g)}\E_\rho e^{t(f-g)}\right] \\
    &\le\exp\left(\frac{L^2t^2}{2}\right)\E_\gamma e^{t(g-\E g)} \\
    &\le2\exp\left(CL^2t^2\right). \qedhere
\end{align*}
\end{prf}

\begin{lemma}[Mutual warm start from Jarzynski's equality]
\label{l:ws-je}
    Suppose that $\hat p_0$ satisfies a LSI with constant $\hat c_0$ and \pref{ass:Lipschitz-je} holds. Let $L_w =L_w(t)=L_\om t\sqrt{t+\hat {c_0}^{-1}}\cdot e^{Lt}$. Then 
\begin{enumerate}
    \item 
$\hat p_t\pa{\dd{\hat p_t}{\rh_t}\ge e^{\Om\pa{L_w\sqrt{\ln \prc{\ep}} + L_w^2}}} \le \ep$, and 
    $\hat p_t$ is a $(e^{O(L_w \sqrt{\ln(1/\ep)}+L_w^2)},\ep)$-warm start to $\rh_t$.
    \item 
$\rh_t\pa{\dd{\rh_t}{\hat p_t}\ge e^{\Om\pa{L_w\sqrt{\ln \prc{\ep}} + L_w^2}}} \le \ep$, and 
    $\rh_t$ is a $(e^{O(L_w \sqrt{\ln(1/\ep)}+L_w^2)},\ep)$-warm start to $\hat p_t$.
\end{enumerate}
\end{lemma}
Note that the parameters are $((1/\ep)^{o(1)}, \ep)$, so this gives a bound in R\'enyi divergence $\calR_p$ for any $p$.
\begin{prf}
By \Cref{lem:cm-lipschitz-two}, $w_t$ is $L_w(t)$-Lipschitz under \(\R^n\times\calH\) shifts in the norm given in \pref{e:norm-y-h}. 
Therefore, by 
\Cref{l:herbst-quick}, $\E e^{w_t} \le 2e^{CL_w^2} e^{\E w_t}$, so 
\begin{align}
\label{e:Ew-ub}
e^{\E w_t}\ge \frac{1}{2}e^{-CL_w^2} \E e^{w_t}.
\end{align}
By convexity of $x^{-u}$ for $u \ge 1$, another application of \pref{l:herbst-quick}, and then the above inequality,
\begin{align*}
    \hat p_t \ba{\dd{\hat p_t}{\rh_t}\ge e^{L}}
    &=_{\eqref{e:je-weights}} \P\ba{\E[e^{w_t}|y_t=x]^{-u}\ge e^{Lu} \E[e^{w_t}]^{-u}}\\
    &\le_{\text{Markov}} \fc{\E\ba{\E[e^{w_t}|y_t=x]^{-u}}}{e^{Lu}\E[e^{w_t}]^{-u}}\\
    &\le_{\text{Jensen}}
    \fc{\E\ba{\E[e^{-w_tu}|y_t=x]}}{e^{Lu}\E[e^{w_t}]^{-u}}\le \fc{\E[e^{-w_tu}]}{e^{Lu}\E[e^{w_t}]^{-u}}\\
    &
    \le_{\textup{\Cref{l:herbst-quick}}} 2e^{-u\E w_t} e^{L_w^2u^2 - Lu}\E[e^{w_t}]^{u}\\
    &\le_{\eqref{e:Ew-ub}} 4e^{CL_w^2u}e^{CL_w^2u^2- Lu}. 
\end{align*}
Now set
\[
    S_-:=\ln\frac{4}{\ep}\,, \qquad u_-:=1+\sqrt{\frac{S_-}{2D}}\,, \qquad L_-:=D+4Du_-=5D+2\sqrt{2DS_-}\,.
\]
where $D := CL^2_w$. Then $u_-\ge1$, and the preceding display gives
\[
    \hat p_t \ba{\dd{\hat p_t}{\rh_t}\ge e^{L_-}} \le 4e^{-2Du_-^2} \le 4e^{-S_-} = \ep\,.
\]
Therefore,
\begin{align}
\label{e:je-subg}
\hat p_t \ba{\dd{\hat p_t}{\rh_t}\ge e^{L_-}} \le \eps =: \de(L).
\end{align}
Letting $\ep=\de(L)$, we have $L_-=2\sqrt{2C}L_\omega\sqrt{\ln \frac{4}{\eps}} + 5CL_w^2$, so 
\[
\hat p_t\pa{\dd{\hat p_t}{\rh_t}\ge e^{2\sqrt{2C}L_\omega\sqrt{\ln \frac{4}{\eps}} + 5CL_w^2}} \le \ep.
\]
This gives the first part.

For the other direction, note by Jensen's inequality that $e^{\E w_t}\le \E e^{w_t}$.
For $u\ge 1$, we similarly have
\begin{align*}
    \hat p_t  \ba{\dd{\rh_t}{\hat p_t}\ge e^{L}}
    &=_{\eqref{e:je-weights}} \P\ba{\E[e^{w_t}|y_t=x]^{u}\ge e^{Lu}\E[e^{w_t}]^u}\\
    &\le_{\text{Markov}} \fc{\E\ba{\E[e^{w_t}|y_t=x]^{u}}}{e^{Lu}\E[e^{w_t}]^u}\\
    &\le_{\text{Jensen}}
    \fc{\E\ba{\E[e^{w_tu}|y_t=x]}}{e^{Lu}\E[e^{w_t}]^u}\le \fc{\E[e^{w_tu}]}{e^{Lu}\E[e^{w_t}]^u}\\
    &
    \le_{\textup{\Cref{l:herbst-quick}}} 2e^{u\E w_t} e^{CL_w^2u^2 - Lu}\E[e^{w_t}]^{-u}\\
    &\le_{\textup{Jensen}}2e^{CL_w^2u^2- Lu}. 
\end{align*}
Suppose that $L\ge 2CL_w^2$. Then, choosing $u=\fc{L}{2CL_w^2} \ge 1$ gives
\begin{align}
\label{e:hatp-rh-hatp-tail}
\hat p_t\ba{\dd{\rh_t}{\hat p_t}\ge e^L} \le 2\exp\pa{-\fc{L^2}{4CL_w^2}}.
\end{align}
Then 
\begin{align*}
    \E_{\hat p_t} \ba{\pa{\dd{\rh_t}{\hat p_t} - e^L}\vee 0}
    &=
    \int_{e^L}^\infty 
    \hat p_t\ba{\dd{\rh_t}{\hat p_t}\ge x}dx
    = \int_L^\infty \hat p_t \ba{\dd{\rh_t}{\hat p_t}\ge e^y}e^y\,dy\\
    &\le_{\eqref{e:hatp-rh-hatp-tail}} 2\int_L^\infty 
    \exp\pa{-\fc{y^2}{4CL_w^2} + y}\,dy
    \le 2\int_L^\iy \exp\pa{-\rc{4CL_w^2}(y-2CL_w^2)^2 + CL_w^2}\,dy\\
    &\le 2\sqrt{\pi}CL_w e^{L_w^2}\Phi\pa{-\fc{L-2CL_w^2}{\sqrt{2C}L_w}}\le_{\eqref{e:Phi-tail-ub}} \fc{4CL_w^2}{L-2CL_w^2} e^{-\fc{L^2}{4CL_w^2} + L}.
\end{align*}
Because $\rh_t\ll \hat p_t$,
\begin{align*}
    \rh_t\pa{\dd{\rh_t}{\hat p_t} \ge e^L} &= \E_{\hat p_t} \ba{\pa{\dd{\rh_t}{\hat p_t} - e^L}\vee 0} + e^L \hat p_t \pa{\dd{\rh_t}{\hat p_t}\ge e^L}\le \left(2 + \frac{4CL^2_w}{L-2CL^2_w}\right)\exp\pa{L-\fc{L^2}{4CL_w^2}}.
\end{align*}
When $L\ge 8CL_w^2$, this is $\le 3e^{-L^2/(8CL_w^2)}$. For $\ep\in (0,1)$, setting $L=\max\left\{8CL_w^2+2\sqrt{2C}L_w\sqrt{2\ln\pf3\ep},\,L_-\right\}$ gives the result.
\end{prf}

We can now prove the main theorem of this section.
\begin{prf}[Proof of \pref{t:annealed-ws}]
    Choose $t_1>t_0=t$. 
    Applying \pref{l:rev-KL-constant-portion} to $\nu = \hat p_{t_1}$ and $\mu = p_{t_1}$ with $L_1=e^{4(E_{t_1}+1)}$ and $L_2=4$, and noting the KL bound of \pref{l:KL} (under \pref{a:m-error}),
    \begin{align}
    \label{e:hatp-slice-p}
        \hat p_{t_1}\pa{\bc{\dd{\hat p_{t_1}}{p_{t_1}} \le 4}} 
        & \ge
        \rc{L_1} \pa{1-\fc{E_{t_1}+1}{\log L_1} - \rc 4} \ge \rc{2L_1} = \rc 2 e^{-4(E_{t_1}+1)}.
    \end{align}
    We transfer the property of ``capturing a constant portion of $p_{t_1}$'' to the TAP distribution using \pref{l:ws-je}, which gives for some $C_1$ and $C_2$ that 
    \begin{enumerate}
        \item $\hat p_{t_1}$ is a $(B_1 := e^{C_1 (L_w \sqrt{E_{t_1}+1} + L_w^2)},\rc{4L_1})$-warm start to $\rh_{t_1}$.
        \item $\rh_{t_1}\pa{\bc{\dd{\rh_{t_1}}{\hat p_{t_1}} > B_2 := e^{C_2 (L_w \sqrt{E_{t_1}+1} + L_w^2)}
        }} \le_{\eqref{e:hatp-slice-p}} \rc{8L_1B_1}$.
    \end{enumerate}
    Then 
    \begin{align*}
        \rh_{t_1}\pa{\bc{\dd{\rh_{t_1}}{p_{t_1}}\le 4B_2}}
        &\ge \rh_{t_1} 
        \pa{\bc{\dd{\rh_{t_1}}{\hat p_{t_1}}\le B_2}\cap \bc{\dd{\hat p_{t_1}}{p_{t_1}} \le 4}
        }\\ 
        &\ge 
        \rh_{t_1} \pa{\bc{\dd{\hat p_{t_1}}{p_{t_1}} \le 4}}
        - \rh_{t_1} 
        \pa{\bc{\dd{\rh_{t_1}}{\hat p_{t_1}}>B_2}}\\
        &\ge_{(1),(2)} 
        \fc{\hat p_{t_1}\pa{\bc{\dd{\hat p_{t_1}}{p_{t_1}} \le 4}} - \rc{4L_1}}{B_1} - \rc{8L_1B_1}\ge \rc{8L_1B_1}.
    \end{align*}
    Now we apply \pref{l:constant-portion-noise-ws} with $L=4B_2$ and $c=\rc{8L_1B_1}$ to obtain that 
    $\rh_{t_1}K_{t_1\to t_0}$ is a $\pa{\fc{L}{c}\exp\pa{O\pa{\rc{c_{t_1}^\rh}\ln\prc{c\ep_1}}}, \ep_1}$-warm start to $p_{t_0}$, where $c_{t_1}^{\rh}$ is the log-Sobolev constant of $\rh_{t_1}$ and constant by \pref{t:rho-ls}.
    Keeping track of just the dependence on $\ep_1$, this is $((1/\ep_1)^{O(1)}, \ep_1)$. 
    Now because $\rh_{t_0}$ and $\rh_{t_1}K_{t_1\to t_0}$ are related by Jarzynski weights by 
    \pref{l:je-reverse}, we can apply \pref{l:ws-je} with initial distribution $\rh_{t_1}$ to get that $\rh_{t_0}$ is a $(e^{O(\sqrt{\ln (1/\ep_2)})},\ep_2)$-warm start for $\rh_{t_1}K_{t_1\to t_0}$. 
    We have
    \[
\rh_{t_0}\xra{\prc{\ep}^{o(1)}, \ep}
\rh_{t_1}K_{t_1\to t_0} \xra{\prc{\ep}^{O(1)}, \ep} p_{t_0}
    \]
    Thus by transitivity (\pref{pitem:ws-preserve-warm-start-transitive} and following remark),
    $\rh_{t_0}\xra{\prc{\ep}^{O(1)}, \ep} p_{t_0}$
    By iterating one more time using the fact that $\hat p_{t_0}$ is a $(e^{O(\sqrt{\ln (1/\ep)})}, \ep)$-warm start to $\rh_{t_0}$, we also get the same fact for $\hat p_{t_0}$.
\end{prf}

\section{Proof of algorithmic result}
\label{s:alg}

We use the following discretization result for Langevin dynamics.
\begin{theorem}[Consequence of {\cite[Theorem 4]{chewi2025analysis}}]
\label{t:ld-disc}
    Let $\pi\propto e^{-V}$ be a density on $\R^n$ with log-Sobolev constant $\rls\le 1$ and such that $\gd V$ is $L$-Lipschitz. Let $\ep<1$. 
    Then 
    Langevin dynamics with initial distribution $\mu_0$, step size $h = \td \Te \pa{\fc{\ep \rls}{nL^2}}$ and number of steps $N\ge \td\Om\pf{nL^2[\log\log(\chi^2(\mu_0\|\pi)+1)\vee 1]}{\ep\rls^2}$ samples from a distribution $\hat \pi$ satisfying $\chi^2(\hat \pi\|\pi)\le \ep$. 
\end{theorem}
\begin{prf}
Apply \cite[Theorem 4]{chewi2025analysis} with $q=3$ and with
$C_{\rm LSI}=1/\rls$, targeting $R_3\le \log(1+\epsilon)$.
Since $R_2\le R_3$ and $R_2(\mu\|\pi)=\log(1+\chi^2(\mu\|\pi))$,
this gives $\chi^2(\mu\|\pi)\le \epsilon$.
\end{prf}

Regarding the other options in \pref{alg:Langevin}, we note that \pref{t:annealed-ws} also shows that the ASL distribution $\hat p_T$ is a warm start to $p_T$.
A discretization result for ASL similar to \pref{t:ld-disc} would then show that using the discretized ASL in \pref{alg:Langevin} can also give an approximate sample from the SK Gibbs distribution. Mixing for the polarized walk on the wedge has already been established in \cite[\S7]{DLSS26}.
%
\begin{prf}[Proof of \pref{t:alg}] We first show a warm-start is obtained, and then prove the sampler can sample from it.
    \ppart{Part 1 (Obtaining a warm start)}
    Let $\calF(y):= \FT(\hat m(y),y)$.
    By \pref{c:lsi-anneal-sk}, the measure $\rh_T(y) \propto e^{\calF(y) - \fc{\ve{y}^2}{2T}}$ satisfies a log-Sobolev inequality with some constant $c_T^\rh$. Note that 
    \[
\gd \log \rh_T(y) = \hat m(y) - \fc{y}{T}
    \]
    is $(L+\rc{T})$-Lipschitz. 
    Note $\gd \calF = \hat m$, so $\calF$ is $\sqrt n$-Lipschitz. Hence we have
    \begin{align}
    \label{e:anneal-Z-ub}
        \int_{\R^n} e^{\pm\calF(y) - \fc{\ve{y}^2}{2T}} \,dy 
        &\le \int_{\R^n} e^{\pm\calF(0) + \sqrt{n}\ve{y} - \fc{\ve{y}^2}{2T}}\,dy
        \le 
        \int_{\R^n}e^{\pm\calF(0) + nT - \fc{\ve{y}^2}{4T}}\,dy
        \le e^{\pm\calF(0) + nT}\pa{2\pi(2T)}^{n/2}.
    \end{align}
    Hence letting $\ga_T$ be the density of $\calN(0,TI_n)$,
    \begin{align*}
        \chi^2(\ga_T \| \rh_T) + 1
        &\le_{\eqref{e:anneal-Z-ub}} e^{\calF(0) + nT} (2\pi(2T))^{n/2}
        \rc{(2\pi T)^n} \int_{\R^n} 
        \fc{(e^{-\fc{\ve{y}^2}{2T}})^2}{e^{\calF(y) - \fc{\ve{y}^2}{2T}}} dy\\
        &\le 
        e^{\calF(0) + nT} 2^{n/2}
        \rc{(2\pi T)^{n/2}} \int_{\R^n}
        e^{-\calF(y)- \fc{\ve{y}^2}{2T}}dy
        \le_{\eqref{e:anneal-Z-ub}} e^{2nT} 2^{n/2} \rc{(2\pi T)^{n/2}} (2\pi (2T))^{n/2} \le e^{2nT}2^n.
    \end{align*}
    Hence, for fixed $T$, by \pref{t:ld-disc}, to get within constant $\chi^2$ distance takes $\td O(n\log n) = \td O(n)$ steps.
    Note that for any measures $\mu,\nu$ with $\mu\ll \nu$ that 
    \[
\mu\pa{\dd{\mu}{\nu}\ge L} \le \fc{\E_\mu[\dd{\mu}{\nu}]}{L} = \fc{\chi^2(\mu\|\nu)+1}{L},
    \]
    so the distribution of the resulting sample is a $(O(1/\ep), \ep)$ warm start to $\rh_T$ for any $\ep>0$, and hence by \pref{t:annealed-ws} and \pref{pitem:ws-preserve-warm-start-transitive} a $((1/\ep)^{O(1)}, \ep)$-warm start to $p_T$. 

    Now consider running Glauber dynamics starting from $\sign(y_T)$. Given $C$, by \pref{t:wpi-localized}, we can choose $T(\be)$ so that with probability $1-e^{-\Om(n)}$ over $A$, for $T\ge T(\be)$, with probability $\ge 1-e^{-Cn}$ over $y_T\sim p_T$, $\mu_{\be A,y_T}$ satisfies a $(\fc{c(\be)}{n}, e^{-Cn})$-WPI. 
    Note that with probability $1-e^{-\Om(n)}$ over $A$, we have $\opnorm{A}\le 3$.
    On that event, for $\si_0:=\sign(y_T)$ and any other $\si\in \{\pm 1\}$, since
$\langle y_T,\sigma_0-\sigma\rangle\ge 0$,
\[
\frac{\mu_{\be A,y_T}(\sigma_0)}
     {\mu_{\be A,y_T}(\sigma)}
=
\exp\left\{
\frac\be2\left(
\langle\sigma_0,A\sigma_0\rangle-\langle\sigma,A\sigma\rangle
\right)
+\langle y_T,\sigma_0-\sigma\rangle
\right\}
\ge e^{-3\be n}.
\]
Therefore
\[
\mu_{\be A,y_T}(\sigma_0)
\ge e^{-3\be n}2^{-n},
\qquad
\left\|\frac{d\delta_{\sigma_0}}{d\mu_{\be A,y_T}}\right\|_\infty
\le e^{3\be n}2^n.
\]
By \pref{l:ws-mix}, letting $\nu_t$ be the distribution after running Glauber for $t$ steps from $\si_0$,
    \[
\chi^2(\nu_t\| \mu_{\be A, y_T}) \le (e^{-2\fc{c(\be)}nt}+e^{-Cn})( e^{3\be n}2^n)^2 =O(1)
    \]
    by choosing $C$ large enough and $t=\Om(n^2)$. Then $\nu_t$ is a 
    $(O(1/\ep),\ep)$-warm start 
    to $\mu_{\be A, y_T}$ for any $\ep>0$. 
    
    Since $\mu_{\be A, y_T}$ is the conditional distribution of $\si \mid y_T$ under the SL process, by \pref{pitem:ws-preserve-warm-start-conditional} and \pref{pitem:ws-preserve-data-processing}, this gives a $((1/\ep)^{O(1)}, \ep)$-warm start to $\mu_{\be A}$. 
    
    \ppart{Part 2 (Sampling from warm start)}
    Set
\[
    \kappa_{\ep_2}:=\frac1n e^{-O(1/\ep_2)}.
\]
By \Cref{c:sk}, $\mu_{\be A}$ satisfies a
$(\kappa_{\ep_2},\ep_2)$-WPI. Since the output of Part 1 is a
$((1/\ep_1)^{O(1)},\ep_1)$-warm start to $\mu_{\be A}$, \Cref{l:ws-mix}
for the Glauber chain gives
\[
\TV(\nu_T,\pi)
\le
\frac12 (1/\ep_1)^{O(1)}
\sqrt{(1-\kappa_{\ep_2})^{2T}+\ep_2}
+\ep_1
\le
\frac12 (1/\ep_1)^{O(1)}
\sqrt{e^{-2\kappa_{\ep_2}T}+\ep_2}
+\ep_1.
\]
Taking $\ep_1=\ep/2$, choosing $\ep_2=\ep_1^{O(1)}$, and taking
\[
    T\gtrsim \kappa_{\ep_2}^{-1}\log\frac{(1/\ep_1)^{O(1)}}{\ep}
    =
    n e^{O(1/\ep_2)}\log\frac1\ep
\]
makes the right-hand side at most $\ep$.
\end{prf}

%% file: decomp.tex
\section{Decomposition theorem for the weak Poincar\'e inequality}

\label{s:decomp}

We show a general decomposition theorem for a WPI, which gives a way to conclude a WPI by proving WPIs for simpler distributions, and then specialize to our setting. We note that \cite{qin2025spectral} shows similar results; the proximal sampler fits into their framework as it can be thought of as a data augmentation algorithm. (See their \S3.3 describing the application to a hybrid data augmentation algorithm, and Proposition 22 for the WPI.) Namely, the proximal sampler is the Gibbs chain obtained from augmenting the state space with the Gaussian-noised observation. Their analysis is based on a sandwich structure, while we directly decompose the Dirichlet form. Note we also allow a continuous chain such as Langevin.

We consider the following setup. 

Let $(\Om_Y,\scr F_Y)$ be a measurable space with a probability measure $\mu_Y$. 

Let $(\Om_X,\scr F_X)$ be a measurable space.
Let $P:\Om_Y\times \scr F_X\to [0,1]$ be a Markov kernel (for each $y\in \Om_Y$, $P(y,\cdot)$ is a probability measure on $(\Om_X,\scr F_X)$, and for each $B\in \scr F_X$, $P(\cdot ,B)$ is a $\scr F_Y$-measurable function on $Y$). Write $P_y = P(y, \cdot)$. 
Define the measure $\mu_{X,Y}$ on the product space $(\Om_X\times \Om_Y, \scr F_X\ot \scr F_Y)$ by
\[
\int_{\Om_X\times \Om_Y} g(x,y) \,d\mu_{X,Y}(x,y) = 
\int_{\Om_Y} \int_{\Om_X} g(x,y)\,dP_y(x) \, d\mu_Y(y).
\]
(That is, this is the measure obtained by drawing $Y\sim \mu_Y$ and $X\sim P(Y,\cdot)$.)
For a measure $\mu$ on $(\Om_Y,\scr F_Y)$, let $\mu P$ denote the ``mixture'' measure defined by 
\[
\int_{\Om_X} g(x)\, d(\mu P)(x)  = \int_{\Om_Y} \int_{\Om_X} g(x)\,dP_y(x) \, d\mu(y).
\]
Informally, $\mu P = \int_Y P_y\,d\mu(y)$. Note $\mu_X:= \mu_Y P$ is the $X$-marginal of $\mu_{X,Y}$.

Suppose that each probability measure $\pi \in \set{P_y}{y\in \Om_Y}\cup \{\mu_X\}$ is associated with a reversible Markov process with generator $\sL_\pi$ with $\pi$ as its stationary distribution, and define the Dirichlet forms
\[\sE_{\pi}(f,g) := -\an{f,\sL_\pi g}_\pi.\]
Assume that for every test function $g: \Om_X \to \R$,
\begin{align}
\label{e:dir-decomp}
    \sE_{\mu_X}(g,g) 
    \ge 
    \int_{\Om_Y}  \sE_{P_y}(g,g)\, d\mu_Y(y).
\end{align}

Assume that $\mu_{X,Y}$ admits conditional distributions $\mu_{Y|X=x}$. 
We define two auxiliary Markov processes.
\begin{enumerate}
    \item We define a Markov process $\sL_Y$ on $(\Om_Y,\scr F_Y)$
with stationary distribution $\mu_Y$ as follows. Let
\[
M_Y(y,dy')
:=
\int_{\Om_X}\mu_{Y|X=x}(dy')\,dP_y(x)
\]
be the $Y$-marginal Gibbs kernel obtained by sampling
$X\mid Y=y$ and then resampling $Y'\mid X$. Define the rate-one
generator
\begin{align}
\label{e:L-Gibbs}
(\sL_Y g)(y)
&=
\int_{\Om_Y}(g(y')-g(y))\,M_Y(y,dy') \\
&=
\int_{\Om_X}\int_{\Om_Y}
(g(y')-g(y))\,d\mu_{Y|X=x}(y')\,dP_y(x).
\end{align}
    \item 
    We define a Markov process $\sL_{X,Y}$ on the product space $(\Om_X\times \Om_Y, \scr F_X\ot \scr F_Y)$ with stationary distribution $\mu_{X,Y}$ by the generator
\[
(\sL_{X,Y} g)(x,y) = \sL_{P_y} (g(\cdot,y))(x) 
+ \pa{\int_{\Om_Y} (g(x,y')-g(x,y))\,d\mu_{Y|X=x}(y')}.
\]
(This is the combination of evolving according to the Markov process for the current value of $Y$, and resampling $Y|X$ with rate 1.)
\end{enumerate}

Note that letting $\sE_{X,Y}(f,g) = -\an{f,\sL_{X,Y}g}_{\mu_{X,Y}}$, we have the decomposition
\begin{align*}
\sE_{X,Y}(g,g)
&= \int_{\Om_Y} \sE_{P_y}(g(\cdot,y),g(\cdot,y)) \,d\mu_Y(y) + 
\sE_{\lr} (g,g)\\
\text{where }
\sE_{\lr} (g,g) &= 
\rc 2
\int_{\Om_X}\int_{\Om_Y}\int_{\Om_Y} (g(x,y_1)-g(x,y_2))^2 \, d\mu_{Y|X=x}(y_1)\, d\mu_{Y|X=x}(y_2)\,d\mu_X(x).
\end{align*}

The following result is similar to \cite[Theorem 6.1]{ge2018simulated}, where we instead analyze the mixture as the $X$-marginal of a joint distribution on $(X,Y)$ as in \cite{zhou2025polynomial}. In this way, the auxiliary chain on $Y$ is exactly the Gibbs chain obtained by
sampling $X\mid Y$ and then resampling $Y\mid X$.
We will apply this to the proximal sampler, which is a Gibbs chain.

\begin{theorem}[Decomposition for weak Poincar\'e inequality with Gibbs sampling]
\label{t:wpi-decomp}
    Consider the setup above. Suppose the following hold.
    \begin{enumerate}
        \item (Weak Poincar\'e inequality for each $P_y$) Let $\sL_y = \sL_{P_y}$, and $\sE_y(f,g) = -\an{f,\sL_yg}_{P_y}$. With probability $\ge 1-\de'$ over $\mu_Y$, $\sE_y$ satisfies the weak Poincar\'e inequality (for $g_X\in L^2(\Om_X)$)
        \[
\Var_{P_y} (g_X) \le C \sE_y(g_X,g_X) + \de \osc(g_X)^2.
        \]
        \item (Weak Poincar\'e inequality for Gibbs chain) 
        Let $\sE_Y(f,g) = -\an{f,\sL_Yg}_{\mu_Y}$.
        The generator for alternating Gibbs sampling for $\mu_Y$ satisfies the weak Poincar\'e inequality (for $g_Y\in L^2(\Om_Y)$)
        \[
\Var_{\mu_Y} (g_Y) \le C_Y \sE_Y(g_Y,g_Y) + \de_Y \osc(g_Y)^2.
        \]
    \end{enumerate}
    Then
    \begin{enumerate}
        \item (Weak Poincar\'e inequality for joint distribution) Let $\sE_{X,Y}(f,g) = -\an{f,\sL_{X,Y}g}_{\mu_{X,Y}}$. Then (for $g\in L^2(\Om_X\times \Om_Y)$)  for any $\lm>0$,
        \begin{align*}
\Var_{\mu_{X,Y}}(g)
    &\le C\pa{1+(1+\lm)C_Y}\int_{\Om_Y} \sE_y(g(\cdot,y),g(\cdot,y))\,d\mu_Y(y)\\
    &\quad 
    + C_Y \pa{1+\rc{\lm}} \sE_{\lr}(g,g)
    + \ba{\pa{1+(1+\lm)C_Y}(\de + \de') + \de_Y}\osc(g)^2 \\
    &\le \max\bc{C\pa{1+(1+\lm)C_Y}, C_Y\pa{1+\rc{\lm}}} \sE_{X,Y}(g,g) + \ba{\pa{1+(1+\lm)C_Y}(\de + \de') + \de_Y}\osc(g)^2
        \end{align*}
        \item (Weak Poincar\'e inequality for mixture)
        Let $\sE_X(f,g) = -\an{f,\sL_{\mu_X}g}_{\mu_X}$. 
        Then for $g\in L^2(\Om_X)$,
        \[
        \Var_{\mu_{X}}(g)\le 
C\pa{1+C_Y} \sE_{X}(g,g)
+ 
\pa{\pa{1+C_Y}(\de + \de') + \de_Y}\osc(g)^2.
        \]
    \end{enumerate}
\end{theorem}

\begin{prf}
Given $g\in L^2(\Om_X\times \Om_Y)$, define $g_Y\in L^2(\Om_Y)$ by $g_Y(y) = \E_{X\sim P_y} g(X,y)$. 
Then variance decomposition and the assumptions give
\begin{align}
\nonumber
    \Var_{\mu_{X,Y}}(g)
    &= \int_{\Om_Y} \Var_{P_y} (g(\cdot,y))\,d\mu_Y(y) + \Var_{\mu_Y} (g_Y)\\
    &\le 
    \int_{\Om_Y} \Var_{P_y} (g(\cdot,y))\,d\mu_Y(y)
    + 
    C_Y \sE_Y(g_Y,g_Y) + \de_Y \osc(g)^2.
\label{e:wpi-decomp-1}
\end{align}

Note that drawing $Y_1\sim \mu_Y$, $X|Y_1$, and then $Y_2|X$ gives the same joint distribution of $(X,Y_1,Y_2)$ as drawing $X\sim \mu_X$, and then $Y_1,Y_2|X$ independently (Nishimori's identity). Hence
\begin{align}
\nonumber
    \sE_Y(g_Y,g_Y)
    &= \rc 2\int_{\Om_Y} \int_{\Om_X} \int_{\Om_Y}
    \pa{\E_{P_{y_2}} g(\cdot, y_2) - \E_{P_{y_1}} g(\cdot, y_1)}^2 \,d\mu_{Y|X=x}(y_2) \,dP_{y_1}(x)\, d\mu_Y(y_1)\\
\nonumber
    &= \rc 2 \int_{\Om_X}\int_{\Om_Y}\int_{\Om_Y} 
    \pa{\E_{P_{y_2}} g(\cdot, y_2) - \E_{P_{y_1}} g(\cdot, y_1)}^2 \,d\mu_{Y|X=x}(y_1) \,d\mu_{Y|X=x}(y_2)\,d\mu_X(x)\\    
\nonumber
    &\le_{\eqref{e:g-y1-y2}} \rc 2 \int_{\Om_X}\int_{\Om_Y}\int_{\Om_Y} 
    \Bigg[
\pa{1+\lm} \pa{\pa{g(x,y_1) - \E_{P_{y_1}} g(\cdot, y_1)} - \pa{g(x,y_2) - \E_{P_{y_2}}g(\cdot, y_2)}}^2\\
\nonumber
&\quad \phantom{\rc 2 \int_{\Om_X}\int_{\Om_Y}\int_{\Om_Y} }+ \pa{1+\rc{\lm}} \pa{g(x,y_2)-g(x,y_1)}^2
    \Bigg] \,d\mu_{Y|X=x}(y_1) \,d\mu_{Y|X=x}(y_2)\,d\mu_X(x)\\
\nonumber
&\le_{\eqref{e:g2-gdiff2}} 
 (1+\lm) \int_{\Om_X}\int_{\Om_Y} 
 (g(x,y) -\E_{P_y}g(\cdot,y))^2 \,d\mu_{Y|X=x}(y)\, d\mu_X(x) + \pa{1+\rc{\lm}}\sE_{\lr}(g,g)\\
 &\le 
 (1+\lm) \int_{\Om_Y} \Var_{P_y}(g(\cdot,y)) d\mu_Y(y) + \pa{1+\rc{\lm}}\sE_{\lr}(g,g)
    \label{e:wpi-decomp-2}
\end{align}
where we use the calculations
\begin{align}
\nonumber
    & \pa{\E_{P_{y_2}} g(\cdot, y_2) - \E_{P_{y_1}} g(\cdot, y_1)}^2 \\
\nonumber
    & = 
    \pa{\pa{g(x,y_1) - \E_{P_{y_1}} g(\cdot, y_1)} - \pa{g(x,y_2) - \E_{P_{y_2}}g(\cdot, y_2)} + 
    \pa{g(x,y_2)-g(x,y_1)} }^2\\
\label{e:g-y1-y2}
    &\le 
    \pa{1+\lm} \pa{\pa{g(x,y_1) - \E_{P_{y_1}} g(\cdot, y_1)} - \pa{g(x,y_2) - \E_{P_{y_2}}g(\cdot, y_2)}}^2 + \pa{1+\rc{\lm}} \pa{g(x,y_2)-g(x,y_1)}^2.
\end{align}
and the following applied to $f(y) := g(x,y) - \E_{P_y}g(\cdot,y)$ for a fixed $x$.
\begin{align}
\label{e:g2-gdiff2}
    \int_{\Om_Y} f^2 \,d\mu_{Y|X=x} &=  
    \ub{\pa{\int_{\Om_Y} f \, d\mu_{Y|X=x}}^2}{\ge0} + 
    \rc 2 \int_{\Om_Y} \int_{\Om_Y} (f(y_1)-f(y_2))^2 \,d\mu_{Y|X=x}(y_1) \,d\mu_{Y|X=x}(y_2).
\end{align}
Then \eqref{e:wpi-decomp-1} and \eqref{e:wpi-decomp-2} together give
{
\allowdisplaybreaks\begin{align}
\nonumber
    &\Var_{\mu_{X,Y}}(g) \\
    \nonumber
    &\le 
    \pa{1+(1+\lm)C_Y} \int_{\Om_Y} \Var_{P_y}(g(\cdot,y)) d\mu_Y(y) 
    + C_Y \pa{1+\rc{\lm}}\sE_{\lr}(g,g) + 
    \de_Y \osc(g)^2\\
\nonumber
    &\le \pa{1+(1+\lm)C_Y}\bc{\int_{\Om_Y} [ C\sE_y(g(\cdot,y),g(\cdot,y)) + \de \osc(g)^2]\,d\mu_Y(y) + \de' \osc (g)^2 } + C_Y \pa{1+\rc{\lm}}\sE_{\lr}(g,g) + \de_Y \osc (g)^2\\
    \label{e:wpi-before-max}
    &\le C\pa{1+(1+\lm)C_Y}\int_{\Om_Y} \sE_y(g(\cdot,y),g(\cdot,y))\,d\mu_Y(y)
    + C_Y \pa{1+\rc{\lm}} \sE_{\lr}(g,g)
    + \ba{\pa{1+(1+\lm)C_Y}(\de + \de') + \de_Y}\osc(g)^2 \\
\nonumber
    &\le \max\bc{C\pa{1+(1+\lm)C_Y}, C_Y\pa{1+\rc{\lm}}} \sE_{X,Y}(g,g) + \ba{\pa{1+(1+\lm)C_Y}(\de + \de') + \de_Y}\osc(g)^2
\end{align}}
This shows the first part. To obtain the weak Poincar\'e inequality for the mixture, given $g\in L^2(\Om_X)$, view it as a function in $L^2(\Om_X\times \Om_Y)$ that is constant in $y$. We have $\sE_{\lr}(g,g) = 0$, so we can take $\lm\to 0^+$ in the above. Then \eqref{e:wpi-before-max} together with the inequality \eqref{e:dir-decomp} gives the result.
\end{prf}
We can directly apply this theorem to obtain a WPI from a WPI from the localized distributions and the proximal sampler.

\newcommand{\wpigdps}[0]{    Given a distribution $\mu$ on $\Om$ where $\Om = \{\pm 1\}^n$ or $r\cdot \bbS^{n-1}$, suppose the following hold. Let $p_t$ be the distribution of $y_t$ under \eqref{e:SL}.
    \begin{enumerate}
        \item (WPI for localized distributions) With probability $1-\de'$ over $y\sim p_{T}$, the tilted distribution $\mu_y$ satisfies the weak Poincar\'e inequality (for $g\in L^2(\Om)$)
\[
    \Var_{\mu_y}(g) \le C\loc \sE_{\mu_y}(g,g) + \de\loc \osc(g)^2,
\]
where the Dirichlet form is that of Glauber on $\{\pm 1\}^n$ or Langevin on $r\cdot\bbS^{n-1}$, respectively.
        \item (WPI for proximal sampler) 
        The proximal sampler $\Kprox{T, \iy}$ satisfies a weak Poincar\'e inequality (for $g\in L^2(\R^n)$)
        \[
\Var_{p_{T}}(g) \le C\prox \sE\prox(g,g) + \de\prox\osc(g)^2.
        \]
    \end{enumerate}
    Then $\mu$ satisfies the weak Poincar\'e inequality
    \[
\Var_\mu(g) \le      
C\loc\pa{1+C\prox} \sE_{\mu}(g,g)
+ 
\pa{\pa{1+C\prox}(\de\loc + \de') + \de\prox}\osc(g)^2.
        \]
}

\begin{corollary}[Weak Poincar\'e inequality for Glauber/Langevin from localized distributions and proximal sampler]
\label{c:wpi-gd-ps}
\wpigdps
\end{corollary}
\begin{prf}
    By construction of the SL process---$Y\sim p_{T}$, then $X \sim \mu_y$---the marginal law of $X$ is exactly $\mu$. 
    The Gibbs sampler given by \eqref{e:L-Gibbs} is exactly the (continuous-time version of the) proximal sampler given by $\Kprox{T, \iy}$. 
    We note that Glauber dynamics and Langevin dynamics both satisfy the Dirichlet form decomposition \eqref{e:dir-decomp} (see e.g., \cite[Proposition 7]{lee2024convergence}). 
    The result then follows from \pref{t:wpi-decomp}.
\end{prf}
The same result holds when using the proximal sampler for a distribution with a density in $\R^n$, replacing the localized distributions with $\mu_{y}^T(x) \propto \mu(x)e^{-\fc{\ve{x}^2}{2T} + \an{y,x}}$ (as this is the form of the posterior distributions).

%% file: appendix.tex
\section{Log-Sobolev inequality under SDE evolution}\label{a:lsi-under-sde}
The following is a generalization of \cite[Lemma 16]{liang2025characterizing}.
\begin{lemma}[LSI under SDE with Lipschitz drift]
\label{l:lsi-Lip-SDE}
    Suppose that $\rh_0$ satisfies a log-Sobolev inequality with constant $\al_0$. Consider the SDE
    \[
dy_t = f_t(y_t)dt + dB_t, \quad y_0\sim \rh_0,
    \]
    where $f_t$ is $L(t)$-Lipschitz, and $L(t)$ is Riemann integrable over $[0,T]$.
    Then for any $t\in [0,T]$, the distribution $\rh_t$ of $y_t$ satisfies a log-Sobolev inequality with constant 
    \[
\pa{e^{\int_0^t 2L_sds}\al_0^{-1} + \int_0^t e^{2\int_s^t L_rdr}ds}^{-1}.
    \]
\end{lemma}
If $\rh_0$ is a point mass, then $\al_0=\iy$ and the first term drops out.
\begin{prf}
    Consider the discretization with step size $h$,
    \begin{align*}
        X_{t+h}^{(h)} &= X_t^{(h)} + hf_t(X_t^{(h)}) + \sqrt{h}Z_t, & Z_t&\sim \calN(0,I_n).
    \end{align*}
    Let $\rh_t^{(h)}$ be the distribution of $X_t^{(h)}$. 
    Note
    \begin{align*}
        \rh_{t+h}^{(h)} &= (\Id + hf_t)_{\#} \rh_t * \mathcal N(0, hI_n).
    \end{align*}
    
    Let $C_0=\al_0^{-1}$ and $C_t^{(h)}=\rls(\rh_t^{(h)})^{-1}$. 
    Note $\Id + hf_t$ is $(1+hL(t))$-Lipschitz.
    Hence we have, using (1) the pushforward of a distribution satisfying a LSI with constant $c$ by a $L$-Lipschitz map satisfies a LSI with constant $\fc{c}{L^2}$ \cite[Remark 7]{chafai2004entropies} and (2) the inverse LSI constant of a convolution of distributions is at most the sum of their inverse LSI constants \cite[Corollary 3.1]{chafai2004entropies},
    \begin{align*}
        \rls((\Id + hf_t)_{\#} \rh_t) &\ge \fc{\rls(\rh_t)}{(1+hL(t))^2} = \rc{(1+hL(t))^2 C_t^{(h)}}\\
        \rls((\Id + hf_t)_{\#} \rh_t * \mathcal N(0, hI_n))^{-1}
        &\le (1+hL(t))^2 C_t^{(h)} + h.
    \end{align*}
    Hence
    \[
C_{kh}^{(h)}\le \prod_{j=0}^{k-1} (1+hL_{jh})^2 C_0 + h\ba{\sum_{i=1}^{k-1}\prod_{j=i}^{k-1}(1+hL_{jh})^2}.
    \]
    Noting that $\rh_{t}^{(t/N)}\to \rh_t$ weakly as $N\to \iy$, we can take the limit of the functinoal inequalities for $\rh_t^{(t/N)}$ to get 
    \[
C_t = \lim_{N\to \iy} C_t^{(t/N)}\le e^{2\int_0^t L_sds}C_0 + \int_0^t e^{2\int_s^t L_rdr}ds. \qedhere
    \]
\end{prf}

\section{Auxiliary analytic lemmata on Wiener spaces}\label{a:wiener}

\subsection{Tensor decomposition for cylindrically reweighed Wiener measures}

We first give a complete proof demonstrating that the Wiener measure $\gamma$ on $(E,\calB(E))$ reweighed by $e^{V_k(B)}/Z_k$ can be decomposed into a tensor product of measures as $\nu_k = (T_k)_{\#}\left(\tilde{\nu}_k \otimes \gamma_{<k}\right)$, where the former corresponds to the push-forward induced by a Lipschitz reweighing of the first $k$ Gaussian coordinates of $\gamma$ and the latter is the ``tail'' corresponding to an infinite-dimensional centered Gaussian process on $E$.
\begin{lemma}[Tensorization of the Lipschitz $k$-reweighed Wiener measure]\label{lem:wiener-lipschitz-reweigh-tail-decomposition}
    Let $\gamma$ be the Wiener measure on $(E,\calB(E))$ and $V$ be a $K$-Cameron--Martin Lipschitz function $V: E \to \R$. Define the measures $\nu_k, \tilde{\nu}_k$ and $\gamma_{>k}$, and the maps $P_k: E \to \calH$ and $R_k: E \to E$ as in \pref{thm:lsi-rhT-lipschitz}. Then, 
    \[
        \left(\pi_k,R_k\right)_{\#}\gamma = \gamma_k \otimes \gamma_{>k}\,,
    \]
    and
    \[
        \left(\pi_k,R_k\right)_{\#}\nu_k = \tilde{\nu}_k \otimes \gamma_{>k}\,.
    \]
\end{lemma}
\begin{prf}
    We prove the claims in sequence using the fact that $\E[\ell(B)X_i(B)] = \ell(e_i)$\footnote{\,The proof for this follows by using the It\^o isometry in conjunction with the definition $X_h(B) = \int_0^T \frac{d}{ds}h(s)dB_s$ to conclude that $\E[X_h(B)B_t]=h(t) = \ell_t(h)$ by first choosing $\ell$ to be the evaluation map, and then extending to any continuous linear functional $\ell$ using the Markov-Riesz-Kakutani representation theorem to conclude $\E[X_h(B)\ell(B)]= \ell(h)$. Finally, set $h = e_i$.}. 
    \ppart{Proof of independence for Wiener decomposition} Note that $\pi_k(B) := \left(X_1(B),\dots,X_k(B)\right)$. Since $\pi_k(B) \sim \calN(0,\Id_k)$ and \pref{lem:gamma-tail-gaussian} implies that $\gamma_{>k} := R_kB$ is also a (centered) Gaussian process on $E$, it suffices to show that $\E[X_i(B)\ell(R_k(B))] = 0$ for any continuous linear functional $\ell : E \to \R$. So,
    \allowdisplaybreaks
    \begin{align*}
        \E\left[\ell(R_k(B))X_i(B)\right] &= \E\left[\left(\ell(B) - \sum_{j=1}^kX_j(B)\ell(e_j)\right)X_i(B)\right] = \E\left[\ell(B)X_i(B)\right] - \sum_{j=1}^k \E\left[X_j(B)X_i(B)\right]\ell(e_j) \\
        &= \ell(e_i) - \ell(e_i) = 0\,.
    \end{align*}
    \ppart{Proof of independence for reweighed decomposition} By the computations in \pref{thm:lsi-rhT-lipschitz}, note that
    \[
        d\nu_k(B) := \frac{e^{V_k(B)}}{Z_k}d\gamma(B) = \frac{e^{h(\pi_k(B))}}{Z_k}d\gamma(B)\, , 
    \]
    for some $\calF(X_1,\dots,X_k)$-measurable function $h: \R^k \to \R$. This immediately implies, since $\gamma_k := (\pi_k)_{\#}\gamma$, that $Z_k = \int_E e^{h(\pi_k(B))}d\gamma(B) = \int_{\R^k} e^{h(a)}d\gamma_k(a)$. Furthermore, the following holds under the push-forward $(\pi_k)_{\#}\gamma = \gamma_k$ for any bounded measurable function $F: E \times \R^k \to \R$ 
    \allowdisplaybreaks
    \begin{align*}
        &\E_{(a,x)\sim (\pi_k,R_k)_{\#}\nu_k}\left[F(a,x)\right] = \int_{a \in \R^k}\int_{x \in E}F(a,x)d\left[(\pi_k,R_k)_{\#}\nu_k\right](a,x) \\
        &= \int_{B \in E}F(\pi_k(B), R_k(B))d\nu_k(B) = \frac{1}{Z_k}\int_{B \in E}F(\pi_k(B), R_k(B))e^{h(\pi_k(B)}d\gamma(B) \\
        &=_{\gamma = \gamma_k \ot \gamma_{>k}} \int_{B \in E}F(\pi_k(B), R_k(B))\frac{e^{h(\pi_k(B))d\gamma_k}}{Z_k}d\gamma_{>k} \\ 
        &= \int_{B \in E}F(\pi_k(B), R_k(B)) d\tilde{\nu}_k(\pi_k(B)) d\gamma_{> k}(B) = \int_{a \in \R^k}\int_{x \in E}F(a,x)d\tilde{\nu}_k(a)d\gamma_{>k}(x) \\
        &= \E_{(a,x) \sim \tilde{\nu}_k\ot \gamma_{>k}}\left[F(a,x)\right]\,.
    \end{align*}
    This immediately implies that $(\pi_k,R_k)_{\#}\nu_k = \tilde{\nu}_k \ot \gamma_{>k}$.
\end{prf}

We now prove, for completeness, that the push-forward measure $\gamma_{> k} := (R_k)_{\#}\gamma$ is a centered Gaussian process on $E$ with Cameron-Martin space $H^\perp_{k}$. Here, $\gamma$ is the Wiener measure on $(E,\calB(E))$ and $R_{k}(B) := B - \sum_{i=1}^k X_i(B)e_i$.
\begin{lemma}[Gaussian tail decomposition of the Lipschitz $k$-reweighed Wiener measure]\label{lem:gamma-tail-gaussian}
    Consider the Wiener measure $\gamma$ on $(E,\calB(E))$ where $E := C^0\left([0,T], \R^n\right)$ with the Cameron-Martin space $\calH := \left\{ B \in E \mid B(0)= 0\,\,\gamma\text{-a.s.},\, \int_0^T \left(\frac{d}{dt}B\right)^2\,dt < \infty\right\}$. Let $\{e_i\}_{i \in \N}$ be an orthonormal basis for $\calH$. Then, for any $k \in \Z$, the map $R_k: E \to E$ defined as $R_k(B) := B - \sum_{i=1}^k X_i(B)e_i$ induces a push-forward measure $\gamma_{>k} := (R_k)_{\#}\gamma$ which is a centered Gaussian process on $E$.
\end{lemma}
\begin{prf}
    A probability measure $\mu$ on a Banach space $E$ is a centered Gaussian measure if, for every finite collection of continuous linear functions $\{\ell_i \in E^*\}_{i \in [m]}$, it is the case that $\left(\ell_1(Y),\dots,\ell_m(Y)\right) \sim \calN(0,\Sigma_m)$ when $Y \sim \mu$. So, fix $\ell \in E^*$ and observe that for all $B \sim \gamma$,
    \allowdisplaybreaks
    \begin{align*}
        \ell\left(R_k(B)\right) &= \ell\left(B - \sum_{i=1}^k X_i(B)e_i\right) = \ell(B) - \sum_{i=1}^kX_i(B)\ell(e_i) \\
        &=_{\text{Riesz representation}} X_{i^*\ell}(B)- \sum_{i=1}^k X_i(B)\an{i^*\ell, e_i}_H \\
        &= X_{i^*\ell - P_{H^k} i^*\ell}(B) = X_{P_{H^\perp_k}i^*\ell}(B) \overset{d}{=} \calN\left(0,\norm{P_{H^\perp_k}(i^*\ell)}^2_H\right)\,,
    \end{align*}
    where $i: H \hookrightarrow E$ is the continuous embedding of the Cameron--Martin space into $E$, $i^* : E^* \to H$ is its adjoint, and $P_{H^\perp_k}: H \to H^\perp_k$ is the standard orthogonal projector\footnote{\,The statement $\ell(B) = X_{i^*\ell}(B)$ is an equality in $L^2(\gamma)$.}. Now, taking the continuous linear functionals $\ell_1,\dots,\ell_m$ and any $a=(a_1,\dots,a_m) \in \R^m$ we obtain that
    \[
        \sum_{r=1}^m a_r\ell_r(R_k(B)) = \sum_{r=1}^m a_r X_{P_{H^\perp_k}i^*\ell_r}(B) = X_{P_{H^\perp_k}(\sum_r a_r i^*\ell_r)}(B) \overset{d}{=} \calN\left(0,\norm{P_{H^\perp_k}\left(\sum_r a_ri^*\ell_r\right)}^2_H\right)\,. 
    \]
    Note, furthermore, that given $B \sim \gamma$ the covariance of the tail $\gamma_{>k}$ can be computed with $Y = R_k(B) \sim \gamma_{>k}$ for any two continuous linear functionals $\ell_a,\ell_b \in E^*$ as
    \allowdisplaybreaks
    \begin{align*}
        \E_{\gamma_{> k}}\left[\ell_a(Y)\ell_b(Y)\right] &= \E_\gamma[\ell_a(R_k(B))\ell_b(R_k(B))] \\
        &= \E_\gamma\left[X_{P_{H^\perp_k}i^*\ell_a}(B)X_{P_{H^\perp_k}i^*\ell_b}(B)\right] \\
        &=_{\text{Riesz representation\,+\,$L^2$-isometry}} \E_\gamma\left[\left(\sum_{r > k}\alpha_r X_r(B)\right)\left(\sum_{r' > k}\beta_{r'}X_{r'}(B)\right)\right]\,.
    \end{align*}
    Note that finite truncations of the sums converge in $L^2(\gamma)$ to $X_{P_{H^\perp_k}i^*\ell_a}$ and $X_{P_{H^\perp_k}i^*\ell_b}$ respectively, and so the expectation and sum may be exchanged. This gives
    \allowdisplaybreaks
    \begin{align*}
        \E_\gamma\left[\left(\sum_{r > k}\alpha_r X_r(B)\right)\left(\sum_{r' > k}\beta_{r'}X_{r'}(B)\right)\right] &= \sum_{r,r' > k} \alpha_r\beta_{r'}\E\left[X_r(B)X_{r'}(B)\right] = \sum_{r > k}\alpha_r\beta_r \,,
    \end{align*}
    which implies that the covariance is $\Id - P_{H_k}$ in the $H$-basis and, therefore, diagonal.
\end{prf}

\subsection{Gaussian isoperimetry in Wiener space}
Remarkably, while analysis in the Cameron--Martin space seems too weak to conclude statements about measurable sets $A \in \calB(E)$ since $\gamma(H) = 0$ a.s., one can extend statements from this space to sets with positive density. Often, this power comes via Gaussian isoperimetry that is available with dimension-free constants on infinite-dimensional Banach spaces (see~\cite[\S 3.3]{hairer2026advanced}). In this spirit, we give a proof of uniform integrability for exponential moments of Cameron--Martin Lipschitz functions of Gaussians on abstract Wiener space using Borell's inequality~(\pref{t:borell}).

\begin{lemma}[Uniform integrability of $e^{pV}$]\label{lem:infinite-dim-gaussian-ui}
    Let $V: E \to \R$ be $K$-Lipschitz in the Cameron--Martin norm. Then, given the Wiener space $(\calE, \calB(E), \gamma)$ with Cameron--Martin space $\calH$, 
    \[
        \E_\gamma\left[e^{pV}\right] < \infty\, ,
    \]
    for every $p \ge 0$.
\end{lemma}
\begin{prf}
    We will apply \pref{t:borell} to a radius $r$ ball in the Cameron--Martin space around the set of all paths $x \in E$ that capture the median behavior of $V$. Fix $m \in \R$ so that $\P_{x\sim\gamma}\{V(x) \le m\} \ge \frac{1}{2}$ and $\P_{x\sim\gamma}\{V(x) \ge m\} \ge \frac{1}{2}$ and let $\calA := \{ x \in E \mid V(x) \le m\}$. Such a $m$ exists by the continuity of $\gamma$ and the continuous mapping theorem preserving this continuity under the action of the Lipschitz pushforward $V$. Now, let $B_\eps := \{ h \in \calH \mid \norm{h}_{\text{CM}} \le \eps \} \subset \calH$ and observe that
    \[
        \calA + B_\eps \subseteq \{ x \mid V(x) \le m + K\eps \}\,.
    \]
    Applying \pref{t:borell} in conjunction with the fact that $\phi(0) = \frac{1}{2}$ gives
    \allowdisplaybreaks
    \begin{align*}
        \P_{x\sim\gamma}\left\{V(x) \le m + K\eps\right\} \ge \P_{x\sim\gamma}\left\{x \in \calA + B_\eps\right\} \ge_{\text{\pref{t:borell}}} \phi(0 + \eps) = \phi(\eps)\,.  
    \end{align*}
    This gives that
    \[
        \P_{x\sim\gamma}\left\{V(x) - m > K\eps\right\} \le 1 - \phi(\eps) \le e^{-\eps^2/2}\,.
    \]
    Set $\eps = t/K$ for any $t > 0$ and one obtains a sub-Gaussian tail around $m$. An exactly analogous argument gives the lower tail using the set $\calA = \{ x \mid V(x) \ge m \}$. Therefore, for any $t > 0$
    \[
        \P_{x\sim\gamma}\left\{|V(x) - m| > t\right\} \le 2e^{-\frac{t^2}{2K^2}}\,.
    \]
    Now, setting $Y := V - m$, the integrability condition follows by making the sub-Gaussian tail fight the exponential growth as
    \allowdisplaybreaks
    \begin{align*}
        \E_\gamma\left[e^{pV}\right] &= e^{pm}\E_\gamma\left[e^{pY}\right] = e^{pm}\left(\E_\gamma\left[\1_{Y \le 0}e^{pY}\right] + \E_\gamma\left[\1_{Y >0}e^{pY}\right] \le 1 + \E_\gamma\left[\1_{Y > 0}e^{pY}\right]\right) \\
        &= e^{pm}\left(1 + \left(1 + p\int_0^\infty e^{pt}\P_{x\sim\gamma}\left\{Y(x) > t\right\}dt\right)\right) \le e^{pm}\left(2 + p\int_0^\infty e^{pt}e^{-\frac{t^2}{2K^2}}dt\right) \\
        &\le_{\text{HS transform}} e^{pm}\left(2 + pe^{\frac{p^2K^2}{2}}\int_{-\infty}^\infty e^{-\frac{(t-pK)^2}{2K^2}}\,dt\right) = e^{pm}\left(2 + \sqrt{2\pi}pKe^{\frac{p^2K^2}{2}}\right) < \infty\,. \qedhere
    \end{align*}
\end{prf}

\begin{remark}\label{rem:lemma-ui}
    Note that \pref{lem:infinite-dim-gaussian-ui} implies that $V \in L^q(\gamma)$ for any finite $q$ and so $\E_{y\sim\gamma}[V(y)] < \infty$. Now, observe by Jensen's inequality and the non-negativity of $|\cdot|$ that
    \[
        \left|\E_{y\sim\gamma}[V(y)] - m\right| \le \E_{y\sim\gamma}\left|V(y) - m\right| = \int_0^\infty \P_{y\sim\gamma}\left\{|V(y)-m|>t\right\}\,dt \le 2\int_{0}^\infty e^{-t^2/2K^2}\,dt \le \sqrt{2\pi}K\,.
    \]
    The proof of \pref{lem:infinite-dim-gaussian-ui} also implies that
    \[
        \E_{y\sim\gamma}\left[e^{t(V-m)}\right]\le 2e^{C_1t^2K^2 + C_2Kt}\, ,
    \]
    for absolute constants $C_1,C_2>0$. Combining the above with the bound on $\left|\E_{y\sim\gamma}[V(y)] - m\right|$ by a symmetrization trick and some algebra yields that
    \[
        \E_{\gamma}\left[e^{t(V - \E_\gamma[V])}\right] \le 2e^{CK^2t^2}\, ,
    \]
    for all $t \ge 1$ for some sufficiently large absolute constant $C>0$.
\end{remark}

\subsection{Closure under graph norm for titled measures in path-space} We prove closability for tilted measures \(d\nu^V=Z_V^{-1}e^Vd\gamma\) where \(V\) is finite and CM-Lipschitz below. For such measures, \pref{lem:infinite-dim-gaussian-ui} gives the integrability needed to transfer closability from \(\gamma\) to \(\nu^V\).

\begin{lemma}[Closability under CM-Lipschitz Gaussian tilts]
\label{l:closability-cm-tilts}
Let $V:E\to\R$ be $\ga$-measurable, $\ga$-a.s. finite, and
$K$-CM-Lipschitz. Set $d\nu^V := Z_V^{-1}e^V\,d\ga$ and $Z_V:=\int_E e^V\,d\ga$. Then $Z_V\in(0,\infty)$ and the operator
\[
    D_H:C_b^1(E)\subset L^2(\nu^V)
    \longrightarrow
    L^2(\nu^V;\calH)
\]
is closable. Hence $\calD^{1,2}(\nu^V)$ is well-defined as the domain of
its closure under the graph norm.
\end{lemma}
\begin{prf}
    Let $w:=Z_V^{-1}e^V$, so $d\nu^V=w\,d\ga$. Since $V$ is CM-Lipschitz,
    so is $-V$. By \pref{lem:infinite-dim-gaussian-ui},
    \[
        e^{qV},e^{-qV}\in L^1(\ga)\,,
    \]
    for every finite $q > 0$. In particular, $Z_V\in(0,\infty)$, $w>0$ $\ga$-a.s., and for every
    finite $q>0$,
    \[
        w^q,w^{-q}\in L^1(\ga).
    \]
    Let $F_j\in C_b^1(E)$ be a Cauchy sequence in the sense that
    \[
        F_j\to_{L^2(\nu^V)} 0\,, \qquad D_HF_j\to_{L^2(\nu^V;\calH)}G.
    \]
    We prove $G=0$. Choose $p\in(1,2)$ and note by H\"older's inequality,
    \[
    \begin{aligned}
        \|F_j\|_{L^p(\ga)}^p &= \int_E |F_j|^p\,d\ga = \int_E \bigl(|F_j|^2w\bigr)^{p/2}w^{-p/2}\,d\ga \\
        &\le \left(\int_E |F_j|^2w\,d\ga\right)^{p/2}\left(\int_E w^{-p/(2-p)}\,d\ga\right)^{(2-p)/2}.
    \end{aligned}
    \]
    The second factor is finite because $w^{-q}\in L^1(\ga)$ for all finite $q$ and, so $F_j\to_{L^p(\ga)}0$. The same estimate, with $|F_j|$ replaced by $\|D_HF_j-G\|_\calH$, gives
    \[
        D_HF_j\to_{L^p(\ga;\calH)} G\,.
    \]
    The Malliavin derivative is closable in $L^p(\ga)$ for $p>1$; see \cite[\S5.1, Proposition 5.7]{tubaro2025introduction}. Therefore, $G=0$ $\ga$-a.s. Since $\nu^V$ is also equivalent to $\ga$, $G=0$ $\nu^V$-a.s. This proves closability in $L^2(\nu^V)$.
\end{prf}

\subsection{Convolution for truncated Lipschitz functions with mollifiers} In this section we collect some standard properties of mollifiers that are well known in distribution theory and functional analysis\footnote{\,See \cite[Mollifiers, \S 4.4]{brezis2011functional} and \cite[\S 8.3]{brezis2011functional} for a collection of basic properties and explicit choice of mollifiers, as well as their analytic properties under convolution.}. The main point is to state these for mollifications of truncated Lipschitz functions that are not necessarily differentiable, and we give proofs of these properties for completeness.

\begin{lemma}[Properties of mollified Lipschitz truncations]\label{lem:mollifiers}
Let $\mu$ be a probability measure on $\left(\R^n, \calB(\R^n)\right)$, and $\eta\in C_c^\infty(\mathbb R^n)$ be the candidate mollifier such that $\eta\ge 0$,
$\mathrm{supp}\left(\eta\right)\subseteq B(0,1)$, and $\int_{\mathbb R^n}\eta(x)\,dx=1$. For $\varepsilon>0$, set
\[
    \eta_\varepsilon(x):=\varepsilon^{-n}\eta(x/\varepsilon)\,,
\]
to be the mollifier. Let $F:\mathbb R^n\to\mathbb R$ be bounded and $A$-Lipschitz, and define
\[
    F_\varepsilon:=\eta_\varepsilon*F\,.
\]
Then, the mollified Lipschitz truncation functions satisfy the following properties:
\begin{enumerate}[itemsep=0.2em]
    \item $F_\varepsilon\in C_b^\infty(\mathbb R^n)$.
    \item $\|F_\varepsilon\|_\infty\le \|F\|_\infty$.
    \item $\|F_\varepsilon\|_{\mathrm{Lip}}\le \|F\|_{\mathrm{Lip}}\le A$. 
    \item $\|F_\varepsilon-F\|_\infty\le A\varepsilon.$ In particular, $F_\varepsilon\rightarrow_{L^2(\mu)} F$.
    \item Suppose, in addition, that $\mu$ is absolutely-continuous with respect to the Lebesgue measure and $F$ is differentiable $\mu$-a.e., with pointwise gradient $G=\nabla F$ on its differentiability set, and $\|G\|_{2}\le A$ $\mu$-a.e. Then,
    \[
        \nabla F_\varepsilon\rightarrow_{L^2(\mu;\R^n)} G\,.
    \]
\end{enumerate}
\end{lemma}
\begin{prf}
    We explicitly write out the convolution of the mollified Lipschitz truncation functions and use the properties of the mollifier to prove the statements.
    \ppart{$F_\eps \in C^\infty_b(\R^n)$} Note that 
    \[
        F_\eps(x) := \eta_\eps* F = \int_{y \in \R^n}\eta_\eps(x-y)F(y)dy\,, 
    \]
    and so by the smoothness of $\eta_\eps$ in conjunction with the dominated convergence theorem applied to the dominating function $\norm{F}_\infty \norm{\partial_i \eta_\eps}_\infty \1_{A}(y)$ for some bounded set $A \subset \R^n$,
    \[
        \partial_i F_\eps(x) = \int_{y \in \R^n}\left(\partial_i \eta_\eps(x-y)\right)F(y)dy\,. 
    \]
    Given any fixed $k \in \N$, applying the above along with a triangle inequality, to any multi-index $\alpha \in [n]^k$ gives that
    \allowdisplaybreaks
    \begin{align*}
        \left|\partial^\alpha F_\eps(x)\right| &= \left|\int_{y \in \R^n}\left(\partial_\alpha \eta_\eps(x-y)\right)F(y)dy\right| \le \norm{F}_\infty \int_{y \in \R^n}  \left|\partial_\alpha \eta_\eps(x-y)\right|\,dy = \norm{F}_\infty\norm{\partial_\alpha\eta_\eps}_{L^1} \le C\,,
    \end{align*}
    since $F$ is bounded and $\eta_\eps \in C^\infty_c(\R^n)$. 
    \ppart{$\norm{F_\eps}_\infty \le \norm{F}_\infty$} Using the facts that $\eta_\eps \ge 0$ and $\int_{\R^n} \eta_\eps(x) dx = \eps^{-n}\int_{\R^n} \eta(x/\eps) dx = 1$,
    \allowdisplaybreaks
    \begin{align*}
        \left|F_\eps(x)\right| &= \left|\int_{y\in\R^n}\eta_\eps(x-y)F(y)\,dy\right| \le \norm{F}_\infty\int_{y}\eta_\eps(x-y)dy = \norm{F}_\infty\,.
    \end{align*}
    \ppart{$\norm{F_\eps}_{\mathrm{Lip}} \le \norm{F}_{\mathrm{Lip}}$} Note that
    \allowdisplaybreaks
    \begin{align*}
        \left|F_\eps(x + h) - F_\eps(x)\right| &= \left|\int_{y \in \R^n}\eta_\eps(y)\left(F(x-y+h) - F(x-y)\right)\,dy\right| \\
        &\le \int_{y\in\R^n}\eta_\eps(y)\left|F(x-y+h) - F(x-y)\right|\,dy \\
        &\le \norm{F}_{\mathrm{Lip}}\norm{h}_2 \int_{y\in\R^n}\eta_\eps(y)dy = \norm{F}_{\mathrm{Lip}}\norm{h}_2 \le A \norm{h}_2\,.
    \end{align*}
    \ppart{$\norm{F_\eps-F}_\infty \le A\eps$ and $F_\eps \rightarrow_{L^2(\mu)} F$} For the uniform convergence estimate, using the fact $\int_y\eta_\eps(y)\,dy = 1$ yields
    \allowdisplaybreaks
    \begin{align*}
        \norm{F_\eps - F}_\infty &= \sup_{x \in\R^n}|F_\eps(x) - F(x)| = \sup_{x\in\R^n}\left|\int_{y\in\R^n}\eta_\eps(y)F(x-y)\,dy - F(x)\int_{y\in\R^n}\eta_\eps(y)dy \right| \\
        &\le  \sup_{x\in\R^n} \int_{y\in\R^n}\eta_\eps(y)\left|F(x) - F(x-y)\right|\,dy \le A\int_{y\in\R^n}\norm{y}_2 \eta_\eps(y)\,dy \le \eps A\,,
    \end{align*}
    where the last bound uses the fact that $\mathrm{supp}(\eta_\eps) \subseteq B(0,\eps)$. This also immediately implies that
    \[
        \norm{F_\eps - F}^2_{L^2(\mu)} = \int_{y \in \R^n} |F_\eps(y) - F(y)|^2d\mu(y) \le \norm{F_\eps - F}^2_\infty \le A^2\eps^2\, ,
    \]
    which goes to $0$ as $\eps \to 0$, and the dominated convergence theorem applies with the dominating function $\norm{F}_\infty\1_{A}(y)$ for some bounded set $A \subseteq \R^n$.
    \ppart{$\nabla F_\eps \to_{L^2(\mu;\R^n)} G$} We will apply a change-of-variables to evaluate the directional derivative of every coordinate of $F_\eps(x)$ in every direction $z \in \R^n$ and show it converges to $G$ in the same direction by applying the dominated convergence theorem. The result then follows by another application of the dominated convergence theorem using the fact that $\sup_{x}\norm{\nabla F_\eps(x) - G(x)}^2_2 \le 2\norm{F}^2_{\mathrm{Lip}} \le 4A^2$. Specifically, for $i \in [n]$ with $y = x - \eps z$ where $x$ is in the differentiability set of $F$, we have 
    \allowdisplaybreaks
    \begin{align*}
        \partial_i F_\eps(x) &= \partial_i \int_{y}\eta_\eps(x-y)F(y)\,dy = \int_y \partial_i \eta_\eps(x-y)F(y)\,dy = \int_z \partial_i\eta_\eps(\eps z)F(x-\eps z)\,dy \\
        &= \frac{1}{\eps}\int_z \partial_i\eta(z)F(x-\eps z)\,dy =_{\int_z \partial_i\eta(z)\,dz = 0} \frac{1}{\eps}\int_z \partial_i\eta(z)\left(F(x-\eps z) - F(x)\right)\,dz\,.  
    \end{align*}
    Since $F$ is differentiable at $x$, note that $\lim_{\eps \to 0}\frac{1}{\eps}(F(x-\eps z)-F(x)) \to -\an{\nabla F(x), z} = -\an{G(x), z}$. Furthermore, since $\norm{F}_{\mathrm{Lip}} \le A$, the function $|\partial_i \eta(z)|A\norm{z}_2$ is supported on a compact set and, therefore, integrable. Consequently, by the dominated convergence theorem, the following pointwise convergence holds
    \[
        \partial_i F_\eps(x) \to -\int_z\partial_i\eta(z)\an{G(x),z}\,dz\,.  
    \]
    Once again, the facts that $\mathrm{supp}(\eta)$ is compact and $\eta(z) \in C^\infty_b(\R^n)$ imply  that $\int_z \partial_i(z_j\eta(z))\,dz = 0$ which immediately gives
    \[
        0 = \delta_{ij}\eta(z) + z_j\partial_i\eta(z)\,.
    \]
    Integrating both sides and using the fact that $\int_z \eta(z)\,dz = 1$ since $\eta$ is a mollifier yields
    \[
        \int_i z_j\partial_i\eta(z)\,dz = -\delta_{ij}\,.
    \]
    Combining the above with the linearity of the inner-product immediately gives the pointwise convergence $\partial_iF_\eps(x) \to G_i(x)$ at every coordinate $i \in [n]$. Boosting this to $L^2(\mu;\R^n)$ convergence follows by another use of the dominated convergence theorem on top of the pointwise convergence with the observation that
    \[
        \norm{\nabla F_\eps(x) - G(x)}^2_2 \le 4A^2\, ,
    \]
    at every point $x$ in the differentiability set.
\end{prf}

\begin{remark}[Convergence of mollified and truncated Lipschitz functions]
A simple corollary of \pref{lem:mollifiers} is that, if $f\in C^1(\mathbb R^n)$ is $A$-Lipschitz and $R>0$ satisfies $\mu\{|f(x)|=R\}=0$, then for
\[
    F=f_R:=\max(-R,\min(f,R)),
\]
one has
\[
    \eta_\varepsilon*f_R\to_{L^2(\mu)} f_R
\]
and
\[
    \nabla(\eta_\varepsilon*f_R)\to_{L^2(\mu;\R^n)} \mathbf 1_{\{|f|<R\}}\nabla f\,.
\]
\end{remark}